\documentclass[reqno]{amsart}

\usepackage{amsmath}
\usepackage{amssymb}
\usepackage{verbatim}
\usepackage{dsfont}
\usepackage{pifont}
\usepackage{enumerate}
\usepackage[all]{xy}
\usepackage{amsthm}
\usepackage{stmaryrd}
\usepackage{amscd}
\usepackage{amsfonts}
\usepackage{latexsym}
\usepackage{amscd}
\usepackage{graphicx}
\usepackage{tikz}
\usetikzlibrary{decorations.markings}
\usetikzlibrary{cd}
\usetikzlibrary{patterns}
\usepackage[vcentermath]{youngtab}
\usepackage{setspace}
\usepackage{wasysym}
\usepackage{bm}
\usepackage{bbm}                     
\usepackage[colorlinks=true, linkcolor=blue, citecolor=blue, urlcolor=blue, breaklinks=true]{hyperref}

\tikzset{anchorbase/.style={baseline={([yshift=-0.5ex]current bounding box.center)}}}
\tikzset{wipe/.style={white,line width=4pt}}

\makeatletter

\theoremstyle{plain}
\newtheorem*{theorem*}{Theorem}
\newtheorem*{remark*}{Remark}
\newtheorem*{example*}{Example}
\newtheorem{lemma}{Lemma}[section]
\newtheorem{proposition}[lemma]{Proposition}
\newtheorem{corollary}[lemma]{Corollary}
\newtheorem{theorem}[lemma]{Theorem}

\newtheorem*{conjecture*}{Conjecture}

\theoremstyle{definition}
\newtheorem{definition}[lemma]{Definition}

\theoremstyle{remark}
\newtheorem{remark}[lemma]{Remark}

\@addtoreset{equation}{section}
\def\theequation{\arabic{section}.\arabic{equation}}
\newcommand{\Tr}{\operatorname{Tr}}
\newcommand{\Dim}{\operatorname{Dim}}
\newcommand{\Hom}{\operatorname{Hom}}
\newcommand{\Z}{\mathbb{Z}}
\newcommand{\C}{\mathbb{C}}
\newcommand{\N}{\mathbb{N}}

\newcommand{\Ext}{\operatorname{Ext}}

\newcommand{\End}{\operatorname{End}}

\newcommand{\cA}{\mathcal{A}}

\newcommand{\cD}{\mathcal{D}}

\newcommand{\Vect}{\mathcal{V}ec}
\renewcommand{\k}{\Bbbk}

\def\quotient#1#2{%
    \raise1ex\hbox{$#1$}\Big/\lower1ex\hbox{$#2$}%
}

\def\T{\pi}
\def\Vdash{\vDash_s}
\def\nset#1{\{1,\dots,#1\}}
\def\Std{\operatorname{Std}}
\def\Nat{\operatorname{Nat}}
\def\Stab{\operatorname{Stab}}
\def\Theta{\mathrm{Mat}}
\def\Web{\mathcal{W}eb}
\def\Schur{\mathcal{S}chur}
\def\diag{\operatorname{diag}}
\def\Kar{\operatorname{Kar}}
\def\Tilt#1{\mathcal{T}\!ilt(#1)}
\def\TILT#1{\overline{\mathcal{T}\!ilt(#1)}}
\def\Rep#1{\mathcal{R}ep(#1)}
\def\eps{\varepsilon}
\newcommand{\arxiv}[1]{{\tt arXiv:#1}}

\def\OB{\mathcal{OB}}
\def\up{\uparrow}
\def\down{\downarrow}
\newcommand{\REP}{{\underline{\mathcal{R}e}p\,}}

\def\bi{\text{\boldmath$i$}}
\def\bj{\text{\boldmath$j$}}
\def\bk{\text{\boldmath$k$}}
\def\I{\mathrm I}

\begin{document}
 
 \title[Semisimplification of $\Tilt{GL_n}$]{\boldmath Semisimplification of the category of tilting modules for $GL_n$}
\author[J. Brundan]{Jonathan Brundan}
\address{Department of Mathematics,
University of Oregon, USA}
\email{brundan@uoregon.edu}

\author[I. Entova]{Inna Entova-Aizenbud}
\address{Department of Mathematics, Ben Gurion University, Israel}
\email{entova@bgu.ac.il}

\author[P. Etingof]{Pavel Etingof}
\address{Department of Mathematics,
Massachusetts Institute of Technology, USA}
\email{etingof@math.mit.edu}

\author[V. Ostrik]{Victor Ostrik}
\address{Department of Mathematics,
University of Oregon, USA \& Laboratory of Algebraic Geometry,
National Research University Higher School of Economics, Moscow, Russia}
\email{vostrik@uoregon.edu}

\thanks{
This material is based on work supported by
the National Science Foundation under Grant No.~DMS-1440140 while two of the authors (P.E. and V.O.) were in residence at the Mathematical Sciences Research Institute in Berkeley, California in Spring 2020. 
The work of J.B. was supported by NSF grant DMS-1700905.
The work of I.E. was supported by the ISF grant 711/18.
The work of P. E. was also partially supported by the NSF grant
DMS-1502244. The work of V. O. was also partially supported by the NSF grant DMS-1702251 and the Russian Academic Excellence Project 5-100.
}

\begin{abstract}
We describe the semisimplification of the monoidal category of tilting modules
for the algebraic group $GL_n$ in characteristic $p > 0$.
In particular, we compute the dimensions of
the indecomposable tilting modules modulo $p$.
\end{abstract}

 \maketitle

\section{Introduction}

Let $\k$ be an algebraically closed field of characteristic $p \geq
0$ and $G_n$ denote the algebraic group $GL_n(\k)$ for $n\geq 0$.
The symmetric tensor category
$\Rep{G_n}$ of finite-dimensional rational representations of
$G_n$ 
is a lower finite highest weight category
with 
irreducible, standard, costandard and indecomposable tilting modules
$L_n(\lambda)$, $\Delta_n(\lambda)$, $\nabla_n(\lambda)$ and $T_n(\lambda)$ 
parametrized by their highest weight $\lambda$.
In the usual coordinates, the dominant weight $\lambda$ appearing here
may be identified with an element of the poset
\begin{equation}
X_n^+ = \left\{\lambda = (\lambda_1,\dots,\lambda_n) \in \Z^n\:\big|\:\lambda_1
\geq \cdots \geq \lambda_n\right\}
\end{equation}
ordered by the usual dominance ordering $\unlhd$.
Let $\Tilt{G_n}$ be the full subcategory of $\Rep{G_n}$ consisting of
all tilting modules, which is a Karoubian
rigid symmetric monoidal category.
The defining $n$-dimensional representation $V_n$ of
$G_n$
is an indecomposable tilting module,
as are all of its (irreducible) exterior powers and their duals.
These modules generate $\Tilt{G_n}$ 
as a Karoubian monoidal category (i.e., taking tensor products,
direct sums and direct summands).

The
{\em
  semisimplification}
\begin{equation}
\TILT{G_n}
:=
\Tilt{G_n} / \mathcal N
\end{equation}
 of the category $\Tilt{G_n}$ is its quotient by the tensor 
ideal $\mathcal N$ consisting of all negligible morphisms.
This is a semisimple symmetric tensor category with irreducible objects 
arising from the indecomposable
tilting modules whose dimension is non-zero modulo $p$; see
\cite{EO} for further discussion and historical remarks.
Of course, if $p=0$ the category $\Rep{G_n}$ is already semisimple so
coincides with the semisimplification
$\TILT{G_n}$, and the irreducible objects in $\TILT{G_n}$
are labeled by the set
$X_{n,0}^+ := X_n^+$ of all dominant weights.
The case $p \geq n$ may also be regarded as classical: in this case, the
category
$\TILT{G_n}$ is the so-called {\em Verlinde category}, with
irreducible objects 
arising from the indecomposable tilting
modules of highest weight belonging to the set
\begin{equation}\label{lowest}
X_{n,p}^+ := \{\lambda =
(\lambda_1,\dots,\lambda_n) \in X_n^+\:|\:\lambda_1-\lambda_n <
p-n+1\},
\end{equation}
interpreting $X_{0,p}^+$ as $\{\varnothing\}$.
The classical proof of this from \cite{GK, GM} goes as follows.
As $X_{n,p}^+$ is the fundamental alcove, the linkage
principle implies that
$T_n(\lambda)=\Delta_n(\lambda)$ for $\lambda$ in
the upper closure $\overline{X}_{n,p}^+$ (defined by
replacing $<$ in (\ref{lowest}) by $\leq$).
By the Weyl dimension formula, it follows that $T_n(\lambda)$ is of
non-zero dimension modulo $p$ for $\lambda \in X_{n,p}^+$, and
its identity morphism is negligible for $\lambda \in \overline{X}_{n,p}^+ \setminus
X_{n,p}^+$.
Then an argument with translation functors gives that the identity
morphism of $T_n(\lambda)$
is negligible for any $\lambda \in X_n^+ \setminus X_{n,p}^+$, hence,
these modules are all of dimension zero modulo $p$.

In this article, we treat the remaining situations when $0 < p < n$.
Note that the case $p=2$ was worked out already in \cite[$\S$8]{EO}.
To formulate the main result in general, assume that $n, p > 0$ and let
\begin{equation}\label{padic}
n = n_0 + n_1 p + \cdots + n_r p^r
\end{equation}
be the $p$-adic decomposition of $n$, so $0 \leq
n_0,\dots,n_{r-1} < p$ and $0 < n_r < p$.
We define an embedding
\begin{equation}\label{interests}
\imath: X^+_{n_0}\times X^+_{n_1}\times\cdots\times X^+_{n_r} \hookrightarrow 
X^+_n
\end{equation}
sending $\underline{\lambda} =
\left(\lambda^{(0)},\dots,\lambda^{(r)}\right)$ to the dominant conjugate of
the $n$-tuple that is the concatenation
$\lambda^{(0)} \sqcup \underbrace{\lambda^{(1)}\sqcup\cdots\sqcup\lambda^{(1)}}_{\text{$p$
    copies}}\sqcup
\underbrace{\lambda^{(2)}\sqcup\cdots\sqcup\lambda^{(2)}}_{\text{$p^2$
    copies}}\sqcup
\cdots\sqcup
\underbrace{\lambda^{(r)}\sqcup\cdots\sqcup\lambda^{(r)}}_{\text{$p^r$
    copies}}.
$
Let
\begin{equation}\label{off}
X_{n,p}^+ := \imath\left(X_{n_0,p}^+\times\cdots\times X_{n_r,p}^+\right) \subset
X_n^+.
\end{equation}
See (\ref{shower})--(\ref{soap}) below
for a more conceptual description of this set.
Also let $\boxtimes$ be the Deligne tensor product of tensor categories 
(e.g., see \cite[$\S$4.6]{EGNO}). The Deligne tensor product of
semisimple symmetric tensor categories is
again a semisimple symmetric tensor category.

\vspace{2mm}
\noindent
{\bf Main Theorem.} {\em
For $p > 0$ as above, there is a symmetric monoidal equivalence
$$
\Xi_n: \TILT{G_{n_0}} \boxtimes\cdots\boxtimes
\TILT{G_{n_r}}
\rightarrow \TILT{G_n}
$$
sending 
$T_{n_0}(\lambda^{(0)}) \boxtimes\cdots \boxtimes
T_{n_r}(\lambda^{(r)})$ for
$\underline{\lambda} = (\lambda^{(0)},\dots,\lambda^{(r)})
 \in X_{n_0,p}^+\times\cdots \times X_{n_r,p}^+$
to $T_n(\imath(\underline{\lambda}))$.
In particular, the irreducible objects of $\TILT{G_n}$
are the indecomposable tilting modules with highest weight 
in $X_{n,p}^+$.
}
\vspace{2mm}

\noindent
{\bf Example.}
If $p=5$ and $n=13 = 3 + 2\cdot 5$,
this implies that $\TILT{G_{13}}$ is equivalent to
$\TILT{G_3} \boxtimes \TILT{G_2}$. 
The bijection 
$\imath:X_{3,5}^+\times X_{2,5}^+\rightarrow X_{13,5}^+$
 between the labeling sets
takes $\underline{\lambda}=(\lambda^{(0)}, \lambda^{(1)}) \in
X_3^+\times X_2^+$
with
$\lambda^{(0)}_1 - \lambda^{(0)}_3 < 3$
and $\lambda^{(1)}_1-\lambda^{(1)}_2 < 4$
to 
$$
\imath(\underline{\lambda}) =
\left(\lambda^{(0)}_1,\lambda^{(0)}_2,\lambda^{(0)}_3,
\lambda^{(1)}_1, 
\lambda^{(1)}_2, 
\lambda^{(1)}_1, 
\lambda^{(1)}_2, 
\lambda^{(1)}_1, 
\lambda^{(1)}_2, 
\lambda^{(1)}_1, 
\lambda^{(1)}_2, 
\lambda^{(1)}_1, 
\lambda^{(1)}_2
\right)^+ \in X_{13}^+
$$
where $+$ denotes dominant conjugate.
So
$\Xi_{13}(V_3 \boxtimes \k) \cong V_{13}$,
$\Xi_{13}(\k \boxtimes V_2) \cong \bigwedge^5 V_{13}$
and
$\Xi_{13}(V_3 \boxtimes V_2) \cong \bigwedge^6 V_{13}
\cong V_{13} \otimes \bigwedge^5 V_{13}$ (isomorphisms in $\TILT{G_{13}}$).

\vspace{2mm}
\noindent
{\bf Corollary.} {\em
If $\lambda \in X_n^+\setminus X_{n,p}^+$
then $\dim T_n(\lambda) \equiv 0 \pmod{p}$.
 If $\lambda \in X_{n,p}^+$, so that
$\lambda = \imath(\underline{\lambda})$
for 
$\underline{\lambda} = (\lambda^{(0)},\dots,\lambda^{(r)})
 \in X_{n_0,p}^+\times\cdots \times X_{n_r,p}^+$,
then
we have that
$$
\dim T_n(\lambda) \equiv \prod_{i=0}^r \dim \Delta_{n_i}(\lambda^{(i)})\pmod{p}.
$$
The right hand side here may be computed explicitly using the Weyl dimension formula.
}

\begin{proof}
For each $i=0,\dots,r$, we have that $p > n_i$, so by the
classical description of Verlinde categories we have that
$\dim T_{n_i}(\lambda^{(i)}) \equiv \dim \Delta_{n_i}(\lambda^{(i)})\pmod{p}$
for $\lambda^{(i)} \in X_{n_i}^+$.
Now the corollary
follows from the theorem
since symmetric monoidal functors are trace-preserving, hence, 
they also
respect categorical dimensions.
\end{proof}

The Main Theorem gives rise to a categorification of Lucas' theorem in the
following sense.
If $k = k_0+k_1p+\cdots+k_r p^{r}$ for $0 \leq k_0,\dots,k_r <
p$, then
$\bigwedge^k V_n \in \TILT{G_n}$ is the image of
the irreducible object $\bigwedge^{k_0} V_{n_0}\boxtimes\cdots\boxtimes\bigwedge^{k_r}
V_{n_r}\in\TILT{G_{n_0}} \boxtimes\cdots\boxtimes
\TILT{G_{n_r}}$
under the equivalence $\Xi_n$ from the theorem.
We deduce on taking categorical dimensions that
\begin{equation}\label{classicallucas}
\binom{n}{k}
\equiv \prod_{i=0}^r \binom{n_i}{k_i}\pmod{p},
\end{equation}
which is exactly the {\em classical Lucas theorem}.

An essential step in the proof is provided by a theorem of Donkin from
\cite{D1},
which gives a version of skew Howe duality for the general linear group.
In fact, we rephrase Donkin's result in terms of 
what we call the {\em Schur category}; see Theorem~\ref{dthm} for the
statement.
The Schur category is a strict monoidal category
closely related to the classical Schur algebra; see Definition~\ref{scat}.
It also has an explicit diagrammatic realization 
in terms of webs, which is due to Cautis, Kamnitzer and Morrison
\cite{CKM}. 
Since we are working in positive characteristic, we
have included a self-contained treatment establishing the connection
between the Schur category and webs via an approach which
is independent of \cite{CKM}; see Theorem~\ref{ckmrevised}.

The Main Theorem reduces the study of $\TILT{G_n}$ for all $p \geq 0$
to the classical cases in which $p=0$ or $p > n$.
In these classical cases, it can be helpful to think about the
combinatorial 
structure of $\TILT{G_n}$ from the perspective of categorification.
Let $\mathfrak{s}$ be the affine Kac-Moody algebra
$\mathfrak{sl}_\infty$ if $p=0$ or
$\widehat{\mathfrak{sl}}_p$ if $p > n$, with fundamental weights
$\Lambda_i$ and simple coroots $h_i$ for $i \in \Z / p \Z$.
There is a well-known categorical action making
$\Rep{G_n}$
into a 2-representation
of the Kac-Moody 2-category 
$\mathfrak{U}(\mathfrak{s})$.
(The quickest way to construct this is to apply \cite[Theorem 4.11]{BSW},
starting from the action of the degenerate 
Heisenberg category of central charge zero
under which $\uparrow$ acts by tensoring with $V_n$ and $\downarrow$
acts by tensoring with $V_n^*$, as is discussed in the introduction of \cite{BSW}.)
This categorical action restricts to give an action of
$\mathfrak{U}(\mathfrak{s})$ on $\Tilt{G_n}$
such that
\begin{equation}\label{iso1}
\C \otimes_{\Z} K_0(\Tilt{G_n})\cong {\textstyle\bigwedge^n \Nat_p}
\end{equation}
as an $\mathfrak{s}$-module,
where $\Nat_p$ is a natural level zero representation of
$\mathfrak{s}$ with basis $(m_i)_{i \in \Z}$
such that $m_i$ is of weight $\Lambda_{i-1}-\Lambda_i$; 
see the discussion in the introduction of \cite{HOMFLY}, or \cite[Proposition 6.5]{RW}.
In particular, $\C \otimes_{\Z} K_0(\Tilt{G_n})$ is generated as an
$\mathfrak{s}$-module by the class $[\k]$ of the trivial module, which
corresponds under (\ref{iso1}) to the vector $m_0 \wedge m_{-1} \wedge\cdots\wedge
m_{1-n}\in  {\textstyle\bigwedge^n \Nat_p}$ of
weight $\Lambda_{-n} - \Lambda_0$.
The ideal $\mathcal{N}$ of negligible morphisms defines
a sub-2-representation, hence, the quotient $\TILT{G_n}$ is a 2-representation as
well. Its complexified Grothendieck ring satisfies
\begin{equation}\label{iso2}
\C \otimes_{\Z} K_0(\TILT{G_n})\cong V(\Lambda_{-n}-\Lambda_0),
\end{equation}
i.e., it is the level zero extremal weight module parametrized by the
minuscule weight $\Lambda_{-n}-\Lambda_0$ in the sense of \cite{Kas}.
This follows because, as an $\mathfrak{s}$-module,
$\C \otimes_{\Z} K_0(\TILT{G_n})$
is generated by a vector of weight $\Lambda_{-n}-\Lambda_0$, and 
it is minuscule as all of its weights $\lambda$ satisfy $\langle
h_i, \lambda \rangle \in \{0,1,-1\}$ for all $i \in \Z / p\Z$. The latter assertion follows 
from the semisimplicity of
the category $\TILT{G_n}$ by invoking some of the general structure
theory of Kac-Moody 2-representations. In more detail, semisimplicity
implies that the representation-theoretic Kashiwara operators $\eps_i,
\phi_i$ as defined e.g. in 
\cite[$\S$5.1]{BSW} satisfy $\eps_i(L), \phi_i(L) \leq 1$ for all
irreducible objects $L \in \TILT{G_n}$ and all $i \in \Z / p\Z$.
Since the weight $\lambda$ of the class of $L$ in $\C \otimes_{\Z}
K_0(\TILT{G_n})$ satisfies $\langle h_i, \lambda \rangle = \phi_i(L) - \eps_i(L)$ by
\cite[Lemma 5.2]{BSW}, this 
implies that $\langle
h_i, \lambda \rangle \in \{0,1,-1\}$ for all $i$.

We remark finally that there is also a generalization of our Main
Theorem to the quantum general linear group $G_{n,q}$ for any
$q \in \k^\times$ such that $q^2$ is a primitive $\ell$th root of
unity. It is related to the {\em quantum Lucas theorem}.
The proof in the quantum case is quite similar,
using Donkin's skew Howe duality established in \cite{D2} formulated
in terms of the {\em $q$-Schur category}, 
which again can be viewed
diagrammatically in terms of the webs of \cite{CKM}.
This will be developed in a subsequent paper.

\vspace{2mm}
\noindent
{\em Conventions.}
 All categories will be $\k$-linear with finite-dimensional
$\Hom$-spaces, and all functors will be $\k$-linear.
A category is {\em Karoubian} if it is additive and idempotent
complete. 
Functors
between Karoubian categories are automatically additive due to the
assumption that they are $\k$-linear, 

\vspace{2mm}
\noindent
{\em Acknowledgements.} 
The first author would like to thank Travis Scrimshaw for suggesting
the connection to extremal weight crystals, and Ben Elias
for many helpful discussions about web categories.

\section{Background about semisimplification}\label{gensec}

In this section, we give a self-contained treatment of some basic facts
about semisimplification which will be needed later. 
The results here are all well known and first appeared in \cite{BW}
(see also \cite[$\S$6]{Del} and \cite{AK02}).
We work in the setting of symmetric monoidal categories for
simplicity, but the arguments are quite general. For further
discussion of the extension to
pivotal categories, see \cite[$\S$2.3]{EO}.

Following our general conventions, all monoidal categories
will be $\k$-linear, meaning in particular that the
tensor product functor 
$-\otimes-$ is bilinear, with finite-dimensional $\Hom$-spaces. 
A {\em tensor category} means a 
monoidal
category which is rigid and Abelian, with all objects having
finite length, and satisfying $\End(\mathbbm{1}) = \k$.
Note that in such a category the functor $-\otimes-$ is biexact.
See \cite[Ch. 4]{EGNO} for a detailed treatment. 

Let $\cD$ be a rigid symmetric monoidal category with
$\End_{\cD}(\mathbbm{1}) = \k$.
By the {\em trace} $\Tr(f)$ of a morphism $f:X \rightarrow X$, we mean
the scalar in $\k$ 
defined by the composition 
$$
\mathbbm{1} \stackrel{\operatorname{coev}_X}{\longrightarrow} X \otimes X^*
\stackrel{f \otimes \operatorname{id}_{X^*}}{\longrightarrow} X \otimes
X^* 
\stackrel{s_{X,X^*}}{\longrightarrow} X^* \otimes X \stackrel{\operatorname{ev}_X}{\longrightarrow} \mathbbm{1},
$$
where $\operatorname{coev}_X$ and $\operatorname{ev}_X$ are the
evaluation
and coevaluation morphisms for the dual $X^*$ of $X$, and $s_{X,Y}:X \otimes Y
\stackrel{\sim}{\rightarrow} Y \otimes X$ is the symmetric braiding.
Then the {\em categorical dimension} $\Dim X$ means
$\Tr(\operatorname{id}_X)$.
Note that symmetric monoidal functors between categories of this sort
preserve trace, hence also
categorical dimensions.
The category $\cD = \Tilt{G_n}$ considered later in the paper
admits a symmetric monoidal 
functor to vector spaces (``fiber functor''),
so for $V \in \Tilt{G_n}$ the 
categorical dimension $\Dim V$ coincides with the image in $\k$ of the usual
dimension $\dim V$ of the underlying vector space.

A category $\cA$ is {\it semisimple} if it is Abelian and
every object is isomorphic to a finite direct sum of irreducible objects.
In a semisimple category, every short exact sequence splits.
The following lemma is taken from \cite[Section 2.1]{Mueger}.

\begin{lemma}\label{Mue}
Let $\cA$ be a $\k$-linear category with finite-dimensional
$\Hom$-spaces.
Then $\cA$ is semisimple if and only if it is Karoubian,
there exists a family $(L_i)_{i \in I}$ of objects such that
$\dim \Hom_{\cA}(L_i, L_j) = \delta_{i,j}$ for all $i,j \in I$, and moreover
any object of $\cA$ is isomorphic to a finite direct sum of objects $L_i\:(i \in I)$.
\end{lemma}

\begin{remark}
 The last condition in Lemma~\ref{Mue} may be replaced by the following: for all $U, V
 \in \cA$ the map 
 $$\bigoplus_{i\in I} \Hom_{\cA}(U, L_i) \otimes_{\k} \Hom_{\cA}( L_i, V) \longrightarrow \Hom_{\cA}(U, V)$$ given by composition is an isomorphism.
\end{remark}

\begin{definition}\label{maindef}
Let $\cD$ be a Karoubian rigid symmetric monoidal category satisfying $\End_{\cD}(\mathbbm{1})=\k$.
For any $X, Y \in \cD$, we let $$\mathcal{N}(X,Y) := \left\{ f: X \to
Y \:\big|\: 
\Tr(g\circ f) = 0\text{ for all }g: Y\to X 
\right\}$$
and denote by $\mathcal{N}$ the corresponding collection of
$\mathcal{N}(X,Y)$ over all $X, Y \in \cD$. Then $\mathcal{N}$ is a
tensor ideal (see e.g. \cite[Lemma 2.3]{EO}), called the {\it tensor ideal of negligible morphisms} in
$\cD$.
We define the {\em semisimplification} of $\cD$ to be the quotient category
\begin{equation*}
\overline{\cD} = {\cD}/{\mathcal{N}},
\end{equation*}
letting $Q:\cD \to \overline{\cD}$ be the canonical quotient functor. 
In particular, this means that the object set of $\overline{\cD}$ is the same as for $\cD$, i.e., $QX
= X$ for all $X \in \cD$, although of course non-isomorphic objects of $\cD$
may be isomorphic in $\overline{\cD}$.
\end{definition}

The category $\overline{\cD}$ in Definition~\ref{maindef}
is again a Karoubian rigid symmetric
monoidal category with $\End_{\overline{\cD}}(\mathbbm{1})=\k$
(see e.g. \cite[$\S$6]{Del}).
Also the quotient functor $Q$ is a full symmetric monoidal functor.

\begin{lemma}\label{didnt}
Let $\cD$ be as in Definition~\ref{maindef}, and assume moreover that
all nilpotent endomorphisms in $\cD$ have trace zero.
Let $X \in \cD$ be an indecomposable object
with endomorphism algebra $E := \End_{\cD}(X)$, and $J := J(E)$ be the Jacobson radical.
\begin{enumerate}
\item If $\Dim X \neq 0$ then 
$\mathcal{N}(X,X) = J$, hence, $\dim \End_{\overline{\cD}}(X) = 1$.\label{i1}
\item If $\Dim X = 0$ then 
$\mathcal{N}(X,X) = E$, hence, $\dim \End_{\overline{\cD}}(X) = 0$.\label{i2}
\item Given another indecomposable object $Y \not\cong X$,
all morphisms $X\rightarrow Y$ are negligible,
hence,
$\dim \Hom_{\overline{\cD}}(X,Y) = 0$.\label{i3}
\end{enumerate}
\end{lemma}

\begin{proof}
Since $E$ is finite-dimensional and local over an algebraically closed
field, its Jacobson radical is of
codimension one.
The assumption on $\cD$ implies that all elements of $J$ are of
trace
zero. Since $J$ is an ideal, we deduce that $J \leq
\mathcal{N}(X,X) \leq E$.

\vspace{1mm}
\noindent (\ref{i1})
As $\Dim X \neq 0$, the identity endomorphism $1_E$ of $X$ is
not negligible. Hence, 
$\mathcal{N}(X,X) \neq E$, so we must have that $\mathcal{N}(X,X) =
J$.

\vspace{1mm}
\noindent(\ref{i2})
We must show that $\Tr(f) = 0$ for all $f \in E$.
To see this, write $f$ as $\lambda 1_E + h$
for $\lambda \in \k$ and $h \in J$.
Then $\Tr(f) = 
\Tr(\lambda 1_E + h) = \lambda \Dim X + \Tr(h) = 0$.

\vspace{1mm}
\noindent(\ref{i3})
We must show 
that $\Tr(g \circ f)  = 0$
for any morphisms $f:X \rightarrow Y$ and $g:Y \rightarrow X$.
Note that $g \circ f$ is not an isomorphism, since otherwise
$f$ would be a split embedding of $X$ into $Y$
with left inverse $(g \circ f)^{-1} \circ g$, contradicting the
assumption that $X$ and $Y$ are indecomposable 
with $X \not\cong Y$.
Hence, $g \circ f \in J$, which we have already observed is contained
in $\mathcal N(X,X)$.
\end{proof}

\begin{theorem}\label{jonsversion}
For $\cD$ as in Definition~\ref{maindef},
the following conditions are equivalent:
 \begin{enumerate}\label{itms:cond_ss}
  \item\label{itm1:cond_ss} $\overline{\cD}$ is a semisimple symmetric
    tensor category;
 \item\label{itm3:cond_ss} there exists a symmetric monoidal functor
   from $\cD$ to a symmetric tensor category;
 \item\label{itm2:cond_ss} all nilpotent endomorphisms in $\cD$ have
   trace zero.
\end{enumerate}
When these conditions hold, 
the irreducible objects in $\overline{\cD}$ are the
indecomposable objects of $\cD$ of non-zero dimension, two such
objects being isomorphic in $\overline{\cD}$ if and only if they are
isomorphic in $\cD$.
\end{theorem}

\begin{proof}
The implication  {\bf \eqref{itm1:cond_ss} $\Rightarrow$
  \eqref{itm3:cond_ss}} follows because $Q:\cD \rightarrow
\overline{\cD}$ is such a functor.
The implication
 {\bf \eqref{itm3:cond_ss} $\Rightarrow$ \eqref{itm2:cond_ss}} follows
 from the fact that in a tensor category, any nilpotent endomorphism
 has trace zero (see \cite[$\S$6]{Del}).
For the remainder of the proof, we assume \eqref{itm2:cond_ss} and must
prove \eqref{itm1:cond_ss} together with the final assertion. 
 
The category $\cD$ is Krull-Schmidt. In
particular, any object is a finite direct sum of indecomposable objects.
This follows from the finite-dimensionality of the endomorphism
algebras $\End_{\cD}(X)$ for all $X \in \cD$.
In view of Lemma~\ref{didnt}(2), indecomposable objects of $\cD$ with
categorical dimension zero become zero objects in $\overline{\cD}$.
Thus, if we let $(L_i)_{i \in I}$ be a system of representatives for 
the isomorphism classes of indecomposable objects of non-zero
categorical dimension in $\cD$, we deduce that every object of
$\overline{\cD}$
is isomorphic to a finite direct sum of $L_i\:(i \in I)$.
The other parts of Lemma~\ref{didnt} check the remaining hypothesis 
$\dim \Hom_{\overline{\cD}}(L_i, L_j) = \delta_{i,j}$ of Lemma~\ref{Mue},
thereby showing that $\overline{\cD}$ is semisimple.
The final assertion follows by Lemma~\ref{didnt} again.
\end{proof}

Finally, we record the following, which makes the universal property
of the semisimplification $\overline{\cD}$ explicit.

\begin{lemma}\label{jonscor}
Suppose that $\cD$ satisfies the conditions of Theorem~\ref{jonsversion}.
Let $F:\cD\to\cA$ be a full symmetric monoidal 
functor to a semisimple symmetric tensor
category $\cA$.
Then there is a unique fully faithful symmetric monoidal functor
$
U: \overline{\cD}  \longrightarrow
\cA
$
such that $F = U \circ Q$.
\end{lemma}

\begin{proof}
Let $\mathcal I$ be the kernel of $F$, that is, the collection of all morphisms $f$ in $\cD$ which
are annihilated by the functor $F:\cD \to \cA$. 
Given $f:X \to Y$ in $\mathcal I$, we have that
$\Tr(g\circ f) = \Tr(F(g) \circ F(f)) = \Tr(0) = 0$
for all $g:Y \rightarrow X$. Hence, $\mathcal I \subseteq \mathcal N$.
As the functor $F$ is full, the image under $F$ of 
any $f \in \mathcal N$ is negligible in $\cA$ as well. 
On the other hand, $\cA$ is semisimple, so it has no non-zero negligible
 morphisms (see \cite[$\S$6]{Del}). 
Hence, $\mathcal I = \mathcal N$.

Now to prove the lemma,
note that the objects of $\overline{\cD}$ are the same as the
objects of $\cD$, so we must take $UX := FX$ for $X \in \cD$. Then
on a morphism $\bar f \in \Hom_{\overline{\cD}}(X,Y)$, we must take $U(\bar
f) := F(f)$ where $f \in \Hom_{\cD}(X,Y)$ is any lift chosen so that $Q(f) =
\bar f$. By the previous paragraph, this is well-defined and faithful.
\end{proof}

\section{Construction of the equivalence}

Given a parameter $t\in \k$, the {\em oriented Brauer
category} $\OB(t)$ is the free rigid symmetric monoidal
category generated by an object of categorical dimension $t$.
It can be realized explicitly using
the usual string calculus for strict monoidal categories, as follows.
The objects of $\OB(t)$ are words 
in the symbols
$\uparrow$ (the generating object) and $\downarrow$ (its dual).
For two such words $X=X_1\cdots X_r$ and $Y=Y_1\cdots Y_s$, 
an $X \times Y$ {\em oriented Brauer diagram} is
a diagrammatic representation of a 
bijection
$$
\{i\:|\:X_i = \up\}\sqcup\{j\:|\:Y_j = \down\}
\stackrel{\sim}{\rightarrow}
\{i\:|\:X_i = \down\}\sqcup\{j\:|\:Y_j = \up\}
$$
obtained by placing vertices labeled in order from left to right
according to the letters of the word $X$ (resp., $Y$) on the top
(resp., bottom) boundary, then connecting these vertices with
strings
as prescribed by the given bijection.
For example, the following is a
$\down \down\up\up\times \down \up\up\down$
oriented Brauer diagram:
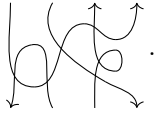
\begin{equation*}
\mathord{\begin{tikzpicture}[baseline = 0,scale=0.7]
\draw[-] (0,1) to [out=-90,in=180] (.45,-0.6);
\draw[-] (.45,-.6) to [out=0,in=180] (1.4,0.6);
\draw[->] (2,.3) to [out=0,in=-90] (2.4,1);
\draw[<-] (0,-1) to [out=70,in=190] (.4,.2);
\draw[-] (.4,.2) to [out=10,in=120] (0.8,-1);
\draw[-] (1.4,.6) to [out=0,in=180] (2,0.3);
\draw[-] (0.8,1) to [out=-120,in=150] (1.8,-0.5);
\draw[->] (1.8,-0.5) to [out=-30,in=90] (2.4,-1);
\draw[-] (2,0.1) to [out=0,in=0] (2,-.3);
\draw[->] (2,-.3) to [out=180,in=-90] (1.6,1);
\draw[-] (1.6,-1) to [out=90,in=-180] (2,0.1);
\end{tikzpicture}}
\,.
\end{equation*}
Two $X \times Y$ oriented Brauer diagrams are {\em equivalent} if they
represent the same bijection.
The morphism space $\Hom_{\OB(t)}(Y,X)$ 
is the vector space with basis given by the equivalence classes $[f]$ of 
$X \times Y$ oriented Brauer diagrams. The tensor product $[f] \otimes
[g]$ of two morphisms is the equivalence class defined by
the horizontal concatenation of the diagrams $f$ and $g$. The composition $[f]
\circ [g]$ is 
obtained by vertically stacking the diagram $f$ on top of $g$ then
removing closed bubbles in the interior of the diagram,
multiplying by $t$ each time a bubble is removed.
Alternatively, the category $\OB(t)$ can be defined rather concisely by
generators and relations; see
\cite{BCNR}.

Let $\Kar(\OB(t))$ be the Karoubi envelope of $\OB(t)$, that is, the idempotent
completion of its additive envelope.
When $\k$ is of characteristic zero, this category is better known as
the {\em Deligne category} $\REP(GL_t)$, but since we are most
interested in the positive characteristic case we will avoid this
terminology\footnote{The appropriate analog of the Deligne category in
positive characteristic is bigger than $\Kar(\OB(t))$.}.
The category $\Kar(\OB(t))$ is relevant to the problem in hand since, 
taking $t$ to be the image of $n \in \N$ in the field
$\k$, 
there is a symmetric monoidal functor
\begin{equation}\label{psin}
\Psi_n:\Kar(\OB(t)) \rightarrow \Tilt{G_n}
\end{equation}
sending $\uparrow$ to the natural $G_n$-module $V_n$ and $\downarrow$
to the dual module $V_n^*$.
By a version of Schur-Weyl duality, 
this functor is {\em full}, and it is {\em dense} if either $p=0$
or $p > n$; e.g., see \cite{HOMFLY}.

\begin{remark}\label{upgrade}
When $p=0$ or $p > n$ (and $t$ is the image of $n$ in $\k$ still), 
the functor $\Psi_n$ induces an equivalence of
symmetric monoidal categories between 
$\Kar(\OB(t) / \mathcal I_n)$ and $\Tilt{G_n}$,
where $\mathcal I_n$ is the tensor ideal
of $\OB(t)$ generated by the endomorphism of
$\up^{\otimes (n+1)}$ associated to the quasi-idempotent
$\sum_{g \in S_{n+1}} (-1)^{\ell(g)} g$ in the group algebra $\k S_{n+1}$ of
the symmetric group.
This is explained in detail in \cite{HOMFLY}. 
This article also constructs a categorical action of the Kac-Moody
2-category $\mathfrak{U}(\mathfrak s)$ on $\Kar(\OB(t))$ in the same
spirit as (\ref{iso1})--(\ref{iso2}), showing 
that
\begin{equation}\label{iso3}
\C\otimes_\Z K_0(\Kar(\OB(t))) \cong V(-\Lambda_0) \otimes
V(\Lambda_{-n})
\end{equation}
as an $\mathfrak{s}$-module, i.e., it is the tensor product of the
integrable lowest weight module of lowest weight $-\Lambda_0$ and the
integrable highest weight module of highest weight $\Lambda_{-n}$.
\end{remark}

\begin{lemma}\label{rainy}
Assume that $t \in \k$ is the image of $n \in \N$.
Then the semisimplifications $\overline{\Kar(\OB(t))}$ and
$\TILT{G_n}$ are semisimple symmetric tensor categories. Moreover,
if $p=0$ or $p > n$,
the functor $\Psi_n$ induces an equivalence of symmetric monoidal
categories
$$
\overline{\Psi}_n:\overline{\Kar(\OB(t))}\rightarrow\TILT{G_n}.
$$
\end{lemma}

\begin{proof}
Since $\Tilt{G_n}$ embeds into the tensor category $\Rep{G_n}$, we get
that
$\TILT{G_n}$ is a semisimple symmetric tensor category by
  Theorem~\ref{jonsversion}.
Similarly, we get that
$\overline{\Kar(\OB(t))}$ is a semisimple symmetric tensor category 
by considering the composition of the symmetric monoidal functor
(\ref{psin}) with the inclusion of $\Tilt{G_n}$ into $\Rep{G_n}$.
If $p = 0$ or $p > n$ then $\Psi_n$ is full and dense, hence, so too is 
$$
\widetilde{\Psi}_n := Q \circ \Psi_n:\Kar(\OB(t)) \rightarrow \TILT{G_n}.
$$
Applying Lemma~\ref{jonscor},
this descends
to give the symmetric monoidal equivalence
$\overline{\Psi}_n$.
\end{proof}

When $0 < p \leq n$, the functor $\Psi_n$ is no longer
dense. To rectify this, we need to work more generally with the
{\em colored oriented Brauer category}
$\OB(t_0,\dots,t_r)$, that is, the free rigid
symmetric 
monoidal category generated by $(r+1)$ objects $\up_0,\dots,\up_r$ of
dimensions 
$t_0,\dots,t_r\in\k$, respectively.
The definition of this is similar to $\OB(t)$, except that now
strings are labeled by an additional color from the set
$\{0,\dots,r\}$.
Thus, $\OB(t_0,\dots,t_r)$ has generating objects $\{\uparrow_i,
\downarrow_i\:|\:i=0,\dots,r\}$, and morphisms are $\k$-linear
combinations of equivalence classes of {\em colored oriented Brauer diagrams}.
Horizontal and vertical composition are as before;
in the latter case,
one multiplies by the parameter $t_i$ each time a closed bubble of
color $i$ is removed.

\begin{lemma}\label{notsnowy}
Suppose that $t_0,\dots,t_r\in\k$ are the images of $n_0,\dots,n_r
\in \N$. Then the semisimplification
$\overline{\Kar(\OB(t_0,\dots,t_r))}$ is a semisimple symmetric tensor category.
Moreover, assuming either $p=0$ or $p > \max(n_0,\dots,n_r)$,
there is an equivalence of symmetric monoidal categories
$$
\overline{\Psi}_{n_0,\dots,n_r}:
\overline{\Kar(\OB(t_0,\dots,t_r))}
\rightarrow 
\TILT{G_{n_0}} \boxtimes\cdots\boxtimes
\TILT{G_{n_r}}.
$$
sending $\up_i$ to $V_{n_i}$, the natural $G_{n_i}$-module, and
$\down_i$ to $V_{n_i}^*$.
\end{lemma}

\begin{proof}
By universal properties,
there is a symmetric monoidal functor
\begin{equation*}
\widetilde{\Psi}_{n_0,\dots,n_r}:\Kar(\OB(t_0,\dots,t_r))
\rightarrow 
\TILT{G_{n_0}} \boxtimes\cdots\boxtimes
\TILT{G_{n_r}}
\end{equation*}
sending $\up_i$ to $V_{n_i}$ and $\down_i$ to $V_{n_i}^*$.
If $p=0$ or $p > \max(n_0,\dots,n_r)$,
the symmetric monoidal functors $\Psi_{n_i}:\Kar(\OB(t_i))\rightarrow \Tilt{G_{n_i}}$ defined as in (\ref{psin}) are all full and
dense,
hence,
$\widetilde{\Psi}_{n_0,\dots,n_r}$ is full and dense too.
Since $\TILT{G_{n_0}} \boxtimes\cdots\boxtimes
\TILT{G_{n_r}}$ is a semisimple symmetric tensor category,
Theorem~\ref{jonsversion} implies
that $\overline{\Kar(\OB(t_0,\dots,t_r))}$ is a semisimple symmetric tensor
category.
Finally, Lemma~\ref{jonscor}
gives that $\widetilde{\Psi}_{n_0,\dots,n_r}$ descends to the desired equivalence
$\overline{\Psi}_{n_0,\dots,n_r}$.
\end{proof}

Now we can explain the strategy for the construction of the
equivalence $\Xi_n$ in the Main Theorem.
Assume that $p > 0$ and fix a $p$-adic decomposition of $n$ as in
(\ref{padic}).
Let $t_i \in \k$ be the image of $n_i$.
By a special case of (\ref{classicallucas}), we have that
$$
\dim
{\textstyle\bigwedge^{p^i}} V_n
= \binom{n}{p^{i}} \equiv n_i\pmod{p}.
$$
Hence, 
there is a symmetric monoidal functor
\begin{equation}\label{xmas}
\Phi_n
:\Kar(\OB(t_0,\dots,t_r)) \rightarrow \Tilt{G_n}
\end{equation}
sending $\up_i$ to $\bigwedge^{p^{i}} V_n$ and $\down_i$ to
$\bigwedge^{p^{i}} V_n^*$.

\begin{lemma}\label{density}
In the setup of (\ref{padic}),
the category $\TILT{G_n}$ is generated as a Karoubian monoidal
category by the exterior powers $\bigwedge^{p^{i}} V_n$ of the natural
$G_n$-module $V_n$ and their duals for $i=0,\dots,r$.
\end{lemma}

\begin{proof}
By highest weight considerations, 
the Karoubian monoidal category $\Tilt{G_n}$ is generated by the
exterior powers $\bigwedge^k V_n$ and their duals for $k=1,\dots,n$. 
By Lucas' theorem (\ref{classicallucas}),
$\dim \bigwedge^k V_n \equiv 0\pmod{p}$, hence, $\bigwedge^k V_n$ is
zero in
$\TILT{G_n}$, unless $k = k_0+k_1p+\cdots+k_r p^{r}$
for $0 \leq k_0 \leq n_0,\dots,0 \leq k_r \leq n_r$.
Therefore, $\TILT{G_n}$ is generated by the exterior powers
$\bigwedge^k V_n$ and their duals for $k$ of this special form.
To complete the proof, we show\footnote{The argument here corrects a mistake in the published version of the article which was pointed out to us by Elijah Bodish.
A related correction has been made to the proof of Lemma~\ref{starters}.} for any such $k$ that
$\bigwedge^k V_n$ is a summand of the tilting
module
\begin{equation*}
\textstyle
T := \left(V_n\right)^{\otimes k_0} 
\otimes \left(\bigwedge^p V_n\right)^{\otimes k_1}
\otimes\cdots\otimes
\left(\bigwedge^{p^{r}} V_n\right)^{\otimes k_r}.
\end{equation*}
First, we claim that $\bigwedge^{k_i p^i} V_n$
is a summand of the tilting module
$\left(\bigwedge^{p^i}
  V_n\right)^{\otimes k_i}$.
  Using the diagrammatic notation for homomorphisms introduced in Theorem~\ref{dthm} in the next section, consider
  $$
  e_i :=\frac{1}{k_i !} \Sigma_n\left(\;
  \begin{tikzpicture}[baseline = 8]
	\draw[-,thick] (0.35,-.1) to [out=90,in=-120] (0.09,0.34);
	\node at (0.17,0) {$\scriptstyle\cdots$};
	\draw[-,thick] (-0.05,-.1) to [out=90,in=-100](.095,0.26);
	\draw[-,thick] (-0.2,-.1) to [in=-60,out=90](0.07,0.34);
	\draw[-,line width=2pt] (0.08,.4) to (0.08,0.2);
	\draw[-,thick] (0.35,.8) to [out=-90,in=120] (0.09,0.36);
	\node at (0.17,.7) {$\scriptstyle\cdots$};
	\draw[-,thick] (-0.05,.8) to [out=-90,in=100](.095,0.44);
	\draw[-,thick] (-0.2,.8) to [in=60,out=-90](0.07,0.36);
	\draw[-,line width=2pt] (0.08,.2) to (0.08,.5);
\end{tikzpicture}\;\right) \in \End_{G_n}\left(\left({\textstyle\bigwedge^{p^i} V_n}\right)^{\otimes {k_i}}\right),
  $$
  where there are $k_i$ strings at the top and bottom, each of thickness $p^i$.
  The division by $k_i!$ makes sense because $k_i < p$.
  This $G_n$-module homomorphism is the composition of the natural projection
 $\Sigma_n\Big(\;\begin{tikzpicture}[baseline = -2,scale=.9]
	\draw[-,thick] (0.35,-.3) to [out=90,in=-120] (0.09,0.14);
	\node at (0.17,-.2) {$\scriptstyle\cdots$};
	\draw[-,thick] (-0.05,-.3) to [out=90,in=-100](.095,0.06);
	\draw[-,thick] (-0.2,-.3) to [in=-60,out=90](0.07,0.14);
	\draw[-,line width=2pt] (0.08,.3) to (0.08,0);
\end{tikzpicture}\;\Big):\left(\bigwedge^{p^i}
  V_n\right)^{\otimes k_i}\rightarrow \bigwedge^{k_i p^i} V_n$ followed by a homomorphism
  $\frac{1}{k_i!}\Sigma_n\Big(\;\begin{tikzpicture}[baseline = -2,scale=.9]
	\draw[-,thick] (0.35,.3) to [out=-90,in=120] (0.09,-0.14);
	\node at (0.17,.2) {$\scriptstyle\cdots$};
	\draw[-,thick] (-0.05,.3) to [out=-90,in=100](.095,-0.06);
	\draw[-,thick] (-0.2,.3) to [in=60,out=-90](0.07,-0.14);
	\draw[-,line width=2pt] (0.08,-.3) to (0.08,0);
\end{tikzpicture}\;\Big):
\bigwedge^{k_i p^i} V_n\rightarrow \left(\bigwedge^{p^i}
  V_n\right)^{\otimes k_i}$. It is an idempotent because we have that 
  $$
  \begin{tikzpicture}[baseline = 5]
\draw[-,thick] (0.35,.3) to [out=-90,in=120] (0.09,-0.14);
	\node at (0.17,.3) {$\scriptstyle\cdots$};
	\draw[-,thick] (-0.05,.3) to [out=-90,in=100](.095,-0.06);
	\draw[-,thick] (-0.2,.3) to [in=60,out=-90](0.07,-0.14);
	\draw[-,line width=2pt] (0.08,-.2) to (0.08,0);
\draw[-,thick] (0.35,.3) to [out=90,in=-120] (0.09,0.74);
	\draw[-,thick] (-0.05,.3) to [out=90,in=-100](.095,0.66);
	\draw[-,thick] (-0.2,.3) to [in=-60,out=90](0.07,0.74);
	\draw[-,line width=2pt] (0.08,.8) to (0.08,0.6);
\end{tikzpicture}
=
\frac{(k_i p^i)!}{(p^i)!\cdots (p^i)!}\:
\begin{tikzpicture}[baseline=5]
\draw[-,line width=2pt] (0.08,-.2) to (0.08,0.8);
\end{tikzpicture}\:
=
k_i!\:
\begin{tikzpicture}[baseline=5]
\draw[-,line width=2pt] (0.08,-.2) to (0.08,0.8);
\end{tikzpicture}\:,
$$
as may be checked using the relations 
  (\ref{assrel})--(\ref{trivial}) for the first equality and Lucas' theorem for the second. This proves the claim.
It follows that $\bigwedge^{k_0} V_n\otimes
\bigwedge^{k_1 p} V_n
\otimes\cdots\otimes \bigwedge^{k_r p^r} V_n$ is a summand of  $T$.
Now let $f:\bigwedge^k V_n\otimes
\bigwedge^{k_1 p} V_n \hookrightarrow \bigwedge^{k_0} V_n
\otimes\cdots\otimes \bigwedge^{k_r p^r} V_n$ be the canonical inclusion
and $g:\bigwedge^{k_0} V_n\otimes
\bigwedge^{k_1 p} V_n
\otimes\cdots\otimes \bigwedge^{k_r p^r} V_n\twoheadrightarrow \bigwedge^k V_n$
be the canonical projection.
Over any field,
the composition $g \circ f$ is $k! / k_0! (k_1 p)! \cdots (k_r
p^{r})!$ times the identity endomorphism.
Since we are in characteristic $p$, this scalar is $1$ by Lucas' theorem.
This shows that $f$ is a split injection, so $\bigwedge^k V_n$ is a
summand of
$\bigwedge^{k_0} V_n\otimes
\bigwedge^{k_1 p} V_n
\otimes\cdots\otimes \bigwedge^{k_r p^r} V_n$, hence, of $T$.
\end{proof}

Unlike the functor $\Psi_n$ considered in (\ref{psin}), the functor
$\Phi_n$ is neither full nor dense. Nevertheless, Lemma~\ref{density}
implies that \begin{equation}\label{haveit}
\widetilde{\Phi}_n := Q \circ \Phi_n:\Kar(\OB(t_0,\dots,t_r)) \rightarrow
\TILT{G_n}
\end{equation}
is dense.
Moreover, and this is the key step in our argument, $\widetilde{\Phi}_n$ 
is also full. This assertion will be justified in $\S$\ref{fullness};
see Theorem~\ref{fullity} (the proof is rather short but there are
lots of preliminaries!).
Given this fact, we can then apply Lemma~\ref{jonscor} to see that $\widetilde{\Phi}_n$ descends to a symmetric monoidal equivalence
\begin{equation}\label{boxing}
\overline{\Phi}_n:\overline{\Kar(\OB(t_0,\dots,t_r))} \rightarrow \TILT{G_n}.
\end{equation}
The equivalence $\Xi_n$ appearing in the Main Theorem may then be
obtained by composing $\overline{\Phi}_n$ with a quasi-inverse of the
equivalence 
$\overline{\Psi}_{n_0,\dots,n_r}$ from Lemma~\ref{notsnowy}.
To complete the proof of the Main Theorem, it just remains to
identify the labelings of the irreducible objects; this will be
explained in $\S$\ref{comb}.

\section{Webs and the Schur category}\label{fullness}

In this section, we show that the functor $\widetilde{\Phi}_n$ from
(\ref{haveit}) is full.
The proof depends ultimately on a result of Donkin \cite[Proposition 3.11]{D1},
which is a version of skew Howe duality for the general linear group. 
We will explain this using a diagrammatic rather than algebraic formalism, viewing the
Schur algebra in terms of a version of the web category from \cite{CKM}.
However, we start from the classical perspective as in \cite{Green}.

A {\em composition} $\lambda\vDash d$ is a finite sequence $\lambda =
(\lambda_1,\dots,\lambda_n)$
of non-negative integers summing to $d$.
We call it a {\em strict composition} and instead write $\lambda
\Vdash d$ if all of its parts are non-zero.
We write $\ell(\lambda)$ for the total number $n$ of parts.
There is a right action of $S_d$ on the set of $d$-tuples of positive 
integers by place permutation: for 
$\bi =
(i_1,\dots,i_d)$
and
$g \in S_d$ the $d$-tuple $\bi \cdot g$ has $r$th entry $i_{g(r)}$.
For $\lambda\vDash d$, the set
\begin{equation}
\I_\lambda := \left\{\bi = (i_1,\dots,i_d) \:\big|\:
\#\!\left\{r=1,\dots,d\:|\:i_r = i\right\} = \lambda_i
\text{ for all }i \in \nset{\ell(\lambda)}\right\}
\end{equation}
of all $d$-tuples with $\lambda_1$ entries
equal to 1, $\lambda_2$ entries equal to 2, and so on,
is a single orbit under this action.

For $\lambda,\mu\vDash d$, the symmetric group $S_d$ acts diagonally
on the right on 
$\I_\lambda \times \I_\mu$. The orbits are parametrized by the set
$\Theta_{\lambda,\mu}$ of all $\ell(\lambda)\times\ell(\mu)$ matrices
with non-negative integer entries such that
the entries in the $i$th row sum to
$\lambda_i$
and the entries in the $j$th column  sum to $\mu_j$ for all
$i\in \nset{\ell(\lambda)}$ and $j\in\nset{\ell(\mu)}$.
For $A = (a_{i,j}) \in \Theta_{\lambda,\mu}$, the corresponding $S_d$-orbit
on $\I_\lambda\times\I_\mu$ is
\begin{equation}
\Pi_A := \left\{(\bi,\bj) \in \I_\lambda\times\I_\mu \:\Bigg|\:
\begin{array}{l}
\#\!\left\{r=1,\dots,d\:|\:(i_r,j_r) = (i,j)\right\} = a_{i,j}\\
\text{for all  }
i \in \nset{\ell(\lambda)},   j\in \nset{\ell(\mu)}\end{array}
\right\}\!.
\end{equation}

For compositions $\lambda,\mu,\nu\vDash d$,
$A \in \Theta_{\lambda,\mu}, B \in \Theta_{\mu,\nu}$ and $C \in
\Theta_{\lambda,\nu}$, define
\begin{equation}\label{Z}
Z(A,B,C) := \#\!\left\{\bj \:\big|\:(\bi,\bj) \in \Pi_A\text{ and } (\bj,\bk) \in \Pi_B
\right\},
\end{equation}
where $(\bi,\bk)$ is some choice of an element of $\Pi_C$.
This is well-defined independent of the choice of
$(\bi,\bk)$.

\begin{lemma}\label{tricky}
In the notation of (\ref{Z}),
suppose that $(\bi,\bj) \in \Pi_A$ and $(\bj,\bk) \in \Pi_B$ satisfy
$\Stab_{S_d}(\bi) \cap \Stab_{S_d}(\bk) = \Stab_{S_d}(\bj)$.
Then $Z(A,B,C) = 1$ if $(\bi,\bk) \in \Pi_C$, and $Z(A,B,C) = 0$ otherwise.
\end{lemma}

\begin{proof}
Pick $(\bi',\bk') \in \Pi_C$.
To calculate $Z(A,B,C)$, we need to count the number of 
$\bj'$ such that $(\bi',\bj') \in \Pi_A$ and
$(\bj',\bk') \in \Pi_B$.
Equivalently, this is the number of $\bj'$
such that $(\bi',\bj') \sim (\bi,\bj)$ and $(\bj',\bk') \sim
(\bj,\bk)$.

If such a $\bj'$ exists, we can find $g \in S_d$ such that $\bj'\cdot g =
\bj$, then have that $(\bi'\cdot g, \bj) \sim (\bi,\bj)$ and $(\bj,
\bk'\cdot g) \sim (\bj,\bk)$.
So there is $h \in \Stab_{S_d}(\bj)$ such that
$\bi'\cdot g = \bi\cdot h$ and $\bk'\cdot g = \bk\cdot h$.
As 
$\Stab_{S_d}(\bj) \subseteq \Stab_{S_d}(\bi) \cap \Stab_{S_d}(\bk)$,
we deduce that $\bi'\cdot g = \bi$ and $\bk'\cdot g = \bk$, hence,
$(\bi,\bk) \in \Pi_C$.

Finally assume that $(\bi,\bk) \in \Pi_C$.
Then, we may as well assume that $(\bi',\bk')= (\bi,\bk)$, 
and $Z(A,B,C)$ is the number of $\bj'$ such that
$(\bi,\bj') \sim (\bi,\bj)$ and $(\bj',\bk) \sim (\bj,\bk)$.
Any such $\bj'$ can be written as $\bj\cdot g$ for $g \in
\Stab_{S_d}(\bi) \cap \Stab_{S_d}(\bk)$. 
As $\Stab_{S_d}(\bi) \cap \Stab_{S_d}(\bk) \subseteq
\Stab_{S_d}(\bj)$, 
we deduce that
$\bj' = \bj$. This shows that $Z(A,B,C) = 1$.
\end{proof}

The numbers $Z(A,B,C)$ arise naturally as the {\em structure
  constants} for multiplication in the Schur algebra.
To recall this, let $V_n$ be the defining representation of $G_n$ with standard
basis $v_1,\dots,v_n$.
The symmetric group $S_d$ acts on the right on the tensor space $V_n^{\otimes d}$ by
permuting tensors.
The {\em Schur algebra} is the endomorphism algebra
\begin{equation}\label{schurdef}
S(n,d) := \End_{S_d}(V_n^{\otimes d}).
\end{equation}
The action of $S_d$ on $V_n^{\otimes d}$ 
commutes with the action of $G_n$, hence, it leaves the
weight spaces of $V_n^{\otimes d}$ invariant. The weights which
arise are the ones in the set
\begin{equation}
\Lambda(n,d) := \{\lambda\vDash d\:|\:\ell(\lambda)=n\}.
\end{equation}
We deduce that the projection $1_\lambda$ of $V_n^{\otimes d}$ onto its
$\lambda$-weight space
gives an idempotent in the Schur algebra. These so-called {\em weight idempotents} for all
$\lambda \in \Lambda(n,d)$ are mutually orthogonal and sum to the
identity in $S(n,d)$.
Note also that $1_\lambda V_n^{\otimes d}$ has basis
$\{v_\bi:= v_{i_1}\otimes\cdots\otimes v_{i_d}\:|\:\bi \in
\I_\lambda\}$, with the action of $g \in S_d$ on this basis satisfying
\begin{equation}\label{brian}
v_\bi g 
= v_{\bi\cdot g}.
\end{equation}
For $\lambda,\mu \in \Lambda(n,d)$
and $A \in \Theta_{\lambda,\mu}$, define the linear map
\begin{equation}\label{cats}
\xi_A:
1_\mu V_n^{\otimes d} \rightarrow 1_\lambda V_n^{\otimes d},
\qquad
v_\bj\mapsto \sum_{\bi\text{\:with\,} (\bi,\bj)\in\Pi_A}
v_\bi.
\end{equation}
The endomorphisms
$\left\{\xi_A\:|\:A \in \Theta_{\lambda,\mu}\right\}$
give {\em Schur's basis} for $1_\lambda S(n,d) 1_\mu$.
 Moreover,
multiplication in the Schur algebra satisfies
\begin{equation}\label{schurs}
\xi_A \circ \xi_B := 
\sum_{C \in \Theta_{\lambda,\nu}} Z(A,B,C)
\xi_C
\end{equation}
for $A \in \Theta_{\lambda,\mu}$ and $B \in
  \Theta_{\mu,\nu}$.
This is {\em Schur's product rule}; e.g., see \cite[2.3b]{Green}.

The algebra $S(n,d)$ can also be constructed starting from the general linear
group $G_n$; see \cite[Ch. 2]{Green}. 
From this approach, one sees that
the category $S(n,d)\operatorname{-mod}$ is identified with the full subcategory
of $\Rep{G_n}$ consisting of the
{\em 
polynomial representations of degree $d$}.
Another important aspect of the theory needed later is the {\em Schur functor}
\begin{equation}\label{schurfunctor}
\T:
S(n,d)\operatorname{-mod} \rightarrow
\k S_d\operatorname{-mod}
\end{equation}
as in \cite[Ch. 6]{Green}.
In Green's approach, 
this is defined only when $n \geq d$, so that the composition $\omega := (1^d,
0^{n-d})$ belongs to $\Lambda(n,d)$.
There is an 
algebra isomorphism 
\begin{equation}\label{cruzy}
\k S_d \stackrel{\sim}{\rightarrow} 1_\omega S(n,d) 1_\omega, 
\qquad
g \mapsto \xi_A
\end{equation}
where $A \in \Theta_{\omega,\omega}$ is the $n\times n$ matrix with
$a_{g(1),1}=\cdots=a_{g(d),d} = 1$ and all other entries zero.
Identifying $\k S_d$ with $1_\omega S(n,d) 1_\omega$ in this way, 
$\T$ is the idempotent truncation functor 
associated to the weight idempotent $1_\omega$.
Note also that there is an isomorphism of $(S(n,d), \k S_d)$-bimodules
\begin{equation}\label{crazy}
V_n^{\otimes d} \stackrel{\sim}{\rightarrow} S(n,d) 1_\omega,
\qquad
v_\bi \mapsto \xi_A
\end{equation}
where $A$ here
is the $n\times n$ matrix with $a_{i_1,1} = \cdots =
a_{i_d,d} = 1$ and all other entries zero.
It follows that the Schur functor $\T$ is isomorphic to
$\Hom_{G_n}(V_n^{\otimes d}, -)$.

\begin{definition}\label{scat}
The {\em Schur category}
is the strict monoidal category 
$\Schur$ with
\begin{itemize}
\item objects that are all strict compositions $\lambda \Vdash d$ for all $d \geq 0$;
\item 
for $\lambda\Vdash d$ and $\mu \Vdash d'$,
the morphism space $\Hom_{\Schur}(\mu,\lambda)$ is zero unless
$d=d'$, and it is 
the vector 
space with basis $\{\xi_A\:|\:A \in \Theta_{\lambda,\mu}\}$ if $d=d'$;
\item
the tensor product of objects is defined by concatenation
$\lambda \otimes \mu := \lambda \sqcup \mu$;
\item
the tensor product 
of morphisms is defined by $\xi_A \otimes \xi_B :=
\xi_{\diag(A,B)}$, where $\diag(A,B)$ is the obvious block diagonal matrix;
\item
vertical composition of morphisms is defined by Schur's product rule
as in (\ref{schurs}).
\end{itemize}
We leave it to the reader to check that the axioms of
a strict monoidal category are satisfied. 
The unit object $\mathbbm{1}$ is the composition of length zero, and
the identity endomorphism $1_\lambda$ of an object
$\lambda \in \Schur$ is $\xi_{\diag(\lambda_1,\dots,\lambda_{\ell(\lambda)})}$.
\end{definition}

\begin{remark}\label{connectr}
Assuming that $n \geq d$, let $\Lambda(n,d)_L$ be the set of
compositions $\lambda \in \Lambda(n,d)$ that are {\em left-justified}, meaning
that
$\lambda = (\lambda_1,\dots,\lambda_m, 0^{n-m})$ with
$\lambda_1,\dots,\lambda_m > 0$.
Let $e := \sum_{\lambda \in \Lambda(n,d)_L} 1_\lambda \in S(n,d)$.
Any weight idempotent in $S(n,d)$ is conjugate to a left-justified
one, hence, 
the algebras $S(n,d)$ and $e S(n,d) e$ are {Morita equivalent}.
Moreover, there is an obvious algebra isomorphism
\begin{equation}\label{connect}
e S(n,d) e =
\bigoplus_{\lambda, \mu \in \Lambda(n,d)_L}
1_\lambda S(n,d) 1_\mu
\cong \bigoplus_{\lambda,\mu\Vdash d}
\Hom_{\Schur}(\mu,\lambda).
\end{equation}
This makes the connection between the Schur algebra and the Schur
category precise.
\end{remark}

\begin{remark}
By (\ref{connect}) and \cite[Theorem 3.2]{FS}, 
the category $\Schur\operatorname{-mod_{fd}}$ of globally
finite-dimensional $\Schur$-modules, i.e., the category of
functors
$V:\Schur \rightarrow \Vect$ 
such that $\bigoplus_{\lambda \in \Schur}
V(\lambda)$ is finite-dimensional, is equivalent to
the
category $\mathcal{P}ol$ of (strict) {\em polynomial functors} from \cite{FS}.
Under this equivalence, the projective $\Schur$-module 
$\Hom_{\Schur}((n),-)$ corresponds to the $n$th
divided power functor $\Gamma^n$.
The category of polynomial functors is symmetric monoidal with a biexact tensor
product functor $-\otimes-$ (see e.g. \cite[Proposition 2.6]{FS}). This structure 
can also be seen directly on $\Schur\operatorname{-mod_{fd}}$
in terms of an induction functor extending the tensor product on the underlying
monoidal category $\Schur$.
In fact, $\mathcal{P}ol$ is the Abelian envelope of the Karoubian
monoidal category $\Schur$ in a
precise sense: any 
functor $F:\Schur \rightarrow \mathcal{A}$ to an Abelian category
$\mathcal{A}$
factors through the embedding $\Schur \to \mathcal{P}ol, \; Z \mapsto
\Hom(Z, -)^*$ to induce a right-exact functor
$\mathcal{P}ol\rightarrow \mathcal{A}$, which is monoidal in case $F$
is monoidal.
\end{remark}

There are some special families of morphisms
$\xi_A$ in the Schur category
which are easy to understand.
\begin{itemize}
\item
If $A$ is a $1 \times n$ row matrix, we call $\xi_A$ an {\em $n$-fold
merge}; the reason for the terminology will become clear when we
switch to the diagrammatic formalism below.
By Schur's product rule, we have in the Schur category that 
\begin{equation}\label{mss}
\xi_{\left(\begin{smallmatrix}\lambda_1&\cdots&
      \lambda_n\end{smallmatrix}\right)}= 
\xi_{\left(\begin{smallmatrix}\lambda_1+\cdots+\lambda_m&\lambda_{m+1}+\cdots+\lambda_n\end{smallmatrix}\right)}\circ
\left(\xi_{\left(\begin{smallmatrix}\lambda_1&\cdots&
      \lambda_m\end{smallmatrix}\right)}\otimes
\xi_{\left(\begin{smallmatrix}\lambda_{m+1}&\cdots&
      \lambda_n\end{smallmatrix}\right)}\right)
\end{equation}
for $\lambda_1,\dots,\lambda_n > 0$ 
and $1 \leq m < n$; cf. (\ref{spots}) below.
Using this formula recursively, it follows that any $n$-fold merge can
be expressed as a composition of tensor products of two-fold merges
$\xi_{\left(\begin{smallmatrix}a&b\end{smallmatrix}\right)}$.
\item
If $A$ is an $n \times 1$ column matrix, we call $\xi_A$ an {\em
  $n$-fold split}. By the analogous (in fact, transpose) formula to (\ref{mss}), in the Schur category,
any $n$-fold split can be expressed as a composition of tensor
products of two-fold splits $\xi_{\left(\begin{smallmatrix}a\\b\end{smallmatrix}\right)}$.
\item
If $A$ is an $n \times n$ monomial matrix, i.e., it has exactly one
non-zero entry in every row and column, we call $\xi_A$ a {\em generalized
  permutation}.
Letting $\lambda$ and $\mu$ be the row and column sums of $A$, so that
$A \in \Theta_{\lambda,\mu}$, we may also use the notation
\begin{equation}\label{genperm}
1_\lambda g = g 1_\mu := \xi_A
\end{equation}
where $g \in S_n$ is defined from $\lambda = g(\mu)$;
here we are using the left action of $S_n$ on $\Lambda(n,d)$ so
$g(\mu) = (\mu_{g^{-1}(1)}, \dots, \mu_{g^{-1}(n)})$.
In other words, $g$ is the permutation such that
$a_{g(1),1} = \mu_1,\dots,a_{g(n),n} = \mu_n$.
Given 
another permutation $h \in S_n$,
Schur's product rule implies that
\begin{equation}\label{genperm2}
1_\lambda (gh) =
 g 1_{\mu} \circ 1_\mu h
= (gh) 1_\nu
\end{equation}
for $\mu = h(\nu)$.
This may also be deduced as a special case of the following lemma.
\end{itemize}

\begin{lemma}\label{harderthanyouthink}
Suppose that $A \in \Theta_{\lambda,\mu}$ and $B \in \Theta_{\mu,\nu}$
for $\lambda,\mu,\nu\Vdash d$.
Assume:
\begin{itemize}
\item
$A$ has a unique non-zero entry in every column,
so that there is an associated function $\alpha:\nset{\ell(\mu)}\rightarrow \nset{\ell(\lambda)}$
sending $i$ to the unique $j$ such that $a_{j,i} \neq 0$;
\item
$B$ has a unique non-zero entry in every row,
 so that there is an associated function
$\beta:\nset{\ell(\mu)} \rightarrow \nset{\ell(\nu)}$
sending $i$ to the unique $j$ such that $b_{i,j} \neq 0$;
\item
the function
$\gamma:\nset{\ell(\mu)}\rightarrow\nset{\ell(\lambda)}\times\nset{\ell(\nu)},
i\mapsto (\alpha(i),\beta(i))$ is injective.
\end{itemize}
Then $\xi_A \circ \xi_B = \xi_C$
where $C \in \Theta_{\lambda,\nu}$ is the matrix with 
$c_{\alpha(i),\beta(i)}=\mu_i$ for $i \in \nset{\ell(\mu)}$, all other entries being zero.
\end{lemma}

\begin{proof}
Let $n
:= \ell(\mu)$ and 
$\bj := (1^{\mu_1},2^{\mu_2}, \dots, n^{\mu_n})$.
Let $\bi := \alpha(\bj)$ and $\bk := \beta(\bj)$, i.e., these are the tuples
obtained by applying the functions $\alpha$ and $\beta$ to the entries of $\bj$.
Then we have that $(\bi,\bj) \in \Pi_A, (\bj,\bk) \in \Pi_B$, and
$(\bi, \bk) \in \Pi_C$.
Moreover, the injectivity of $\gamma$ implies that
$\Stab_{S_d}(\bi) \cap \Stab_{S_d}(\bk) = \Stab_{S_d}(\bj)$. Now apply Lemma~\ref{tricky}.
\end{proof}

Now suppose that $A \in \Theta_{\lambda,\mu}$ for
$\lambda, \mu \Vdash d$.
\begin{itemize}
\item
Let $A^-$ be the block diagonal
matrix $\diag(A_1,\dots,A_{\ell(\lambda)})$ where $A_i$ is the 
$1 \times n_i$
matrix obtained from the $i$th row of $A$ by removing all entries $0$.
Note that
\begin{equation}\label{dare}
\xi_{A^-} = \xi_{A_1} \otimes \cdots \otimes \xi_{A_{\ell(\lambda)}},
\end{equation}
with each $\xi_{A_i}$ being an $n_i$-fold merge.
Also let $\lambda^-$ be the composition recording the column sums of
$A^-$, so that $A^- \in \Theta_{\lambda,\lambda^-}$.
The $i$th entry $\lambda^-_i$ of $\lambda^-$
is the $i$th
{non-zero} entry
of the sequence
$a_{1,1},a_{1,2},\dots, a_{1,\ell(\mu)}, a_{2,1},\dots$
that is the {\em row reading} of the matrix $A$.
\item
Let $A^+$ be the block diagonal matrix
$\diag(A^1,\dots,A^{\ell(\mu)})$ where $A^i$ is the 
$n^i \times 1$ matrix obtained from the $i$th column of $A$ by removing all entries
$0$.
We then have that 
\begin{equation}\label{dare2}
\xi_{A^+} = \xi_{A^1} \otimes \cdots \otimes \xi_{A^{\ell(\mu)}},
\end{equation}
with each $\xi_{A^i}$ being an $n^i$-fold split.
Also let $\mu^+$ be the composition recording the row sums of $A^+$, so 
that $A^+ \in \Theta_{\mu^+,\mu}$.
The $i$th entry $\mu^+_i$ of $\mu^+$ is the $i$th {non-zero}
entry of the sequence $a_{1,1}, a_{2,1},\dots,a_{\ell(\lambda),1},
a_{1,2},\dots$
that is the {\em column reading} of $A$.
\item 
The composition 
$\lambda^-$ is a rearrangement of $\mu^+$, in particular, $n :=
\ell(\lambda^-) = \ell(\mu^+)$.
Let $f_1:\nset{n} \rightarrow \nset{\ell(\lambda)}$ and $f_2:\nset{n}\rightarrow \nset{\ell(\mu)}$ 
be defined so that $\lambda_i^-$, the $i$th non-zero entry
of the row reading of $A$, is in row $f_1(i)$ and column $f_2(i)$.
Let
$h_1:\nset{n} \rightarrow \nset{\ell(\lambda)}$ and $h_2:\nset{n} \rightarrow \nset{\ell(\mu)}$ be 
defined so that $\mu_i^+$, the $i$th non-zero entry of the
column reading of $A$, is in row $h_1(i)$ and column $h_2(i)$.
There is then a unique permutation $g \in S_n$ such that
$(f_1(g(i)), f_2(g(i))) = (h_1(i), h_2(i))$ for each $i\in\nset{n}$.
We have in particular that
$g(\mu^+) = \lambda^-$. Let $A^\circ \in \Theta_{\lambda^-,\mu^+}$ 
be the $n \times n$ monomial matrix with $(g(i),i)$-entry equal to $\mu^+_i$ for
$i=1,\dots,n$, 
all other entries being zero.
We have that
\begin{equation}\label{dare3} 
\xi_{A^\circ} = g 1_{\mu^+},
\end{equation}
notation as in (\ref{genperm}).
\end{itemize}
For example, suppose that $A=
\begin{pmatrix} 1&0&3\\2&2&1\end{pmatrix}$, 
so $\lambda = (4,5)$ and $\mu = (3,2,4)$.
Then
\begin{align}
A^- &= \begin{pmatrix}1 & 3 & 0 & 0 & 0\\0 & 0 & 2 & 2 &
  1\end{pmatrix},&
A^\circ &= \begin{pmatrix}
1&0&0&0&0\\
0&0&0&3&0\\
0&2&0&0&0\\
0&0&2&0&0\\
0&0&0&0&1
\end{pmatrix},&
 A^+ &= \begin{pmatrix}1 & 0 & 0 \\ 2 & 0 & 0 \\ 0 & 2 & 0\\ 0 &
  0 & 3\\ 0 & 0 & 1\end{pmatrix}.\label{ducks}
\end{align}
Also
$\lambda^- = (1,3,2,2,1)$
and $\mu^+ = (1,2,2,3,1)$,
so that $\xi_{A^\circ} = g 1_{\mu^+}$ where $g = (2\:3\:4)$;
see also (\ref{firstex}) below for a helpful picture of this situation.

\begin{lemma}\label{fixed}
For $A \in \Theta_{\lambda,\mu}$, we have that
$\xi_A = \xi_{A^-} \circ \xi_{A^\circ} \circ \xi_{A^+}$.
\end{lemma}

\begin{proof}
Define $n, \lambda^-, \mu^+$ and $f_1,f_2,g,h_1,h_2$ as above.
First, we 
apply Lemma~\ref{harderthanyouthink} 
with $\alpha = g$ and $\beta = h_2$
to deduce that $\xi_{A^\circ} \circ \xi_{A^+} = \xi_B$
for $B \in \Theta_{\lambda^-,\mu}$ defined so that
$b_{g(i), h_2(i)} = \mu^+_i$ for $i=1,\dots,n$, all other entries
being zero.
Then apply it again with $\alpha = f_1$ and $\beta = g \circ h_2$ 
to show that $\xi_{A^-} \circ \xi_B = \xi_A$.
\end{proof}

Lemma~\ref{fixed} shows that any $\xi_A$ can be expressed as the
vertical composition of some tensor product of merges, 
a generalized permutation, and some tensor product of splits. 
This statement is very natural from the diagrammatic point of view
which we are going to explain next. 

In fact, we are going to prove that $\Schur$
is isomorphic to a version of the web category 
from \cite[$\S$5]{CKM}\footnote{This extended work of G. Kuperberg to whom the reference to
spiders is credited.} for polynomial representations of the
general linear group, but in the stable limit as the rank tends to
infinity. This stable version, which is well known to the experts,
is easier than the finite rank version in \cite{CKM}
since one can exploit the connection to $\Schur$ and the defining basis for morphism spaces in 
the latter category. We will explain this in detail below since it is hard to
extract from the existing literature.
See also Remark~\ref{tosay} which explains how to recover the finite
rank
cases (together with a natural basis for their morphism spaces) via this
approach.

\begin{definition}\label{webcat}
The {\em polynomial web category} $\Web$ is the strict monoidal category 
defined by generators and relations as follows.
Its objects are all strict compositions with tensor product being by
concatenation as in Definition~\ref{scat}. 
The one-part compositions $(a)$ for $a > 0$ give a family of
generating objects.
In string diagrams, we will represent the generating object $(a)$
as a string labeled by the {\em thickness} $a$,
and a general object $\lambda = (\lambda_1,\dots,\lambda_n)$
will be a sequence of strings of thicknesses
$\lambda_1,\dots,\lambda_n>0$ in order from left to right.
Then there are generating morphisms
\begin{align}\label{ms}
\begin{tikzpicture}[baseline = -.5mm]
	\draw[-,line width=1pt] (0.28,-.3) to (0.08,0.04);
	\draw[-,line width=1pt] (-0.12,-.3) to (0.08,0.04);
	\draw[-,line width=2pt] (0.08,.4) to (0.08,0);
        \node at (-0.22,-.4) {$\scriptstyle a$};
        \node at (0.35,-.4) {$\scriptstyle b$};
\end{tikzpicture} 
&:(a,b) \rightarrow (a+b),&
\begin{tikzpicture}[baseline = -.5mm]
	\draw[-,line width=2pt] (0.08,-.3) to (0.08,0.04);
	\draw[-,line width=1pt] (0.28,.4) to (0.08,0);
	\draw[-,line width=1pt] (-0.12,.4) to (0.08,0);
        \node at (-0.22,.5) {$\scriptstyle a$};
        \node at (0.36,.5) {$\scriptstyle b$};
\end{tikzpicture}
&:(a+b)\rightarrow (a,b),
&\begin{tikzpicture}[baseline=-.5mm]
	\draw[-,thick] (-0.3,-.3) to (.3,.4);
	\draw[-,thick] (0.3,-.3) to (-.3,.4);
        \node at (0.3,-.4) {$\scriptstyle b$};
        \node at (-0.3,-.4) {$\scriptstyle a$};
\end{tikzpicture}
&:(a,b) \rightarrow (b,a)
\end{align}
for $a,b > 0$, which we call the two-fold merge, the two-fold split,
and the thick crossing, respectively.
The generating morphisms are subject to the following relations
for $a,b,c,d > 0$ with $d-a=c-b$:
\begin{align}
\label{assrel}
\begin{tikzpicture}[baseline = 0]
	\draw[-,thick] (0.35,-.3) to (0.08,0.14);
	\draw[-,thick] (0.1,-.3) to (-0.04,-0.06);
	\draw[-,line width=1pt] (0.085,.14) to (-0.035,-0.06);
	\draw[-,thick] (-0.2,-.3) to (0.07,0.14);
	\draw[-,line width=2pt] (0.08,.45) to (0.08,.1);
        \node at (0.45,-.41) {$\scriptstyle c$};
        \node at (0.07,-.4) {$\scriptstyle b$};
        \node at (-0.28,-.41) {$\scriptstyle a$};
\end{tikzpicture}
&=
\begin{tikzpicture}[baseline = 0]
	\draw[-,thick] (0.36,-.3) to (0.09,0.14);
	\draw[-,thick] (0.06,-.3) to (0.2,-.05);
	\draw[-,line width=1pt] (0.07,.14) to (0.19,-.06);
	\draw[-,thick] (-0.19,-.3) to (0.08,0.14);
	\draw[-,line width=2pt] (0.08,.45) to (0.08,.1);
        \node at (0.45,-.41) {$\scriptstyle c$};
        \node at (0.07,-.4) {$\scriptstyle b$};
        \node at (-0.28,-.41) {$\scriptstyle a$};
\end{tikzpicture}\:,
\qquad
\begin{tikzpicture}[baseline = -1mm]
	\draw[-,thick] (0.35,.3) to (0.08,-0.14);
	\draw[-,thick] (0.1,.3) to (-0.04,0.06);
	\draw[-,line width=1pt] (0.085,-.14) to (-0.035,0.06);
	\draw[-,thick] (-0.2,.3) to (0.07,-0.14);
	\draw[-,line width=2pt] (0.08,-.45) to (0.08,-.1);
        \node at (0.45,.4) {$\scriptstyle c$};
        \node at (0.07,.42) {$\scriptstyle b$};
        \node at (-0.28,.4) {$\scriptstyle a$};
\end{tikzpicture}
=\begin{tikzpicture}[baseline = -1mm]
	\draw[-,thick] (0.36,.3) to (0.09,-0.14);
	\draw[-,thick] (0.06,.3) to (0.2,.05);
	\draw[-,line width=1pt] (0.07,-.14) to (0.19,.06);
	\draw[-,thick] (-0.19,.3) to (0.08,-0.14);
	\draw[-,line width=2pt] (0.08,-.45) to (0.08,-.1);
        \node at (0.45,.4) {$\scriptstyle c$};
        \node at (0.07,.42) {$\scriptstyle b$};
        \node at (-0.28,.4) {$\scriptstyle a$};
\end{tikzpicture}\:,\end{align}\begin{align}
\label{trivial}
\begin{tikzpicture}[anchorbase,scale=.8]
	\draw[-,line width=2pt] (0.08,-.8) to (0.08,-.5);
	\draw[-,line width=2pt] (0.08,.3) to (0.08,.6);
\draw[-,thick] (0.1,-.51) to [out=45,in=-45] (0.1,.31);
\draw[-,thick] (0.06,-.51) to [out=135,in=-135] (0.06,.31);
        \node at (-.33,-.05) {$\scriptstyle a$};
        \node at (.45,-.05) {$\scriptstyle b$};
\end{tikzpicture}
&= 
\binom{a+b}{a}\:\:
\begin{tikzpicture}[anchorbase,scale=.8]
	\draw[-,line width=2pt] (0.08,-.8) to (0.08,.6);
        \node at (.62,-.05) {$\scriptstyle a+b$};
\end{tikzpicture},\end{align}
\begin{align}
\label{mergesplit}
\begin{tikzpicture}[anchorbase,scale=1]
	\draw[-,line width=1.2pt] (0,0) to (.275,.3) to (.275,.7) to (0,1);
	\draw[-,line width=1.2pt] (.6,0) to (.315,.3) to (.315,.7) to (.6,1);
        \node at (0,1.13) {$\scriptstyle b$};
        \node at (0.63,1.13) {$\scriptstyle d$};
        \node at (0,-.1) {$\scriptstyle a$};
        \node at (0.63,-.1) {$\scriptstyle c$};
\end{tikzpicture}
&=
\sum_{\substack{0 \leq s \leq \min(a,b)\\0 \leq t \leq \min(c,d)\\t-s=d-a}}
\begin{tikzpicture}[anchorbase,scale=1]
	\draw[-,thick] (0.58,0) to (0.58,.2) to (.02,.8) to (.02,1);
	\draw[-,thick] (0.02,0) to (0.02,.2) to (.58,.8) to (.58,1);
	\draw[-,thin] (0,0) to (0,1);
	\draw[-,line width=1pt] (0.61,0) to (0.61,1);
        \node at (0,1.13) {$\scriptstyle b$};
        \node at (0.6,1.13) {$\scriptstyle d$};
        \node at (0,-.1) {$\scriptstyle a$};
        \node at (0.6,-.1) {$\scriptstyle c$};
        \node at (-0.1,.5) {$\scriptstyle s$};
        \node at (0.77,.5) {$\scriptstyle t$};
\end{tikzpicture}.
\end{align}
In diagrams for morphisms in $\Web$,
we often omit thickness labels on strings when
they are implicitly determined by the other labels. 
We have not defined any
morphisms that could be drawn as cups or caps, so the strings in 
these
diagrams have
singular points where crossings and splits/merges occur, but no
critical points of slope zero.
\end{definition}
 
The relation (\ref{assrel}) means that we can introduce more general 
{\em $n$-fold merges} and {\em $n$-fold splits} for $n \geq 2$ 
by composing the
two-fold ones in an obvious way (cf. (\ref{mss})). 
For example, the three-fold merges
and splits are defined from
\begin{align}
\begin{tikzpicture}[baseline = 0]
	\draw[-,thick] (0.35,-.3) to (0.09,0.14);
	\draw[-,thick] (0.08,-.3) to (.08,0.1);
	\draw[-,thick] (-0.2,-.3) to (0.07,0.14);
	\draw[-,line width=2pt] (0.08,.45) to (0.08,.1);
        \node at (0.45,-.41) {$\scriptstyle c$};
        \node at (0.07,-.4) {$\scriptstyle b$};
        \node at (-0.28,-.41) {$\scriptstyle a$};
\end{tikzpicture}
:=
\begin{tikzpicture}[baseline = 0]
	\draw[-,thick] (0.35,-.3) to (0.08,0.14);
	\draw[-,thick] (0.1,-.3) to (-0.04,-0.06);
	\draw[-,line width=1pt] (0.085,.14) to (-0.035,-0.06);
	\draw[-,thick] (-0.2,-.3) to (0.07,0.14);
	\draw[-,line width=2pt] (0.08,.45) to (0.08,.1);
        \node at (0.45,-.41) {$\scriptstyle c$};
        \node at (0.07,-.4) {$\scriptstyle b$};
        \node at (-0.28,-.41) {$\scriptstyle a$};
\end{tikzpicture}
&=
\begin{tikzpicture}[baseline = 0]
	\draw[-,thick] (0.36,-.3) to (0.09,0.14);
	\draw[-,thick] (0.06,-.3) to (0.2,-.05);
	\draw[-,line width=1pt] (0.07,.14) to (0.19,-.06);
	\draw[-,thick] (-0.19,-.3) to (0.08,0.14);
	\draw[-,line width=2pt] (0.08,.45) to (0.08,.1);
        \node at (0.45,-.41) {$\scriptstyle c$};
        \node at (0.07,-.4) {$\scriptstyle b$};
        \node at (-0.28,-.41) {$\scriptstyle a$};
\end{tikzpicture}\:,&
\begin{tikzpicture}[baseline = 0]
	\draw[-,thick] (0.35,.3) to (0.09,-0.14);
	\draw[-,thick] (0.08,.3) to (.08,-0.1);
	\draw[-,thick] (-0.2,.3) to (0.07,-0.14);
	\draw[-,line width=2pt] (0.08,-.45) to (0.08,-.1);
        \node at (0.45,.41) {$\scriptstyle c$};
        \node at (0.07,.43) {$\scriptstyle b$};
        \node at (-0.28,.41) {$\scriptstyle a$};
\end{tikzpicture}
:=
\begin{tikzpicture}[baseline = 0]
	\draw[-,thick] (0.35,.3) to (0.08,-0.14);
	\draw[-,thick] (0.1,.3) to (-0.04,0.06);
	\draw[-,line width=1pt] (0.085,-.14) to (-0.035,0.06);
	\draw[-,thick] (-0.2,.3) to (0.07,-0.14);
	\draw[-,line width=2pt] (0.08,-.45) to (0.08,-.1);
        \node at (0.45,.41) {$\scriptstyle c$};
        \node at (0.07,.43) {$\scriptstyle b$};
        \node at (-0.28,.41) {$\scriptstyle a$};
\end{tikzpicture}
&=
\begin{tikzpicture}[baseline = 0]
	\draw[-,thick] (0.36,.3) to (0.09,-0.14);
	\draw[-,thick] (0.06,.3) to (0.2,.05);
	\draw[-,line width=1pt] (0.07,-.14) to (0.19,.06);
	\draw[-,thick] (-0.19,.3) to (0.08,-0.14);
	\draw[-,line width=2pt] (0.08,-.45) to (0.08,-.1);
        \node at (0.45,.41) {$\scriptstyle c$};
        \node at (0.07,.43) {$\scriptstyle b$};
        \node at (-0.28,.41) {$\scriptstyle a$};
\end{tikzpicture}\:.
\end{align}
By the symmetry of Definition~\ref{webcat}, there are isomorphisms of strict
monoidal categories
\begin{align}
\mathtt{T}:\Web &\rightarrow \Web^{\operatorname{op}},&
\mathtt{R}:\Web &\rightarrow \Web^{\operatorname{rev}}
\end{align}
defined by reflecting diagrams in a horizontal or vertical axis,
respectively.

We will need various other relations which are consequences of the
defining relations. 
The proofs of these are elementary relation chases and will be explained in the appendix.
\begin{align}
\begin{tikzpicture}[anchorbase,scale=1]
	\draw[-,thick] (0,0) to (0,1);
	\draw[-,thick] (0.015,0) to (0.015,.2) to (.57,.4) to (.57,.6)
        to (.015,.8) to (.015,1);
	\draw[-,line width=1.2pt] (0.6,0) to (0.6,1);
        \node at (0.6,-.1) {$\scriptstyle b$};
        \node at (0,-.1) {$\scriptstyle a$};
        \node at (0.3,.82) {$\scriptstyle c$};
        \node at (0.3,.19) {$\scriptstyle d$};
\end{tikzpicture}
&=
\sum_{t=\max(0,c-b)}^{\min(c,d)}
\binom{a\!-\!d\!+\!t}{t}
\begin{tikzpicture}[anchorbase,scale=1]
	\draw[-,thin] (0,0) to (0,1);
	\draw[-,thick] (0.02,0) to (0.02,.2) to (.88,.8) to (.88,1);
	\draw[-,thick] (0.88,0) to (0.88,.2) to (.02,.8) to (.02,1);
	\draw[-,thin] (0.9,0) to (0.9,1);
        \node at (0,-.1) {$\scriptstyle a$};
        \node at (0.9,-.1) {$\scriptstyle b$};
        \node at (.33,.18) {$\scriptstyle d-t$};
        \node at (.33,.84) {$\scriptstyle c-t$};
\end{tikzpicture}=
\sum_{t=\max(0,c-b)}^{\min(c,d)}
\binom{a\!-\!b\!+\!c\!-\!d}{t}
\begin{tikzpicture}[anchorbase,scale=1]
	\draw[-,line width=1.2pt] (0,0) to (0,1);
	\draw[-,thick] (0.79,0) to (0.79,.2) to (.03,.4) to (.03,.6)
        to (.79,.8) to (.79,1);
	\draw[-,thin] (0.81,0) to (0.81,1);
        \node at (0.8,-.1) {$\scriptstyle b$};
        \node at (0,-.1) {$\scriptstyle a$};
        \node at (0.38,.9) {$\scriptstyle d-t$};
        \node at (0.38,.13) {$\scriptstyle c-t$};
\end{tikzpicture},
\label{jonsquare}
\\
\begin{tikzpicture}[anchorbase,scale=1]
	\draw[-,thick] (0,0) to (0,1);
	\draw[-,thick] (-.015,0) to (-0.015,.2) to (-.57,.4) to (-.57,.6)
        to (-.015,.8) to (-.015,1);
	\draw[-,line width=1.2pt] (-0.6,0) to (-0.6,1);
        \node at (-0.6,-.1) {$\scriptstyle a$};
        \node at (0,-.1) {$\scriptstyle b$};
        \node at (-0.3,.84) {$\scriptstyle d$};
        \node at (-0.3,.19) {$\scriptstyle c$};
\end{tikzpicture}
&=
\sum_{t=\max(0,d-a)}^{\min(c,d)}
\binom{b\!-\!c\!+\!t}{t}
\begin{tikzpicture}[anchorbase,scale=1]
	\draw[-,thin] (0,0) to (0,1);
	\draw[-,thick] (0.02,0) to (0.02,.2) to (.88,.8) to (.88,1);
	\draw[-,thick] (0.88,0) to (0.88,.2) to (.02,.8) to (.02,1);
	\draw[-,thin] (0.9,0) to (0.9,1);
        \node at (0,-.1) {$\scriptstyle a$};
        \node at (.9,-.1) {$\scriptstyle b$};
        \node at (.33,.18) {$\scriptstyle c-t$};
        \node at (.33,.84) {$\scriptstyle d-t$};
\end{tikzpicture}=
\sum_{t=\max(0,d-a)}^{\min(c,d)}
\binom{b\!-\!a\!+\!d\!-\!c}{t}
\begin{tikzpicture}[anchorbase,scale=1]
	\draw[-,line width=1.2pt] (0,0) to (0,1);
	\draw[-,thick] (-0.8,0) to (-0.8,.2) to (-.03,.4) to (-.03,.6)
        to (-.8,.8) to (-.8,1);
	\draw[-,thin] (-0.82,0) to (-0.82,1);
        \node at (-0.81,-.1) {$\scriptstyle a$};
        \node at (0,-.1) {$\scriptstyle b$};
        \node at (-0.4,.9) {$\scriptstyle c-t$};
        \node at (-0.4,.13) {$\scriptstyle d-t$};
\end{tikzpicture},
\label{jonsquare2}
\end{align}\begin{align}
\label{thickcrossing}
\begin{tikzpicture}[anchorbase,scale=1]
	\draw[-,line width=1.2pt] (0,0) to (.6,1);
	\draw[-,line width=1.2pt] (0,1) to (.6,0);
        \node at (0,-.1) {$\scriptstyle a$};
        \node at (0.6,-0.1) {$\scriptstyle b$};
\end{tikzpicture}
&=
\begin{tikzpicture}[anchorbase,scale=1]
	\draw[-,line width=1.2pt] (0,0) to (.285,.3) to (.285,.7) to (0,1);
	\draw[-,line width=1.2pt] (.6,0) to (.315,.3) to (.315,.7) to (.6,1);
        \node at (0,-.1) {$\scriptstyle a$};
        \node at (0.6,-.1) {$\scriptstyle b$};
\end{tikzpicture}
\!-\!\sum_{t=1}^{\min(a,b)}
\begin{tikzpicture}[anchorbase,scale=1]
	\draw[-,thin] (0,0) to (0,1);
	\draw[-,thick] (0.02,0) to (0.02,.2) to (.58,.8) to (.58,1);
	\draw[-,thick] (0.58,0) to (0.58,.2) to (.02,.8) to (.02,1);
	\draw[-,thin] (0.6,0) to (0.6,1);
        \node at (0,-.1) {$\scriptstyle a$};
        \node at (0.6,-.1) {$\scriptstyle b$};
        \node at (-0.1,.5) {$\scriptstyle t$};
        \node at (0.7,.5) {$\scriptstyle t$};
\end{tikzpicture}
=\sum_{t=0}^{\min(a,b)}
(-1)^t
\begin{tikzpicture}[anchorbase,scale=1]
	\draw[-,thick] (0,0) to (0,1);
	\draw[-,thick] (0.015,0) to (0.015,.2) to (.57,.4) to (.57,.6)
        to (.015,.8) to (.015,1);
	\draw[-,line width=1.2pt] (0.6,0) to (0.6,1);
        \node at (0.6,-.1) {$\scriptstyle b$};
        \node at (0,-.1) {$\scriptstyle a$};
        \node at (-0.1,.5) {$\scriptstyle t$};
\end{tikzpicture}
=\sum_{t=0}^{\min(a,b)}
(-1)^t
\begin{tikzpicture}[anchorbase,scale=1]
	\draw[-,thick] (0,0) to (0,1);
	\draw[-,thick] (-.015,0) to (-0.015,.2) to (-.57,.4) to (-.57,.6)
        to (-.015,.8) to (-.015,1);
	\draw[-,line width=1.2pt] (-0.6,0) to (-0.6,1);
        \node at (-0.6,-.1) {$\scriptstyle a$};
        \node at (0,-.1) {$\scriptstyle b$};
        \node at (0.13,.5) {$\scriptstyle t$};
\end{tikzpicture},\end{align}
\begin{align}\label{serre}
2\!\begin{tikzpicture}[anchorbase,scale=1]
  \draw[-,thick](0,0) to (0,1);
  \draw[-,thick](0.5,0) to (0.5,1);
  \draw[-,thick](1,0) to (1,1);
  \draw[-,thin](0.02,0) to (0.02,.65) to (0.48,.85) to (0.48,1);
  \draw[-,thin](0.03,0) to (0.03,.1) to (0.48,.3) to (0.48,.5) to
  (0.98,.7) to (.98,1);
        \node at (0,-.1) {$\scriptstyle a+2$};
        \node at (.5,-.1) {$\scriptstyle b$};
        \node at (1,-.1) {$\scriptstyle c$};
        \node at (0,1.1) {$\scriptstyle a$};
        \node at (.4,1.1) {$\scriptstyle b+1$};
        \node at (1,1.1) {$\scriptstyle c+1$};
\end{tikzpicture}
\!=
\!\begin{tikzpicture}[anchorbase,scale=1]
  \draw[-,thick](0,0) to (0,1);
  \draw[-,line width=.5pt](0.5,0) to (0.5,1);
  \draw[-,thick](1,0) to (1,1);
  \draw[-,thin](0.02,0) to (0.02,.57) to (0.48,.77) to (0.48,1);
  \draw[-,thin](0.03,0) to (0.03,.3) to (0.49,.5) to (0.49,1);
  \draw[-,thin](0.49,0) to (0.49,.2) to (0.98,.4) to (0.98,1);
        \node at (0,-.1) {$\scriptstyle a+2$};
        \node at (.5,-.1) {$\scriptstyle b$};
        \node at (1,-.1) {$\scriptstyle c$};
        \node at (0,1.1) {$\scriptstyle a$};
        \node at (.4,1.1) {$\scriptstyle b+1$};
        \node at (1,1.1) {$\scriptstyle c+1$};
\end{tikzpicture}\!\!\!+\!\!\!
\begin{tikzpicture}[anchorbase,scale=1]
  \draw[-,thick](0,0) to (0,1);
  \draw[-,line width=.5pt](0.5,0) to (0.5,1);
  \draw[-,thick](1,0) to (1,1);
  \draw[-,thin](0.02,0) to (0.02,.33) to (0.48,.53) to (0.48,1);
  \draw[-,thin](0.03,0) to (0.03,.05) to (0.49,.25) to (0.49,1);
  \draw[-,thin](0.51,0) to (0.51,.75) to (0.98,.95) to (0.98,1);
        \node at (0,-.1) {$\scriptstyle a+2$};
        \node at (.5,-.1) {$\scriptstyle b$};
        \node at (1,-.1) {$\scriptstyle c$};
        \node at (0,1.1) {$\scriptstyle a$};
        \node at (.4,1.1) {$\scriptstyle b+1$};
        \node at (1,1.1) {$\scriptstyle c+1$};
\end{tikzpicture}\!\!,
\:\quad
2\!\begin{tikzpicture}[anchorbase,scale=1]
  \draw[-,thick](0,0) to (0,1);
  \draw[-,thick](-0.5,0) to (-0.5,1);
  \draw[-,thick](-1,0) to (-1,1);
  \draw[-,thin](-0.02,0) to (-0.02,.65) to (-0.48,.85) to (-0.48,1);
  \draw[-,thin](-0.03,0) to (-0.03,.1) to (-0.48,.3) to (-0.48,.5) to
  (-0.98,.7) to (-.98,1);
        \node at (0,-.1) {$\scriptstyle c+2$};
        \node at (-.5,-.1) {$\scriptstyle b$};
        \node at (-1,-.1) {$\scriptstyle a$};
        \node at (-0,1.1) {$\scriptstyle c$};
        \node at (-.4,1.1) {$\scriptstyle b+1$};
        \node at (-1,1.1) {$\scriptstyle a+1$};
\end{tikzpicture}
\!=
\!\begin{tikzpicture}[anchorbase,scale=1]
  \draw[-,thick](0,0) to (0,1);
  \draw[-,line width=.5pt](-0.5,0) to (-0.5,1);
  \draw[-,thick](-1,0) to (-1,1);
  \draw[-,thin](-0.02,0) to (-0.02,.57) to (-0.48,.77) to (-0.48,1);
  \draw[-,thin](-0.03,0) to (-0.03,.3) to (-0.49,.5) to (-0.49,1);
  \draw[-,thin](-0.49,0) to (-0.49,.2) to (-0.98,.4) to (-0.98,1);
        \node at (0,-.1) {$\scriptstyle c+2$};
        \node at (-.5,-.1) {$\scriptstyle b$};
        \node at (-1,-.1) {$\scriptstyle a$};
        \node at (-0,1.1) {$\scriptstyle c$};
        \node at (-.4,1.1) {$\scriptstyle b+1$};
        \node at (-1,1.1) {$\scriptstyle a+1$};
\end{tikzpicture}\!\!\!+\!\!\!
\begin{tikzpicture}[anchorbase,scale=1]
  \draw[-,thick](0,0) to (0,1);
  \draw[-,line width=.5pt](-0.5,0) to (-0.5,1);
  \draw[-,thick](-1,0) to (-1,1);
  \draw[-,thin](-0.02,0) to (-0.02,.33) to (-0.48,.53) to (-0.48,1);
  \draw[-,thin](-0.03,0) to (-0.03,.05) to (-0.49,.25) to (-0.49,1);
  \draw[-,thin](-0.51,0) to (-0.51,.75) to (-0.98,.95) to (-0.98,1);
        \node at (0,-.1) {$\scriptstyle c+2$};
        \node at (-.5,-.1) {$\scriptstyle b$};
        \node at (-1,-.1) {$\scriptstyle a$};
        \node at (-0,1.1) {$\scriptstyle c$};
        \node at (-.4,1.1) {$\scriptstyle b+1$};
        \node at (-1,1.1) {$\scriptstyle a+1$};
\end{tikzpicture}\!\!,
\end{align}
\begin{align}
\begin{tikzpicture}[anchorbase,scale=.7]
	\draw[-,line width=2pt] (0.08,.3) to (0.08,.5);
\draw[-,thick] (-.2,-.8) to [out=45,in=-45] (0.1,.31);
\draw[-,thick] (.36,-.8) to [out=135,in=-135] (0.06,.31);
        \node at (-.3,-.95) {$\scriptstyle a$};
        \node at (.45,-.95) {$\scriptstyle b$};
\end{tikzpicture}
&=
\begin{tikzpicture}[anchorbase,scale=.7]
	\draw[-,line width=2pt] (0.08,.1) to (0.08,.5);
\draw[-,thick] (.46,-.8) to [out=100,in=-45] (0.1,.11);
\draw[-,thick] (-.3,-.8) to [out=80,in=-135] (0.06,.11);
        \node at (-.3,-.95) {$\scriptstyle a$};
        \node at (.43,-.95) {$\scriptstyle b$};
\end{tikzpicture},
\qquad
\begin{tikzpicture}[baseline=.3mm,scale=.7]
	\draw[-,line width=2pt] (0.08,-.3) to (0.08,-.5);
\draw[-,thick] (-.2,.8) to [out=-45,in=45] (0.1,-.31);
\draw[-,thick] (.36,.8) to [out=-135,in=135] (0.06,-.31);
        \node at (-.3,.95) {$\scriptstyle a$};
        \node at (.45,.95) {$\scriptstyle b$};
\end{tikzpicture}
=
\begin{tikzpicture}[baseline=.3mm,scale=.7]
	\draw[-,line width=2pt] (0.08,-.1) to (0.08,-.5);
\draw[-,thick] (.46,.8) to [out=-100,in=45] (0.1,-.11);
\draw[-,thick] (-.3,.8) to [out=-80,in=135] (0.06,-.11);
        \node at (-.3,.95) {$\scriptstyle a$};
        \node at (.43,.95) {$\scriptstyle b$};
\end{tikzpicture},
\label{swallows}\end{align}\begin{align}
\begin{tikzpicture}[anchorbase,scale=0.7]
	\draw[-,thick] (0.4,0) to (-0.6,1);
	\draw[-,thick] (0.08,0) to (0.08,1);
	\draw[-,thick] (0.1,0) to (0.1,.6) to (.5,1);
        \node at (0.6,1.13) {$\scriptstyle c$};
        \node at (0.1,1.16) {$\scriptstyle b$};
        \node at (-0.65,1.13) {$\scriptstyle a$};
\end{tikzpicture}
&\!\!=\!\!
\begin{tikzpicture}[anchorbase,scale=0.7]
	\draw[-,thick] (0.7,0) to (-0.3,1);
	\draw[-,thick] (0.08,0) to (0.08,1);
	\draw[-,thick] (0.1,0) to (0.1,.2) to (.9,1);
        \node at (0.9,1.13) {$\scriptstyle c$};
        \node at (0.1,1.16) {$\scriptstyle b$};
        \node at (-0.4,1.13) {$\scriptstyle a$};
\end{tikzpicture},\:
\begin{tikzpicture}[anchorbase,scale=0.7]
	\draw[-,thick] (-0.4,0) to (0.6,1);
	\draw[-,thick] (-0.08,0) to (-0.08,1);
	\draw[-,thick] (-0.1,0) to (-0.1,.6) to (-.5,1);
        \node at (0.7,1.13) {$\scriptstyle c$};
        \node at (-0.1,1.16) {$\scriptstyle b$};
        \node at (-0.6,1.13) {$\scriptstyle a$};
\end{tikzpicture}
\!\!=\!\!
\begin{tikzpicture}[anchorbase,scale=0.7]
	\draw[-,thick] (-0.7,0) to (0.3,1);
	\draw[-,thick] (-0.08,0) to (-0.08,1);
	\draw[-,thick] (-0.1,0) to (-0.1,.2) to (-.9,1);
        \node at (0.4,1.13) {$\scriptstyle c$};
        \node at (-0.1,1.16) {$\scriptstyle b$};
        \node at (-0.95,1.13) {$\scriptstyle a$};
\end{tikzpicture},
\:\begin{tikzpicture}[baseline=-3.3mm,scale=0.7]
	\draw[-,thick] (0.4,0) to (-0.6,-1);
	\draw[-,thick] (0.08,0) to (0.08,-1);
	\draw[-,thick] (0.1,0) to (0.1,-.6) to (.5,-1);
        \node at (0.6,-1.13) {$\scriptstyle c$};
        \node at (0.07,-1.13) {$\scriptstyle b$};
        \node at (-0.6,-1.13) {$\scriptstyle a$};
\end{tikzpicture}
\!\!=\!\!
\begin{tikzpicture}[baseline=-3.3mm,scale=0.7]
	\draw[-,thick] (0.7,0) to (-0.3,-1);
	\draw[-,thick] (0.08,0) to (0.08,-1);
	\draw[-,thick] (0.1,0) to (0.1,-.2) to (.9,-1);
        \node at (1,-1.13) {$\scriptstyle c$};
        \node at (0.1,-1.13) {$\scriptstyle b$};
        \node at (-0.4,-1.13) {$\scriptstyle a$};
\end{tikzpicture},\:
\begin{tikzpicture}[baseline=-3.3mm,scale=0.7]
	\draw[-,thick] (-0.4,0) to (0.6,-1);
	\draw[-,thick] (-0.08,0) to (-0.08,-1);
	\draw[-,thick] (-0.1,0) to (-0.1,-.6) to (-.5,-1);
        \node at (0.6,-1.13) {$\scriptstyle c$};
        \node at (-0.1,-1.13) {$\scriptstyle b$};
        \node at (-0.6,-1.13) {$\scriptstyle a$};
\end{tikzpicture}
\!\!=\!\!
\begin{tikzpicture}[baseline=-3.3mm,scale=0.7]
	\draw[-,thick] (-0.7,0) to (0.3,-1);
	\draw[-,thick] (-0.08,0) to (-0.08,-1);
	\draw[-,thick] (-0.1,0) to (-0.1,-.2) to (-.9,-1);
        \node at (0.34,-1.13) {$\scriptstyle c$};
        \node at (-0.1,-1.13) {$\scriptstyle b$};
        \node at (-0.95,-1.13) {$\scriptstyle a$};
\end{tikzpicture},
\label{sliders}\end{align}\begin{align}
\label{symmetric}
\mathord{
\begin{tikzpicture}[baseline = -1mm,scale=0.8]
	\draw[-,thick] (0.28,0) to[out=90,in=-90] (-0.28,.6);
	\draw[-,thick] (-0.28,0) to[out=90,in=-90] (0.28,.6);
	\draw[-,thick] (0.28,-.6) to[out=90,in=-90] (-0.28,0);
	\draw[-,thick] (-0.28,-.6) to[out=90,in=-90] (0.28,0);
        \node at (0.3,-.75) {$\scriptstyle b$};
        \node at (-0.3,-.75) {$\scriptstyle a$};
\end{tikzpicture}
}&=
\mathord{
\begin{tikzpicture}[baseline = -1mm,scale=0.8]
	\draw[-,thick] (0.2,-.6) to (0.2,.6);
	\draw[-,thick] (-0.2,-.6) to (-0.2,.6);
        \node at (0.2,-.75) {$\scriptstyle b$};
        \node at (-0.2,-.75) {$\scriptstyle a$};
\end{tikzpicture}
}\:,\end{align}\begin{align}\label{braid}
\mathord{
\begin{tikzpicture}[baseline = -1mm,scale=0.8]
	\draw[-,thick] (0.45,.6) to (-0.45,-.6);
	\draw[-,thick] (0.45,-.6) to (-0.45,.6);
        \draw[-,thick] (0,-.6) to[out=90,in=-90] (-.45,0);
        \draw[-,thick] (-0.45,0) to[out=90,in=-90] (0,0.6);
        \node at (0,-.77) {$\scriptstyle b$};
        \node at (0.5,-.77) {$\scriptstyle c$};
        \node at (-0.5,-.77) {$\scriptstyle a$};
\end{tikzpicture}
}
&=
\mathord{
\begin{tikzpicture}[baseline = -1mm,scale=0.8]
	\draw[-,thick] (0.45,.6) to (-0.45,-.6);
	\draw[-,thick] (0.45,-.6) to (-0.45,.6);
        \draw[-,thick] (0,-.6) to[out=90,in=-90] (.45,0);
        \draw[-,thick] (0.45,0) to[out=90,in=-90] (0,0.6);
        \node at (0,-.77) {$\scriptstyle b$};
        \node at (0.5,-.77) {$\scriptstyle c$};
        \node at (-0.5,-.77) {$\scriptstyle a$};
\end{tikzpicture}
}\:.
\end{align}
The relations (\ref{sliders})--(\ref{braid}) imply that
$\Web$ has the structure of a strict symmetric monoidal
category, with symmetric braiding defined on generating objects by
the thick crossings.

\begin{remark}\label{mp}
In view of (\ref{thickcrossing}), the thick crossings can be
expressed in terms of the two-fold merges and splits, so they are
redundant as generators.
In fact, as will also be proved in the appendix,
$\Web$ is isomorphic to the strict monoidal category
with generators that are just the two-fold merges and splits, subject
to the relations (\ref{assrel}) and (\ref{trivial}) as before together
with the {\em square switch relations}
\begin{align}
\begin{tikzpicture}[anchorbase,scale=1]
	\draw[-,thick] (0,0) to (0,1);
	\draw[-,thick] (.015,0) to (0.015,.2) to (.57,.4) to (.57,.6)
        to (.015,.8) to (.015,1);
	\draw[-,line width=1.2pt] (0.6,0) to (0.6,1);
        \node at (0.6,-.1) {$\scriptstyle b$};
        \node at (0,-.1) {$\scriptstyle a$};
        \node at (0.3,.84) {$\scriptstyle c$};
        \node at (0.3,.19) {$\scriptstyle d$};
\end{tikzpicture}
&=
\sum_{t=\max(0,c-b)}^{\min(c,d)}
\binom{a\! -\!b\!+\!c\!-\!d}{t}
\begin{tikzpicture}[anchorbase,scale=1]
	\draw[-,line width=1.2pt] (0,0) to (0,1);
	\draw[-,thick] (0.8,0) to (0.8,.2) to (.03,.4) to (.03,.6)
        to (.8,.8) to (.8,1);
	\draw[-,thin] (0.82,0) to (0.82,1);
        \node at (0.81,-.1) {$\scriptstyle b$};
        \node at (0,-.1) {$\scriptstyle a$};
        \node at (0.4,.9) {$\scriptstyle d-t$};
        \node at (0.4,.13) {$\scriptstyle c-t$};
\end{tikzpicture},\label{sl2rel}\\
\label{sl2rel2}
\begin{tikzpicture}[anchorbase,scale=1]
	\draw[-,thick] (0,0) to (0,1);
	\draw[-,thick] (-.015,0) to (-0.015,.2) to (-.57,.4) to (-.57,.6)
        to (-.015,.8) to (-.015,1);
	\draw[-,line width=1.2pt] (-0.6,0) to (-0.6,1);
        \node at (-0.6,-.1) {$\scriptstyle a$};
        \node at (0,-.1) {$\scriptstyle b$};
        \node at (-0.3,.84) {$\scriptstyle d$};
        \node at (-0.3,.19) {$\scriptstyle c$};
\end{tikzpicture}
&=
\sum_{t=\max(0,d-a)}^{\min(c,d)}
\binom{b\! -\!a\!+\!d\!-\!c}{t}
\begin{tikzpicture}[anchorbase,scale=1]
	\draw[-,line width=1.2pt] (0,0) to (0,1);
	\draw[-,thick] (-0.8,0) to (-0.8,.2) to (-.03,.4) to (-.03,.6)
        to (-.8,.8) to (-.8,1);
	\draw[-,thin] (-0.82,0) to (-0.82,1);
        \node at (-0.81,-.1) {$\scriptstyle a$};
        \node at (0,-.1) {$\scriptstyle b$};
        \node at (-0.4,.9) {$\scriptstyle c-t$};
        \node at (-0.4,.13) {$\scriptstyle d-t$};
\end{tikzpicture},
\end{align}
which are as in (\ref{jonsquare})--(\ref{jonsquare2}) above.
This is the original presentation from
\cite[$\S$5]{CKM}, where
the square switch relations are interpreted in terms of the commutator
relation between the divided powers $e_i^{(c)}, f_i^{(d)}$.
From this perspective, the relations (\ref{serre}) come from the Serre relations.
Then the thick crossings get {\em defined}
from the formula
\begin{equation}
\begin{tikzpicture}[anchorbase,scale=1]
	\draw[-,line width=1.2pt] (0,0) to (.6,1);
	\draw[-,line width=1.2pt] (0,1) to (.6,0);
        \node at (0,-.1) {$\scriptstyle a$};
        \node at (0.6,-0.1) {$\scriptstyle b$};
\end{tikzpicture}
:=\sum_{t=0}^{\min(a,b)}
(-1)^t
\begin{tikzpicture}[anchorbase,scale=1]
	\draw[-,thick] (0,0) to (0,1);
	\draw[-,thick] (.015,0) to (0.015,.2) to (.57,.4) to (.57,.6)
        to (.015,.8) to (.015,1);
	\draw[-,line width=1.2pt] (0.6,0) to (0.6,1);
        \node at (0.6,-.1) {$\scriptstyle b$};
        \node at (0,-.1) {$\scriptstyle a$};
        \node at (-0.13,.5) {$\scriptstyle t$};
\end{tikzpicture}\label{askthink},
\end{equation}
which is \cite[Corollary 6.2.3]{CKM} (up to multiplication by the sign $(-1)^{ab}$ which also
appears in the statement of Theorem~\ref{dthm} below).
In \cite{CKM}, this formula is explained in terms of the action of the $i$th simple
reflection 
on the appropriate weight space of a polynomial representation of $GL_n$:
$s_i =
e_i^{(b)} f_i^{(a)} - e_i^{(b-1)} f_i^{(a-1)} + \cdots$.
\end{remark}

For $\lambda,\mu\Vdash d$,
a $\lambda \times \mu$ {\em chicken foot diagram}\footnote{This
  terminology was suggested to the first author by A. Kleshchev.}
is a diagram representing a morphism 
in $\Hom_{\Web}(\mu,\lambda)$ in which the thick strings determined by
$\mu$ at the bottom of the diagram split
into thinner strings,
then these thinner strings
cross each other in some way in
the middle of the diagram,
before merging back into the thick strings determined by $\lambda$ at the top.
This means that a chicken foot diagram has three distinct parts, the
top and bottom parts
which consist just of merges and splits, respectively, all of which occur
at the same horizontal level, and
the middle
part which is a generalized permutation diagram.
Here is an example with $\lambda = (4,5)$ and $\mu = (3,2,4)$:
\begin{align}\label{firstex}
  \begin{tikzpicture}[anchorbase,scale=1.5]
\draw[-,line width=.6mm] (.212,.5) to (.212,.39);
\draw[-,line width=.75mm] (.595,.5) to (.595,.39);
\draw[-,line width=.15mm] (0.0005,-.396) to (.2,.4);
\draw[-,line width=.3mm] (0.01,-.4) to (.59,.4);
\draw[-,line width=.3mm] (.4,-.4) to (.607,.4);
\draw[-,line width=.45mm] (.79,-.4) to (.214,.4);
\draw[-,line width=.15mm] (.8035,-.398) to (.614,.4);
\draw[-,line width=.3mm] (.4006,-.5) to (.4006,-.395);
\draw[-,line width=.6mm] (.788,-.5) to (.788,-.395);
\draw[-,line width=.45mm] (0.011,-.5) to (0.011,-.395);
\node at (0.05,0.15) {$\scriptstyle 1$};
\node at (0.79,0.05) {$\scriptstyle 1$};
\node at (0.35,0.37) {$\scriptstyle 3$};
\node at (0.17,-0.3) {$\scriptstyle 2$};
\node at (0.5,-0.3) {$\scriptstyle 2$};
\end{tikzpicture}.
\end{align}
We say that a chicken foot diagram is {\em reduced} if there is at
most one intersection or join between every pair of the
thinner strings in the diagram.
Thus, for each $i \in \nset{\ell(\lambda)}$ and $j \in \nset{\ell(\mu)}$,
there is at most one string connecting the $i$th vertex at
the top to the $j$th vertex at the bottom, and moreover
the generalized permutation diagram in the middle of the diagram
corresponds to a reduced word in the symmetric group.
The {\em type} of a reduced chicken foot diagram
is the matrix $A \in \Theta_{\lambda,\mu}$ whose $(i,j)$-entry
is the thickness of the unique string connecting the $i$th vertex
at the top to the $j$th vertex at the bottom, or zero if there is no
such string.
For example, (\ref{firstex}) is a reduced chicken foot diagram of
type $A =
\begin{pmatrix} 1&0&3\\2&2&1\end{pmatrix} \in \Theta_{\lambda,\mu}$,
and 
the top, middle and bottom
parts of (\ref{firstex}) are reduced chicken foot diagrams whose types
are given by the matrices $A^-$, $A^0$ and $A^+$ 
from (\ref{ducks}).

By the braid relations (\ref{braid}),
all reduced chicken foot diagrams of the same type $A \in \Theta_{\lambda,\mu}$
represent the same morphism $[A] \in \Hom_{\Web}(\mu,\lambda)$.
In fact, we are going to prove that these morphisms for all $A \in \Theta_{\lambda,\mu}$ give a
basis for space $\Hom_{\Web}(\mu,\lambda)$.
The fact that they span is established in the next lemma, which gives a straightening algorithm to
convert an arbitrary diagram for a morphism in $\Web$ into a linear
combination of reduced chicken foot diagrams.

\begin{lemma}\label{straightening}
The morphism space $\Hom_{\Web}(\mu,\lambda)$ is spanned 
by the morphisms $[A]$ for all $A \in \Theta_{\lambda,\mu}$.
\end{lemma}

\begin{proof}
We have observed already that $\Web$ is generated by its two-fold
merges and splits. Since these are themselves defined by
reduced chicken foot diagrams, it suffices to show 
for any morphism $f$
that consists of a two-fold
merge or a two-fold split
(tensored on the left and right by appropriate identity morphisms),
and any morphism $g$ defined by a reduced $\lambda \times \mu$ chicken foot diagram,
that the vertical
composition $f \circ g$ can be expressed as a linear combination of
reduced chicken foot diagrams.

Suppose first that $f$ involves a two-fold merge joining to the $i$th
and $(i+1)$th strings at the top of $g$. If $g$ has an $r$-fold merge
at its $i$th vertex and an $s$-fold merge at its $(i+1)$th vertex,
then we can use (\ref{assrel}) to rewrite $f \circ g$ so that it is a
$\lambda'\times \mu$ chicken foot diagram
with
an $(r+s)$-fold merge at its $i$th vertex, where
$\lambda'$ is the composition 
$(\lambda_1,\dots,\lambda_{i-1},\lambda_i+\lambda_{i+1},\lambda_{i+2},\dots,
\lambda_{\ell(\lambda)})$.
For example:
\begin{equation}\label{spots}
\begin{tikzpicture}[anchorbase]
	\draw[-,line width=2pt] (0.08,.3) to (0.08,-0.04);
	\draw[-,line width=1pt] (0.28,-.4) to (0.08,0);
	\draw[-,line width=1pt] (-0.12,-.4) to (0.08,0);
	\draw[-,thin] (-0.127,.-.39) to (-0.28,-.7);
	\draw[-,thin] (-0.118,.-.39) to (-0.118,-.7);
	\draw[-,thin] (-0.109,.-.39) to (0.04,-.7);
	\draw[-,line width=.5pt] (.28,.-.39) to (.15,-.7);
	\draw[-,line width=.5pt] (0.285,.-.39) to (0.4,-.7);
\node at (-0.7,.05) {$\scriptstyle f$};
\node at (-0.7,-0.5) {$\scriptstyle g$};
\end{tikzpicture}=
\begin{tikzpicture}[anchorbase]
	\draw[-,line width=2pt] (0.08,.3) to (0.08,-0.06);
	\draw[-,thin] (.053,.-.04) to (-0.32,-.7);
	\draw[-,thin] (.067,.-.04) to (-0.14,-.7);
	\draw[-,thin] (.081,.-.04) to (0.08,-.7);
	\draw[-,thin] (.095,.-.04) to (0.3,-.7);
	\draw[-,thin] (.108,.-.04) to (0.52,-.7);
\end{tikzpicture}.
\end{equation}
However the resulting chicken foot diagram is not necessarily reduced.
It remains to observe that the morphism defined by a 
non-reduced chicken foot diagram can be
converted to a scalar multiple of a morphism defined by a reduced one
just using
the relations (\ref{assrel})--(\ref{trivial}) and (\ref{swallows})--(\ref{braid}). 

Now suppose that $f$ involves a two-fold split joining to the $i$th
vertex at the top of $g$. Say this vertex of $g$ involves an $n$-fold merge.
Using (\ref{assrel}), (\ref{mergesplit}) and (\ref{sliders}), we
rewrite the composition of the split in $f$ and this merge in $g$ as a
sum of reduced chicken foot diagrams. For example:
\begin{equation}\label{pots}
\begin{tikzpicture}[anchorbase]
	\draw[-,line width=1.2pt] (-0.118,-.4) to (-0.118,-.18);
	\draw[-,thin] (-0.127,.-.39) to (-0.28,-.7);
	\draw[-,thin] (-0.118,.-.39) to (-0.118,-.7);
	\draw[-,thin] (-0.109,.-.39) to (0.04,-.7);
	\draw[-,line width=.8pt] (-.116,-.21) to (-.3,.1);
	\draw[-,line width=.8pt] (-0.12,-.21) to (0.05,.1);
\node at (-0.7,-.1) {$\scriptstyle f$};
\node at (-0.7,-0.54) {$\scriptstyle g$};
\end{tikzpicture}=
\sum\:
\begin{tikzpicture}[anchorbase]
	\draw[-,line width=.5pt] (0.089,.01) to (0.089,.04);
	\draw[-,line width=.5pt] (-0.089,.01) to (-0.089,.04);
	\draw[-,thin] (0.09,.03) to (-0.24,-.7);
	\draw[-,thin] (0.09,.03) to (0,-.7);
	\draw[-,thin] (0.09,.03) to (0.24,-.7);
	\draw[-,thin] (-0.09,.03) to (-0.24,-.7);
	\draw[-,thin] (-0.09,.03) to (0,-.7);
	\draw[-,thin] (-0.09,.03) to (0.24,-.7);
\end{tikzpicture}.
\end{equation}
Then compose these diagrams with the remainder of the diagram,
using (\ref{sliders}) then (\ref{assrel}) again
to commute the splits at the bottom of this part of the
resulting diagrams downwards past the generalized permutation part
of $g$. 
\end{proof}

\begin{theorem}\label{ckmrevised}
There is an isomorphism of strict monoidal categories
$$
F:\Web \stackrel{\sim}{\rightarrow} \Schur
$$
which is the identity on objects (i.e., strict compositions) and 
sends the morphism 
$[A]\in\Hom_{\Web}(\mu,\lambda)$
defined by a reduced chicken foot diagram of type $A \in \Theta_{\lambda,\mu}$
to Schur's basis element $\xi_A \in \Hom_{\Schur}(\mu,\lambda)$.
In particular, the functor $F$ sends the generating morphisms (\ref{ms})
to the two-fold merge
$\xi_{\left(\begin{smallmatrix}a&b\end{smallmatrix}\right)}$,
the two-fold split
$\xi_{\left(\begin{smallmatrix}a\\b\end{smallmatrix}\right)}$
and the generalized permutation
$\xi_{\left(\begin{smallmatrix}0&b\\a&0\end{smallmatrix}\right)}$,
respectively.
\end{theorem}

\begin{proof}
We define $F$ to be the identity on objects, and define it on
the generating morphisms for $\Web$ 
so that
\begin{align*}
\begin{tikzpicture}[baseline = -.5mm]
	\draw[-,line width=1pt] (0.28,-.3) to (0.08,0.04);
	\draw[-,line width=1pt] (-0.12,-.3) to (0.08,0.04);
	\draw[-,line width=2pt] (0.08,.4) to (0.08,0);
        \node at (-0.22,-.4) {$\scriptstyle a$};
        \node at (0.35,-.4) {$\scriptstyle b$};
\end{tikzpicture} 
&\mapsto \xi_{\left(\begin{smallmatrix}a&b\end{smallmatrix}\right)},
&
\begin{tikzpicture}[baseline = -.5mm]
	\draw[-,line width=2pt] (0.08,-.3) to (0.08,0.04);
	\draw[-,line width=1pt] (0.28,.4) to (0.08,0);
	\draw[-,line width=1pt] (-0.12,.4) to (0.08,0);
        \node at (-0.22,.5) {$\scriptstyle a$};
        \node at (0.36,.5) {$\scriptstyle b$};
\end{tikzpicture}
&\mapsto
\xi_{\left(\begin{smallmatrix}a\\b\end{smallmatrix}\right)},
&\begin{tikzpicture}[baseline=-.5mm]
	\draw[-,thick] (-0.3,-.3) to (.3,.4);
	\draw[-,thick] (0.3,-.3) to (-.3,.4);
        \node at (0.3,-.4) {$\scriptstyle b$};
        \node at (-0.3,-.4) {$\scriptstyle a$};
\end{tikzpicture}
&\mapsto
\xi_{\left(\begin{smallmatrix}0&b\\a&0\end{smallmatrix}\right)}.
\end{align*}
To see that this is well-defined, we just need to verify that the
defining 
relations
(\ref{assrel})--(\ref{mergesplit}) of $\Web$ are satisfied in $\Schur$. This is
an application of Schur's product rule; in particular, 
(\ref{assrel}) for merges follows by the identity (\ref{mss}) already
checked above.

Now take $A \in
\Theta_{\lambda,\mu}$.
The morphism $[A]\in\Hom_{\Web}(\mu,\lambda)$ 
is the vertical concatenation $[A^-] \circ [A^\circ] \circ
[A^+]$
for $A^-$, $A^\circ$ and $A^+$ defined prior to
Lemma~\ref{fixed}. This follows because the reduced
chicken foot diagrams for $A^-, A^\circ$ and $A^+$ 
give the top, middle and bottom parts of the one for $A$.
From (\ref{mss}) (and its analog for splits) and (\ref{dare})--(\ref{dare2}), it follows that
$F([A^-]) = \xi_{A^-}$ and $F([A^+]) = \xi_{A^+}$.
Also $[A^\circ]$ is a generalized permutation, 
so by (\ref{dare3}) we have that
$F([A^\circ]) = \xi_{A^\circ}$.
It remains to apply Lemma~\ref{fixed} to deduce that $F([A]) = \xi_A$.

Since the morphisms $\xi_A$ for $A \in \Theta_{\lambda,\mu}$ 
form a basis for $\Hom_{\Schur}(\mu,\lambda)$ by Definition~\ref{scat}, and
the corresponding 
morphisms $[A]$ span $\Hom_{\Web}(\mu,\lambda)$ by
Lemma~\ref{straightening},
we deduce that $F$ is full and faithful. Hence, it is an isomorphism.
\end{proof}

From now on, we will {\em identify} the categories $\Web$ and $\Schur$
via the isomorphism $F$ from Theorem~\ref{ckmrevised}. We will refer
to this category as the Schur category rather than the polynomial web
category, and will not use the notation $\Web$ again.

\begin{remark}\label{codet}
The Schur algebra possesses another classical basis, namely, Green's
basis of {codeterminants}; see \cite{Greenco, W}.
Using Remark~\ref{connectr}, it is straightforward to translate 
Green's result to obtain another basis 
for the morphism space
$\Hom_{\Schur}(\mu,\lambda)$, as follows.
Suppose that
$\lambda,\mu \Vdash d$.
For a partition $\kappa \vdash d$,
let $\Std(\lambda,\kappa)$ denote the set of all semistandard
Young tableaux of shape $\kappa$ and 
content $\lambda$, i.e., fillings of the Young
diagram of $\kappa$ with $\lambda_1$ entries equal to 1, $\lambda_2$
entries equal to 2,
\dots, so that the entries are weakly increasing along rows and
strictly decreasing down columns.
Define $\Std(\mu,\kappa)$ similarly.
For $P \in \Std(\lambda,\kappa)$
and $Q \in \Std(\mu,\kappa)$, let
\begin{equation}\label{gct}
\gamma_{P,Q} := \xi_A \circ \xi_B
\end{equation}
where $A \in \Theta_{\lambda,\kappa}$ (resp., $B \in
\Theta_{\kappa,\mu}$)
is defined so that $a_{i,j}$ is the number of entries $i$ in the $j$th
row of $P$ (resp., $b_{i,j}$ is the number of entries $j$ in the $i$th row of $Q$).
Note that a reduced chicken foot diagram of type
$A$ has no merges, while one of type $B$ has no splits.
Consequently, the diagram for
$\gamma_{P,Q}$ can look rather different than a chicken foot diagram: it
has generalized permutations at the top and bottom and
merges and splits in the middle.
The {\em codeterminant basis} for
$\Hom_{\Schur}(\mu,\lambda)$ is
\begin{equation}
\{\gamma_{P,Q} \:|\:d \geq 0, \kappa \vdash d, P \in \Std(\lambda,\kappa), Q \in
\Std(\mu,\kappa)\}.
\end{equation}
This basis is of a similar nature to the basis recently
constructed from a completely different viewpoint by Elias
\cite{Elias}.
It gives $\Schur$ the structure of an object-adapted cellular category
in the sense of \cite[Definition 2.1]{EL}.
\end{remark}

It is time to return to the study of the category $\Tilt{G_n}$ of tilting
modules for $G_n$.
For $\lambda \vDash d$, let
\begin{equation}\label{bwl}
\textstyle\bigwedge^\lambda V_n := 
\bigwedge^{\lambda_1} V_n \otimes\cdots\otimes
\bigwedge^{\lambda_{\ell(\lambda)}} V_n \in \Tilt{G_n}.
\end{equation}
Let $S_\lambda$ denote the standard parabolic subgroup
$S_{\lambda_1}\times\cdots\times S_{\lambda_{\ell(\lambda)}}$ of the
symmetric group $S_d$.
Given also $\mu \vDash d$, let
$(S_\lambda\backslash S_d)_{\min}$ and
$(S_d / S_\mu)_{\min}$ 
be the sets of minimal length
 $S_\lambda\backslash S_d$- and $S_d / S_\mu$-coset
representatives, respectively.
Then
$$
(S_\lambda \backslash S_d / S_\mu)_{\min} := (S_\lambda\backslash
S_d)_{\min} \cap (S_d / S_\mu)_{\min}
$$
is the set of minimal length $S_\lambda\backslash S_d / S_\mu$-double
coset representatives, and 
there is a bijection
\begin{equation}\label{dcosets}
\Theta_{\lambda,\mu} \stackrel{\sim}{\rightarrow} (S_\lambda \backslash S_d / S_\mu)_{\min},
\qquad
A \mapsto d_A.
\end{equation}
To construct $d_A$ from $A$, take a reduced chicken foot diagram of
type $A$; for once, we are not assuming $\lambda$ and $\mu$ are strict here, so $A$ may
have rows or columns of zeros, in which case we mean the same diagram as for the
matrix obtained from $A$ by removing these trivial rows and columns.
Then expand this diagram by
replacing each string of thickness $r$ by $r$ parallel strings of unit
thickness. The desired double coset representative $d_A$ is the
element of $S_d$ defined by the resulting permutation diagram.
For example, for $A$ as in (\ref{firstex}), the diagram expands
as
$$
  
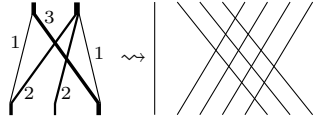
\begin{tikzpicture}[anchorbase,scale=1.5]
\draw[-,line width=.6mm] (.212,.5) to (.212,.39);
\draw[-,line width=.75mm] (.595,.5) to (.595,.39);
\draw[-,line width=.15mm] (0.0005,-.396) to (.2,.4);
\draw[-,line width=.3mm] (0.01,-.4) to (.59,.4);
\draw[-,line width=.3mm] (.4,-.4) to (.607,.4);
\draw[-,line width=.45mm] (.79,-.4) to (.214,.4);
\draw[-,line width=.15mm] (.8035,-.398) to (.614,.4);
\draw[-,line width=.3mm] (.4006,-.5) to (.4006,-.395);
\draw[-,line width=.6mm] (.788,-.5) to (.788,-.395);
\draw[-,line width=.45mm] (0.011,-.5) to (0.011,-.395);
\node at (0.05,0.15) {$\scriptstyle 1$};
\node at (0.79,0.05) {$\scriptstyle 1$};
\node at (0.35,0.37) {$\scriptstyle 3$};
\node at (0.17,-0.3) {$\scriptstyle 2$};
\node at (0.5,-0.3) {$\scriptstyle 2$};
\end{tikzpicture}
\rightsquigarrow
  \begin{tikzpicture}[anchorbase,scale=1.5]
\draw[-] (0,0) to (0,1);
\draw[-] (0.2,0) to (.8,1);
\draw[-] (0.4,0) to (1,1);
\draw[-] (0.6,0) to (1.2,1);
\draw[-] (.8,0) to (1.4,1);
\draw[-] (1,0) to (.2,1);
\draw[-] (1.2,0) to (.4,1);
\draw[-] (1.4,0) to (.6,1);
\draw[-] (1.6,0) to (1.6,1);
\end{tikzpicture}
$$
and $d_A = (2\:5\:8\:4\:7\:3\:6)$.

\begin{lemma}\label{haircut}
Suppose $\lambda,\mu \vDash d$ and $A \in 
\Theta_{\lambda,\mu}$.
We have that $d_A^{-1} S_\lambda d_A \cap S_\mu = S_{\mu^+}$ for some $\mu^+
\vDash d$ (see the discussion after (\ref{dare2}) for an explicit
construction of $\mu^+$).
There is a unique $G_n$-module homomorphism
$\phi_A$ making the diagram
$$
\begin{CD}
V_n^{\otimes d} &@> >>& V_n^{\otimes d}\\
@VVV&&@VVV\\
{\textstyle\bigwedge^\mu} V_n &@>\phi_A >>& {\textstyle\bigwedge^\lambda} V_n
\end{CD}
$$
commute,
where the top map is
the $G_n$-module homomorphism defined by right multiplication by
$\sum_{g \in  (S_\mu / S_{\mu^+})_{\min}} (-1)^{\ell(g d_A^{-1})} g d_A^{-1} \in \k S_d$, and
the vertical maps are the natural quotients.
\end{lemma}

\begin{proof}
The first statement follows from \cite[Lemma 1.6(ii)]{DJ}.
The kernel of the projection $V_n^{\otimes d} \twoheadrightarrow
\bigwedge^\mu V_n$ is spanned by the fixed point sets of the involutions of $V_n^{\otimes d}$
defined by right multiplication by 
all simple reflections $s \in S_\mu$.
Thus, to complete the proof, we need to show for such an $s$ and $v \in V_n^{\otimes d}$ with
$vs = v$
that the vector $$
w := \sum_{g \in  (S_\mu / S_{\mu^+})_{\min}} (-1)^{\ell(g d_A^{-1})} v g d_A^{-1}
$$ 
is in the kernel of the projection $V_n^{\otimes d} \twoheadrightarrow
\bigwedge^\lambda V_n$.
For $g \in (S_\mu / S_{\mu^+})_{\min}$,
we either have that $s g S_{\mu^+} \neq g S_{\mu^+}$, in which case $sg
\in (S_\mu / S_{\mu^+})_{\min}$ too, or $sg S_{\mu^+} = gS_{\mu^+}$, in which
case
$g^{-1} sg \in S_{\mu^+}$; see \cite[Lemma 1.1]{DJ}.
It follows that $(S_\mu / S_{\mu^+})_{\min}$ decomposes as $X \sqcup s X
\sqcup Y$ such that 
$\ell(sx) = \ell(x)+1$ for all $x \in X$, and 
$y^{-1} sy \in S_{\mu^+}$ for all $y \in Y$.
For $x \in X$, we have that
$(-1)^{\ell(x d_A^{-1})} v x d_A^{-1}  + (-1)^{\ell(sx d_A^{-1})} v s x d_A^{-1} = 0$
as $v s = v$.
This implies that
$$
w = \sum_{y \in Y} (-1)^{\ell(y d_A^{-1})} v y d_A^{-1}.
$$
It remains to show for $y \in Y$ that $v y d_A^{-1}$ is in the kernel
of $V_n^{\otimes d} \twoheadrightarrow
\bigwedge^\lambda V_n$.
We have that $s y d_A^{-1} = y d_A^{-1} t$
for $t := d_A (y^{-1} s y) d_A^{-1} \in
S_\lambda$.
By \cite[Lemma 1.6(iv)]{DJ}, $\ell(y d_A^{-1} t) =
\ell(y)+\ell(d_A^{-1})+\ell(t)$. Since $\ell(s y d_A^{-1}) \leq
\ell(y)+\ell(d_A^{-1})+1$,
we deduce that $\ell(t) = 1$. Moreover
$v y d_A^{-1} t = v s y d_A^{-1} = v y d_A^{-1}$.
This shows that $v y d_A^{-1}$ is a fixed point for the simple reflection 
$t \in S_\lambda$, thus, it is in the kernel of the projection.
\end{proof}

\begin{proposition}[Donkin]\label{donuts}
Fix integers $m,d \geq 0$.
For any $n \geq 0$,
there is a surjective algebra homomorphism
\begin{equation}\label{donkiniso}
f_{n}:S(m,d) \twoheadrightarrow \End_{G_n}\left(\bigoplus_{\lambda \in \Lambda(m,d)}
{\textstyle\bigwedge^\lambda} V_n\right)
\end{equation}
sending 
$\xi_A \in 1_\lambda S(m,d) 1_\mu$
to the endomorphism that is equal to the homomorphism
$\phi_A$ from Lemma~\ref{haircut} 
on the summand
$\bigwedge^\mu V_n$,
and is zero on all other summands.
Moreover, $f_n$ is an isomorphism if $n\geq d$.
\end{proposition}

\begin{proof}
This is proved in \cite{D1}, but we need to go through the
argument in detail in order to identify the map $f_{n}$ explicitly.
We just treat the case that $n \geq d$. Then the existence and
surjectivity
of $f_{n}$ for $n < d$ follows from the existence and
surjectivity of $f_{N}$
for $N \geq d$
by an argument involving truncation to the subgroup $G_n < G_N$.
This step
is 
explained in the proof of \cite[Proposition 3.11]{D1}; it depends on
\cite[Proposition 1.5]{D1}, 
hence, on homological properties arising from the 
fact that Schur algebras are quasi-hereditary algebras.

So now assume that $n \geq d$. We must show that $f_{n}$ is a
well-defined algebra isomorphism.
For $\lambda \in \Lambda(m,d)$, 
let $M(\lambda)$ be the right permutation module
$X_\lambda \otimes_{\k S_\lambda} \k S_d$, where
$X_\lambda$ is the trivial one-dimensional right $S_\lambda$-module 
with generator $x_\lambda$.
The module $M(\lambda)$
is isomorphic to the $\lambda$-weight
space $1_\lambda V_m^{\otimes d}$ of $V_m^{\otimes d}$
via the unique $S_d$-module homomorphism sending
$x_\lambda \otimes 1\in M(\lambda)$ to 
$v_1^{\otimes \lambda_1} \otimes \cdots \otimes v_m^{\otimes \lambda_m}$.
By the definition
(\ref{schurdef}) (with $n$ replaced by $m$ of course), we have that
\begin{equation*}
S(m,d) \cong \End_{S_d}\left(\bigoplus_{\lambda \in \Lambda(m,d)}
M(\lambda)\right).
\end{equation*}
Under this isomorphism, $\xi_A \in 1_\lambda S(m,d) 1_\mu$ corresponds
to the unique $S_d$-module
homomorphism 
$M(\mu) \rightarrow M(\lambda)$ sending
$x_\mu\otimes 1$ to $x_\lambda \otimes \sum_{g \in (S_{\mu^+}\backslash S_\mu)_{\min}} d_A g$,
where $S_{\mu^+} = d_A^{-1} S_\lambda d_A \cap S_\mu$ as in Lemma~\ref{haircut}.
This follows from (\ref{cats}),
noting that $S_{\mu^+}=
\Stab_{S_d} (\bi\cdot d_A) \cap \Stab_{S_d}(\bj)$
where $\bi = (1^{\lambda_1},\dots,m^{\lambda_m})$ and $\bj = (1^{\mu_1},\dots,m^{\mu_m})$.

Consider instead the left signed permutation module $N(\lambda):=
\k S_d \otimes_{\k S_\lambda} Y_\lambda$,
where $Y_\lambda$ is the one-dimensional left $S_\lambda$-module
with generator $y_\lambda$
such that $g y_\lambda =
(-1)^{\ell(g)} y_\lambda$ for all $g \in S_\lambda$.
Noting that $N(\lambda)$ is isomorphic to $M(\lambda)$ tensored by 
sign and converted from a right module to a left module using the
antiautomorphism $g \mapsto g^{-1}$,
we deduce from the previous paragraph
that there is an algebra isomorphism
\begin{equation*}
S(m,d) \cong \End_{S_d} \left(\bigoplus_{\lambda \in \Lambda(m,d)}
    N(\lambda)\right).
\end{equation*}
Under this isomorphism, $\xi_A \in 1_\lambda S(m,d) 1_\mu$ corresponds
to the unique $S_d$-module homomorphism
$N(\mu) \rightarrow N(\lambda)$
sending
$1 \otimes y_\mu$ to 
$
\sum_{g \in (S_\mu / S_{\mu^+})_{\min}} (-1)^{\ell(g d_A^{-1})} g d_A^{-1} \otimes y_\lambda$.

Now we are going to apply the Schur functor $\T$ from
(\ref{schurfunctor}).
The key observation is that $\T \left(\bigwedge^\lambda V_n\right) \cong
N(\lambda)$,
there being a unique such isomorphism sending
the canonical image of $v_1 \otimes \cdots \otimes v_d$
in $\bigwedge^\lambda V_n$ to $1 \otimes y_\lambda$.
The head of $\bigwedge^\mu V_n$ and the socle of $\bigwedge^\lambda
V_n$ are $p$-restricted in the sense that they only involve irreducible
modules which are not annihilated by $\T$.
Indeed, these modules are both submodules and quotient modules of the tensor
space $V_n^{\otimes d}$, which has $p$-restricted head and socle by
\cite[Corollary 2.12]{LR}.
Consequently, by \cite[Lemma 2.17(ii)]{LR}
(another well-known property of Schur functors), the Schur functor induces an isomorphism
$$
\Hom_{S(n,d)}(\textstyle\bigwedge^\mu V_n, \bigwedge^\lambda V_n)
\stackrel{\sim}{\rightarrow}
\Hom_{S_d}(N(\mu), N(\lambda));
$$
see also \cite[Lemma 3.6]{D1}.
It follows that $\T$ induces an algebra isomorphism
$$
\End_{G_n}\left(\bigoplus_{\lambda \in \Lambda(m,d)} \textstyle \bigwedge^\lambda V_n\right)
\cong
\End_{S_d}\left(\bigoplus_{\lambda \in \Lambda(m,d)} N(\lambda)\right).
$$
Composing this with the isomorphism in the previous paragraph gives the
desired isomorphism $f_{n}$.

It just remains to identify the endomorphism $f_{n}(\xi_A)$ 
with $\phi_A$.
For this, it suffices to check for $\xi_A \in 1_\lambda S(m,d) 1_\mu$
that the maps 
$f_{n}(\xi_A)$ and $\phi_A$  are equal on the canonical
image of $v_1\otimes\cdots\otimes v_d$ in $\bigwedge^\mu V_n$.
By the definition from Lemma~\ref{haircut}, $\phi_A$ sends this vector to 
the canonical image of 
$$
\sum_{g \in (S_\mu / S_{\mu^+})_{\min}} (-1)^{\ell(g
  d_A^{-1})} (v_1\otimes\cdots\otimes v_d) g d_A^{-1}
$$ 
in $\bigwedge^\lambda V_n$.
On the other hand, by the construction of $f_{n}$
from the previous two paragraphs, $f_{n}(\xi_A)$ takes
this vector to the image of
$$
\sum_{g \in (S_\mu / S_{\mu^+})_{\min}} (-1)^{\ell(g
  d_A^{-1})} g d_A^{-1} (v_1\otimes\cdots\otimes v_d),
$$ 
where $g d_A^{-1} \in S_d$ is being identified with an element of
$1_\omega S(n,d) 1_\omega$ via the isomorphism (\ref{cruzy}).
It remains to observe for $g \in S_d$ that $g (v_1 \otimes\cdots\otimes v_d) = (v_1
\otimes \cdots \otimes v_d) g$. This follows because the isomorphism (\ref{crazy}) 
maps $v_1 \otimes \cdots \otimes v_d$ to $1_\omega$.
\end{proof}

The following theorem gives a reformulation of Proposition~\ref{donuts} from
the perspective of the Schur category.

\begin{theorem}\label{dthm}
There is a full monoidal functor $\Sigma_n:\Schur
\rightarrow \Tilt{G_n}$ sending 
an object $\lambda \Vdash d$ to
$\bigwedge^\lambda V_n \in \Tilt{G_n}$,
and a morphism $\xi_A$ for 
$\lambda,\mu \Vdash d$ and $A \in \Theta_{\lambda,\mu}$
to the homomorphism 
$\phi_A:\bigwedge^\mu V_n \rightarrow\bigwedge^\lambda V_n$ from Lemma~\ref{haircut}.
In particular, $\Sigma_n$ maps
the two-fold merge from (\ref{ms})
to the projection
${\textstyle\bigwedge^a V_n \otimes \bigwedge^b V_n} \twoheadrightarrow {\textstyle\bigwedge^{a+b}
V_n}$,
the two-fold split to the inclusion
\begin{align*}
{\textstyle \bigwedge^{a+b} V_n} &\hookrightarrow
\textstyle{\bigwedge^a V_n \otimes \bigwedge^b V_n},\\
v_{i_1}\wedge\cdots\wedge v_{i_{a+b}}
&\mapsto 
\sum_{g \in (S_{a+b} / S_a \times S_b)_{\min}}
(-1)^{\ell(g)}
v_{i_{g(1)}} \wedge \cdots \wedge v_{i_{g(a)}}\otimes
v_{i_{g(a+1)}} \wedge \cdots \wedge v_{i_{g(a+b)}},
\end{align*}
and the thick crossing 
to
the isomorphism
${\textstyle\bigwedge^a V_n  \otimes \bigwedge^b V_n}
\stackrel{\sim}{\rightarrow} {\textstyle\bigwedge^b V_n \otimes \bigwedge^a V_n},
v \otimes w \mapsto (-1)^{ab} w \otimes v$.
\end{theorem}

\begin{proof}
To see that $\Sigma_n$ is a well-defined functor, we need to show that
$\Sigma_n(\xi_A \circ \xi_B) = \Sigma_n(\xi_A) \circ \Sigma_n(\xi_B)$
for $A \in \Theta_{\lambda,\mu}$ and $B \in
  \Theta_{\mu,\nu}$
for $\lambda,\mu,\nu \Vdash d$ and $d \geq 0$.
By Schur's product rule, $\Sigma_n(\xi_A \circ \xi_B) 
= \sum_{C \in \Theta_{\lambda,\nu}} Z(A,B,C) \Sigma_n(\xi_C)
= \sum_{C \in \Theta_{\lambda,\nu}} Z(A,B,C) \phi_C$.
We need to show this equals
$\phi_A \circ \phi_B$.
This follows from 
Proposition~\ref{donuts} 
and (\ref{connect}) with $n$ replaced by $m \geq d$.
The proposition also shows that $\Sigma_n$ is full.
Finally, to see that $\Sigma_n$ is a monoidal functor,
we need to check that
$\phi_A \otimes \phi_B = \phi_{\diag(A,B)}$.
This is clear from the explicit description of these maps given by
Lemma~\ref{haircut}.
\end{proof}

\begin{remark}\label{tosay}
The functor $\Sigma_n$ in Theorem~\ref{dthm} is certainly not faithful, but
it is {\em asymptotically faithful} in the sense that
it induces an
isomorphism 
\begin{equation}
\Hom_{\Schur}(\mu,\lambda) \stackrel{\sim}{\rightarrow} 
\Hom_{G_n}(\textstyle\bigwedge^\mu V_n, \bigwedge^\lambda V_n)
\end{equation}
for $n$
sufficiently large relative to $\lambda$ and $\mu$. In fact, 
if $\lambda,\mu\Vdash d$ then one just needs that $n \geq d$, as is clear
from the last part of Proposition~\ref{donuts}. Let
\begin{equation}
\Schur_n := \Schur / \mathcal J_n
\end{equation}
where $\mathcal J_n$ is the tensor ideal of $\Schur$ that
is the kernel of $\Sigma_n$.
Then $\Sigma_n$ induces an equivalence of symmetric  monoidal
categories between $\Schur_n$ and the full monoidal subcategory of
$\Tilt{G_n}$ generated by the exterior powers $\bigwedge^a V_n$ for all $a > 0$.
In fact, $\mathcal J_n$ is the tensor ideal of $\Schur$ generated by the morphisms
$1_{(m)}$ for all $m > n$; cf. Remark~\ref{upgrade}.
Together with Theorem~\ref{ckmrevised}, this
identifies $\Schur_n$ with the
polynomial web category for $GL_n$ from \cite[$\S$5]{CKM}.
This can be seen from \cite{CKM}, but also it
can be proved quite easily using the codeterminant basis from
Remark~\ref{codet}, as follows.

Note first that the tensor ideal $\mathcal
K_n$ of 
$\Schur$ generated by the morphisms
$1_{(m)}$ for all $m > n$ 
is contained in $\mathcal J_n$
as $\bigwedge^m V_n = 0$ for $m> n$.
Now take $\lambda,\mu\Vdash d$.
The codeterminants $\gamma_{P,Q}$
for $\kappa \vdash d$ with $\kappa_1 > n$,
$P \in \Std(\lambda,\kappa)$ and $Q \in \Std(\mu,\kappa)$
belong to $\mathcal K_n(\mu,\lambda)$ since their diagrams involve a
string of thickness $\kappa_1$.
Hence, $\Hom_{\Schur}(\mu,\lambda) / \mathcal K_n(\mu,\lambda)$
is spanned by all $\gamma_{P,Q}$
for $\kappa \vdash d$ with $\kappa_1 \leq n$,
$P \in \Std(\lambda,\kappa)$ and $Q \in \Std(\mu,\kappa)$.
In fact, we have that $\mathcal K_n(\mu,\lambda) = \mathcal
J_n(\mu,\lambda)$ (proving the assertion), 
and these codeterminants with $\kappa_1 \leq n$ give a basis for
$\Hom_{\Schur_n}(\mu,\lambda)
\cong \Hom_{G_n}(\bigwedge^\mu V_n, \bigwedge^\lambda V_n)$. This follows because
$$
\dim \Hom_{G_n}({\textstyle\bigwedge^\mu} V_n, {\textstyle\bigwedge^\lambda} V_n)
= \#\left\{(\kappa,P,Q)\:\bigg|\:\begin{array}{l}
\kappa\vdash d\text{ with }\kappa_1 \leq n,\\
P \in \Std(\lambda,\kappa), Q \in \Std(\mu,\kappa)
\end{array}
\right\}.
$$
(Proof: For $\kappa \vdash d$ with $\kappa_1 \leq n$, let $\kappa^T$ be
the transpose partition viewed as a weight in $X_n^+$.
By the Littlewood-Richardson rule and character considerations, the
tilting module 
$\bigwedge^\mu V_n$ has a $\Delta$-flag with sections
$\Delta_n(\kappa^T)$ for all such $\kappa$, each appearing with multiplicity
$\#\Std(\mu,\kappa)$.
Similarly
$\bigwedge^\lambda V_n$ has a $\nabla$-flag with sections
$\nabla_n(\kappa^T)$, each appearing with multiplicity
$\#\Std(\lambda,\kappa)$.
Now use $\dim \Ext^i_{G_n}(\Delta_n(\sigma),\nabla_n(\tau)) =
\delta_{\sigma,\tau} \delta_{i,0}$.)
\end{remark}

At last, all of the background is in place, and we can achieve the
main goal of the section.
The composition of the functor $\Sigma_n$ from Theorem~\ref{dthm} with the
quotient functor $Q:\Tilt{G_n} \rightarrow \TILT{G_n}$ gives us a full
monoidal functor
\begin{equation}
\widetilde{\Sigma}_n:\Schur \rightarrow \TILT{G_n}.
\end{equation}
We just need one more elementary observation.

\begin{lemma}\label{gray}
Suppose that $p > 0$ and $a$, $b$ are positive integers summing
to $p^m$.
The images under $\widetilde{\Sigma}_n$ of the two-fold merge and split
morphisms from (\ref{ms})
are both zero.
\end{lemma}

\begin{proof}
By weight considerations, 
$\Hom_{G_n}\left(\bigwedge^{p^m} V_n, \bigwedge^a V_n \otimes
\bigwedge^b V_n\right)$ is of dimension one with basis given by the two-fold
split.
So $\Hom_{\TILT{G_n}}\left(\bigwedge^{p^m} V_n, \bigwedge^a V_n \otimes
\bigwedge^b V_n\right)$ is spanned by the
image $f$ of the two-fold split.
Similarly,
the image $g$ of the two-fold merge spans
$\Hom_{\TILT{G_n}}\left(\bigwedge^a V_n \otimes
\bigwedge^b V_n, \bigwedge^{p^m} V_n\right)$.
By semisimplicity, if one of these morphisms is non-zero, so is the
other, and $g \circ f$
is an automorphism of $\bigwedge^{p^m} V_n$.
But this composition is zero by (\ref{trivial}).
\end{proof}

\begin{theorem}\label{fullity}
The functor $\widetilde{\Phi}_n:\Kar(\OB(t_0,\dots,t_r)) \rightarrow
\TILT{G_n}$ 
from (\ref{haveit}) is full.
\end{theorem}

\begin{proof}
Let $X$ and $Y$ be objects of $\OB(t_0,\dots,t_r)$,
so they are both words in the symbols
$\up_i$'s and $\down_i$'s for $i=0,\dots,r$.
Their images $\overline{X}$ and $\overline{Y}$
under the functor
$\widetilde{\Phi}_n$ are corresponding tensor products of
the modules $\bigwedge^{p^{i}} V_n$ and $\bigwedge^{p^{i}} V_n^*$,
notation as in (\ref{padic}).
We need to show
that the linear map
$$
\Hom_{\OB(t_0,\dots,t_r)}(X, Y) \rightarrow
\Hom_{\TILT{G_n}}(\overline{X}, \overline{Y})
$$
defined by the functor $\widetilde{\Phi}_n$ is surjective.
Since this is a symmetric monoidal functor, we may assume that all of
the $\down_i$'s in $X$ appear at the
beginning of this word.
Then using duality we can transfer them from the beginning of $X$ to 
$\up_i$'s appearing at the beginning of $Y$.
Thus we are reduced to the case that $X$ only involves $\up_i$'s.
Repeating the argument for $Y$, we reduce further to the case that $Y$
only involves $\up_i$'s too.

So now $X$ and $Y$ are words just in the symbols $\up_i$ for $i=0,\dots,r$,
and $\overline{X}$ and $\overline{Y}$ are corresponding tensor
products of the modules $\bigwedge^{p^i} V_n$, i.e.,
we have that $\overline{X} = \bigwedge^\mu V_n$ and $\overline{Y} =
\bigwedge^\lambda V_n$ for strict compositions $\lambda,\mu$ all of
whose parts are of the form $p^{i}$ for $i=0,\dots,r$.
Since the functor $\widetilde{\Sigma}_n$
is full, it follows that
$\Hom_{\TILT{G_n}}(\overline{X}, \overline{Y})$ is spanned by the
images of the morphisms $\xi_A$ for $A \in \Theta_{\lambda,\mu}$.
In view of Lemma~\ref{gray}, these images are zero unless $A = A^\circ$,
i.e., $\xi_A$ is a generalized permutation. As 
generalized permutations are generated by thick crossings,
and $\Sigma_n$ maps thick crossings to tensor flips (up to a
sign) 
according to Theorem~\ref{dthm},
it remains to observe that the 
tensor flip 
$$\textstyle
\bigwedge^{p^{i}} V_n \otimes \bigwedge^{p^{j}} V_n
\rightarrow 
\bigwedge^{p^{j}} V_n \otimes \bigwedge^{p^{i}} V_n,
\qquad
v \otimes w \mapsto w \otimes v
$$
is the image under
$\Phi_n$ of the crossing 
in $\OB(t_0,\dots,t_r)$ 
of strings of color $i$ and $j$.
\end{proof}

\section{Identification of labelings}\label{comb}

Let notation be as in (\ref{padic}),
and recall (\ref{interests})--(\ref{off}).
We have now proved the existence of a symmetric monoidal equivalence
\begin{equation}\label{absurd}
\Xi_n: \TILT{G_{n_0}} \boxtimes\cdots\boxtimes
\TILT{G_{n_r}}
\rightarrow \TILT{G_n}
\end{equation}
sending $V_{n_i} \in \TILT{G_{n_i}}$ to $\bigwedge^{p^{i}} V_n \in \TILT{G_n}$
for $i=0,\dots,r$.
To complete the proof of the Main Theorem, it remains to show that
$\Xi_n$ sends
$T_{n_0}(\lambda^{(0)}) \boxtimes\cdots \boxtimes
T_{n_r}(\lambda^{(r)})$ to $T_n(\imath(\underline{\lambda}))$
for
$\underline{\lambda} = (\lambda^{(0)},\dots,\lambda^{(r)}) \in
X_{n_0,p}^+\times\cdots\times X_{n_r,p}^+$.

Let $\Lambda_n^+ \subset X_n^+$ 
denote the set of polynomial dominant weights,
i.e., the weights $\lambda \in \Z^n$
such that $\lambda_1 \geq \cdots \geq \lambda_n \geq 0$.
Let $\Lambda_{n,p}^+ := \Lambda_n^+ \cap X_{n,p}^+$.
Let $\varpi_i = (1^i, 0^{n-i})$ 
be the highest weight of $\bigwedge^i V_n$ and $\det_n := \bigwedge^n V_n$ 
be the determinant representation.

\begin{lemma}\label{starters}
Given $0 \leq k \leq n$, we have that
$\binom{n}{k} \not\equiv 0\pmod{p}$ if and only if 
\begin{equation}\label{kform}
k = k_0+k_1p+\cdots + k_r p^{r}\quad\text{with}\quad 0 \leq
k_i \leq n_i
\text{ for all $i=0,\dots,r$}.
\end{equation}
Assuming this is the case,
the function $\imath$ takes
$(\varpi_{k_0},\dots,\varpi_{k_r}) \in X_{n_0}^+\times\cdots\times
X_{n_r}^+$
to $\varpi_k \in X_n^+$. Also the equivalence $\Xi_n$ sends
$\bigwedge^{k_0} V_{n_0}\boxtimes\cdots\boxtimes \bigwedge^{k_r}
V_{n_r}$ to a copy of $\bigwedge^k V_n$.
\end{lemma}

\begin{proof}
The first statement follows from Lucas' theorem (\ref{classicallucas}).
Also it is easy to see that
$\imath((\varpi_{k_0},\dots,\varpi_{k_r})) = \varpi_k$ 
just 
using the combinatorial definition of $\imath$.
For the final assertion, note that $k_i < p$, so 
$\bigwedge^{k_i} V_{n_i}$ is the summand 
 of $V_{n_i}^{\otimes k_i}$
defined by the idempotent $\frac{1}{k_i!} \sum_{g \in S_{k_i}}
(-1)^{\ell(g)} g \in \k S_{k_i} = \End_{G_{n_i}}(V_{n_i}^{\otimes k_i})$.
So by the definition of the functor, 
$\Xi_n$ takes 
$\bigwedge^{k_0} V_{n_0}\boxtimes\cdots\boxtimes \bigwedge^{k_r}
V_{n_r} \in \TILT{G_{n_0}}\boxtimes\cdots\boxtimes \TILT{G_{n_r}}$ 
to the tensor product
$W_0 \otimes \cdots \otimes W_r \in \TILT{G_n}$
where $W_i = \bigwedge^{k_i}\left(\bigwedge^{p^{i}} V_n\right)$ 
is the summand of $\left(\bigwedge^{p^{i}}V_n\right)^{\otimes k_i}$
defined by 
$\frac{1}{k_i!} \sum_{g \in S_{k_i}}
(-1)^{\ell(g)} g$
viewed
as an idempotent endomorphism of
this new tensor space in the obvious way.
Equivalently, in the setup of Theorem~\ref{dthm}, this idempotent endomorphism is $f_i := \frac{1}{k_i!} \sum_{g \in S_{k_i}} \Sigma_n\left(g 1_{(p^i,\dots,p^i)}\right)$,
where
$g 1_{(p^i,\dots,p^i)}$ is the diagram of the permutation $g$ drawn with $k_i$ strings each of thickness $p^i$.
In particular, since $\Xi_n$ is an equivalence, this shows that 
$W_0 \otimes\cdots\otimes W_r$ is an irreducible object
in $\TILT{G_n}$.
It remains to observe that $W_0 \otimes\cdots\otimes W_r \cong
\bigwedge^k V_n$ in $\TILT{G_n}$.
This follows because $\bigwedge^k V_n$
is not zero in $\TILT{G_n}$ as it is of non-zero 
categorical dimension,
and it is a summand of $W_0 \otimes\cdots\otimes W_r$
in $\Tilt{G_n}$.
To prove the last assertion,
we already showed at the end of the proof of Lemma~\ref{density}
that $\bigwedge^k V_n$ is a summand of
$\bigwedge^{k_0} V_n
\otimes \bigwedge^{k_1 p} V_n
\otimes\cdots\otimes
\bigwedge^{k_r p^r} V_n$,
so it suffices to prove that
$\bigwedge^{k_i p^i} V_n$ is a summand of $W_i$
for each $i=0,1,\dots,r$.
Both of 
$\bigwedge^{k_i p^i} V_n$ and $W_i$
are summands of $\left(\bigwedge^{p^i} V_n\right)^{\otimes k_i}$ defined by the images of the idempotent $e_i$ from the proof of Lemma~\ref{density} 
and the idempotent $f_i$ here, respectively.
It remains to observe that $e_i f_i = e_i$
by (\ref{swallows}).
\end{proof}

\begin{corollary}\label{see}
The equivalence $\Xi_n$ sends
$\det_{n_0} \boxtimes \cdots \boxtimes \det_{n_r}$ to $\det_n$.
\end{corollary}

Using Corollary~\ref{see}, the
problem in hand reduces easily to the case of
polynomial weights. 
To analyze polynomial weights, we need one more observation.
For $\lambda \in \Lambda_n^+$, let $\lambda^T$ be the usual transpose
partition. 
By the definitions, a weight $\lambda \in \Lambda_n^+$
belongs to $\Lambda_{n,p}^+$ if and
only if
\begin{equation}\label{shower}
\lambda^T = (\lambda^{(0)})^T+p (\lambda^{(1)})^T+ \cdots + p^{r}
(\lambda^{(r)})^T\quad
\text{with}\quad
\lambda^{(i)} \in \Lambda_{n_i,p}^+\text{ for }i=0,\dots,r.
\end{equation}
Choose $m \geq \lambda_1$. Then we can view 
all of the partitions in the decomposition (\ref{shower})
as elements of $\Lambda_m^+$.
Recall that $\lambda \in \Lambda_m^+$ is {\em $p$-restricted} if
$\lambda_i - \lambda_{i+1} < p$ for each $i=1,\dots,m-1$.
Since $n_i < p$ for each $i$, the weight $(\lambda^{(i)})^T$ 
has first part that is smaller than $p$, so it is certainly
$p$-restricted.
We deduce by the Steinberg tensor product theorem that
\begin{equation}\label{soap}
L_m(\lambda^T) 
\cong 
L_m((\lambda^{(0)})^T)
\otimes L_m((\lambda^{(1)})^T)^{[1]}
\otimes\cdots\otimes
L_m((\lambda^{(r)})^T)^{[r]},
\end{equation}
where $[k]$ denotes the $k$th Frobenius twist.
This observation will be used in the proof of the next result.

\begin{theorem}\label{hard}
For
 $\lambda \in \Lambda_n^+ \setminus \Lambda_{n,p}^+$, we have that
$\dim T_n(\lambda) \equiv 0\pmod{p}$.
If $\lambda \in \Lambda_{n,p}^+$, so that it is the image under $\imath$ of
some $(\lambda^{(0)},\dots,\lambda^{(r)}) \in
\Lambda_{n_0,p}^+\times\cdots\times \Lambda_{n_r,p}^+$, we have that
$T_n(\lambda)\cong
\Xi_n
\big(T_{n_0}(\lambda^{(0)}) \boxtimes\cdots \boxtimes
T_{n_r}(\lambda^{(r)})\big)$ in $\TILT{G_n}$
\end{theorem}

\begin{proof}
We proceed by induction on the lexicographic ordering on $\Lambda_n^+$.
The base case $\lambda = 0$ is trivial as
$\Xi_n$ sends $\mathbbm{1}$ to $\mathbbm{1}$.
For the induction step, take $0 \neq \lambda \in \Lambda_n^+$
and pick $m \geq \lambda_1$.
Let $\mu \in \Lambda_n^+$ be obtained by removing some column of height
$0 < k \leq n$ from the Young diagram of $\lambda$, i.e., $\lambda =
\mu + \varpi_k$.
Then $T_n(\lambda)$ is a summand of $\bigwedge^k V_n \otimes T_n(\mu)$.
If $\lambda \in \Lambda_{n,p}^+$ then 
$\mu \in \Lambda_{n,p}^+$ 
too and $k$ is of the form (\ref{kform}). This follows from the
combinatorial definition of the function $\imath$.
By induction, $\mu \in \Lambda_{n,p}^+$ if and only if $\dim T_n(\mu)
\not\equiv 0 \pmod{p}$.

Suppose that
$\dim \bigwedge^k V_n \otimes T_n(\mu) \equiv 0 \pmod{p}$.
Then $\bigwedge^k V_n \otimes T_n(\mu)$ is zero in $\TILT{G_n}$ (as one
of the tensor factors is negligible),
hence, so is its summand $T_n(\lambda)$.
Thus, $\dim T_n(\lambda) \equiv 0 \pmod{p}$.
Using Lemma~\ref{starters} and the observations made at the end of previous paragraph,
we also have that $\lambda \notin \Lambda_{n,p}^+$ in this situation,
so this is consistent with what we are trying to prove.

Now suppose that 
$\dim \bigwedge^k V_n \otimes T_n(\mu) \not\equiv 0 \pmod{p}$.
Then we can write $k$ 
as 
$k_0+k_1p+\cdots+k_r p^{r}$
as in (\ref{kform})
and $\mu^T$ as
$(\mu^{(0)})^T+\cdots+p^{r} (\mu^{(r)})^T$
as in (\ref{shower}).
Note also that $\mu_1 \leq m-1$, so that we can view $\mu^T$
and all $(\mu^{(i)})^T$ here as elements of $\Lambda_{m-1}^+$.
By \cite[Theorem B(ii)]{LR}, we have that
$$
{\textstyle\bigwedge^k V_n} \otimes T_n(\mu)
\cong T_n(\lambda) \oplus \bigoplus_{\lambda \rhd \nu \in \Lambda^+_n}
T_n(\nu)^{\oplus \left[L_m(\nu^T)_k:L_{m-1}(\mu^T)\right]},
$$
where for a $G_m$-module $M$ we write $M_k$ for the sum of its weight
spaces for
all weights with $m$th coordinate equal to $k$, viewing this as
a module over the naturally embedded subgroup $G_{m-1}$.
By induction, $T_n(\nu)$ is zero in $\TILT{G_n}$ unless $\nu \in
\Lambda_{n,p}^+$.
So we deduce in $\TILT{G_n}$ that
\begin{equation}\label{toast}
{\textstyle\bigwedge^k V_n} \otimes T_n(\mu)
\cong T_n(\lambda) \oplus \bigoplus_{\lambda \rhd \nu \in \Lambda^+_{n,p}}
T_n(\nu)^{\oplus \big[L_m(\nu^T)_k:L_{m-1}(\mu^T)\big]}.
\end{equation}
Each $\nu$ here can be decomposed as 
$(\nu^{(0)})^T+\cdots+p^{r} (\nu^{(r)})^T$
according to (\ref{shower}), and then we can use the Steinberg
decomposition (\ref{soap}) to see that
\begin{equation}\label{beach}
\big[L_m(\nu^T)_k:L_{m-1}(\mu^T)\big]
=
\prod_{i=0}^r \big[L_m((\nu^{(i)})^T)_{k_i}:L_{m-1}((\mu^{(i)})^T)\big].
\end{equation}
Now we apply \cite[Theorem B(ii)]{LR} again
to see that
$$
{\textstyle\bigwedge^{k_i} V_{n_i}} \otimes T_{n_i}(\mu^{(i)})
\cong T_{n_i}(\lambda^{(i)}) \oplus \bigoplus_{\lambda^{(i)} \rhd \nu^{(i)} \in \Lambda^+_{n_i,p}}
T_{n_i}(\nu^{(i)})^{\oplus [L_m((\nu^{(i)})^T)_k:L_{m-1}((\mu^{(i)})^T)]}
$$
 in 
$\TILT{G_{n_i}}$,
where $\lambda^{(i)} := \mu^{(i)} + \varpi_{k_i} \in \Lambda_{n_i}^+$,
i.e., its Young diagram is obtained from the one for $\mu^{(i)}$ by
adding a column of height $k_i$
(we do not claim here that $\lambda^{(i)} \in \Lambda_{n_i,p}^+$
necessarily).
We deduce from this isomorphism for all $i=0,\dots,r$ plus
(\ref{beach}) 
that 
\begin{multline}
\left({\textstyle\bigwedge^{k_0} V_{n_0}\otimes 
T_{n_0}(\mu^{(0)})}\right)\boxtimes \cdots \boxtimes
\left({\textstyle\bigwedge^{k_r} V_{n_r}
\otimes T_{n_r}(\mu^{(r)})}
\right)
\cong
T_{n_0}(\lambda^{(0)}) \boxtimes\cdots\boxtimes T_{n_r}(\lambda^{(r)})
\\\oplus \bigoplus_{\mu \rhd \nu \in \Lambda_{n,p}^+}
\left(T_{n_0}(\nu^{(0)}) 
\boxtimes \cdots \boxtimes T_{n_r}(\nu^{(r)})\right)^{\oplus \big[L_m(\nu^T)_k:L_{m-1}(\mu^T)\big]}
\label{quick}\end{multline}
in $\TILT{G_{n_0}}\boxtimes\cdots\boxtimes \TILT{G_{n_r}}$,
for $\nu^{(i)}$ defined from
$\nu^T =  (\nu^{(0)})^T+p (\nu^{(1)})^T+\cdots+p^{r} (\nu^{(r)})^T$ again.
Now we apply the monoidal functor $\Xi_n$ to (\ref{quick}) using
Lemma~\ref{starters} and the induction hypothesis. Comparing the result with
(\ref{toast}) and using semisimplicity
shows finally 
that
$$
\Xi_n\left(T_{n_0}(\lambda^{(0)}) \boxtimes\cdots\boxtimes T_{n_r}(\lambda^{(r)})
\right)
\cong
T_n(\lambda)
$$
in $\TILT{G_n}$.
In particular, $\dim T_n(\lambda) \equiv 0\pmod{p}$ unless
$\lambda^{(i)} \in \Lambda_{n_i,p}^+$ for all $i=0,\dots,r$.
Since $\lambda$ is $\mu$ with a column of height
$k$ added and $\lambda^{(i)}$ is $\mu^{(i)}$ with a column of height
$k_i$ added, 
the weight $\lambda$ is the image of
$(\lambda^{(0)},\dots,\lambda^{(r)})$ under $\imath$. The induction
step now follows from this isomorphism.
\end{proof}

Theorem~\ref{hard} and Corollary~\ref{see} together complete the proof of the
Main Theorem.

\appendix
\section{Relations}


\def\theequation{A.\arabic{equation}}

In this appendix, we prove the relations formulated in $\S$\ref{fullness}.

To start with, we explain how to deduce (\ref{mergesplit}) from the relations
(\ref{assrel})--(\ref{trivial}) and the square switch relations
(\ref{sl2rel})--(\ref{sl2rel2}), interpreting thick crossings as the morphisms defined by (\ref{askthink}). 
Note for this that, in the presence of the square switch relations,
the definition (\ref{askthink}) is equivalent to 
\begin{equation}
\begin{tikzpicture}[anchorbase,scale=1]
	\draw[-,line width=1.2pt] (0,0) to (.6,1);
	\draw[-,line width=1.2pt] (0,1) to (.6,0);
        \node at (0,-.1) {$\scriptstyle a$};
        \node at (0.6,-0.1) {$\scriptstyle b$};
\end{tikzpicture}
:=\sum_{t=0}^{\min(a,b)}
(-1)^t
\begin{tikzpicture}[anchorbase,scale=1]
	\draw[-,thick] (0,0) to (0,1);
	\draw[-,thick] (-.015,0) to (-0.015,.2) to (-.57,.4) to (-.57,.6)
        to (-.015,.8) to (-.015,1);
	\draw[-,line width=1.2pt] (-0.6,0) to (-0.6,1);
        \node at (-0.6,-.1) {$\scriptstyle a$};
        \node at (0,-.1) {$\scriptstyle b$};
        \node at (0.13,.5) {$\scriptstyle t$};
\end{tikzpicture}\label{askthink2}.
\end{equation}
This is an easy exercise.
Now let notation be as in (\ref{mergesplit}) and set $r := d-a$. We just treat the case $r \geq 0$; the other case $r
\leq 0$ then follows by reflecting in a vertical axis and using (\ref{askthink2}). 
We must prove that
\begin{equation}\label{want}
\begin{tikzpicture}[anchorbase,scale=1]
	\draw[-,line width=1.2pt] (0,0) to (.285,.3) to (.285,.7) to (0,1);
	\draw[-,line width=1.8pt] (.61,0) to (.325,.3) to (.325,.7) to (.61,1);
        \node at (0,1.13) {$\scriptstyle b$};
        \node at (0.6,1.13) {$\scriptstyle a+r$};
        \node at (0,-.1) {$\scriptstyle a$};
        \node at (0.6,-.1) {$\scriptstyle b+r$};
\end{tikzpicture}
=
\sum_{s=0}^{\min(a,b)}
\begin{tikzpicture}[anchorbase,scale=1]
	\draw[-,thick] (0.58,0) to (0.58,.2) to (.02,.8) to (.02,1);
	\draw[-,thick] (0.02,0) to (0.02,.2) to (.58,.8) to (.58,1);
	\draw[-,thin] (0,0) to (0,1);
	\draw[-,line width=1pt] (0.61,0) to (0.61,1);
        \node at (0,1.13) {$\scriptstyle b$};
        \node at (0.6,1.13) {$\scriptstyle a+r$};
        \node at (0,-.1) {$\scriptstyle a$};
        \node at (0.6,-.1) {$\scriptstyle b+r$};
        \node at (-0.1,.5) {$\scriptstyle s$};
        \node at (0.92,.5) {$\scriptstyle r+s$};
\end{tikzpicture},
\end{equation}
We first substitute the definition (\ref{askthink}) into the right hand
side of (\ref{want}), using (\ref{assrel})--(\ref{trivial}),  to get
\begin{equation}
\sum_{s=0}^{\min(a,b)}
\sum_{t=0}^{\min(a,b)-s}
(-1)^t
\begin{tikzpicture}[anchorbase,scale=1]
	\draw[-,thin] (0,0) to (0,1);
	\draw[-,thick] (0.02,0) to (0.02,.2) to (.7,.3) to (.7,.7) to (1.08,.8) to (1.08,1);
	\draw[-,thick] (1.08,0) to (1.08,.2) to (.7,.3) to (.7,.7) to (.02,.8) to (.02,1);
	\draw[-,thin] (.01,0) to (0.01,.22) to (.35,.27) to (.35,.73) to
        (0.01,.78) to (.01,1);
	\draw[-,thin] (1.1,0) to (1.1,1);
        \node at (0,1.13) {$\scriptstyle b$};
        \node at (1.13,1.13) {$\scriptstyle a+r$};
        \node at (0,-.1) {$\scriptstyle a$};
        \node at (1.13,-.1) {$\scriptstyle b+r$};
        \node at (-0.1,.5) {$\scriptstyle s$};
        \node at (1.4,.5) {$\scriptstyle r+s$};
        \node at (0.25,.5) {$\scriptstyle t$};
\end{tikzpicture}
=
\sum_{s=0}^{\min(a,b)}
\sum_{u=s}^{\min(a,b)}
(-1)^{s+u}
\binom{u}{s}
\begin{tikzpicture}[anchorbase,scale=1]
	\draw[-,thin] (0,0) to (0,1);
	\draw[-,thick] (0.02,0) to (0.02,.2) to (.56,.35) to (.56,.65) to (1.08,.8) to (1.08,1);
	\draw[-,thick] (1.08,0) to (1.08,.2) to (.56,.35) to (.56,.65) to (.02,.8) to (.02,1);
	\draw[-,thin] (1.1,0) to (1.1,1);
        \node at (0,1.13) {$\scriptstyle b$};
        \node at (1.13,1.13) {$\scriptstyle a+r$};
        \node at (0,-.1) {$\scriptstyle a$};
        \node at (1.13,-.1) {$\scriptstyle b+r$};
        \node at (-0.11,.5) {$\scriptstyle u$};
        \node at (1.4,.5) {$\scriptstyle r+s$};
\end{tikzpicture}.\label{ohgood}
\end{equation}
Then we square switch
to see that this equals
$$
\sum_{s=0}^{\min(a,b)}
\sum_{u=s}^{\min(a,b)}
\sum_{t=u-s}^{\min(a-s,b-s)}
(-1)^{s+u}
\binom{u}{s}
\binom{u+r}{t}
\begin{tikzpicture}[anchorbase,scale=1]
	\draw[-,thin] (0,0) to (0,1);
	\draw[-,thick] (0.02,0) to (0.02,.2) to (.86,.35) to
        (.86,.65) to (.02,.8) to (.02,1);
        \draw[-,line width=.5pt] (.86,.655) to (1.48,.575) to (1.48,1);
	\draw[-,line width=.5pt] (1.48,0) to (1.48,.425) to (.86,.345);
	\draw[-,thin] (1.485,0) to (1.485,1);
        \node at (0,1.13) {$\scriptstyle b$};
        \node at (1.13,1.13) {$\scriptstyle a+r$};
        \node at (0,-.1) {$\scriptstyle a$};
        \node at (1.13,-.1) {$\scriptstyle b+r$};
        \node at (-0.11,.5) {$\scriptstyle u$};
        \node at (.5,.5) {$\scriptstyle s+\!t\!-\!u$};
\end{tikzpicture}\:.
$$
Using (\ref{assrel})--(\ref{trivial}) again, 
this simplifies to
\begin{equation*}
\sum_{s=0}^{\min(a,b)}
\sum_{u=s}^{\min(a,b)}
\sum_{v=u}^{\min(a,b)}
(-1)^{s+u}
\binom{u}{s}
\binom{u+r}{v-s}
\binom{v}{u}
\begin{tikzpicture}[anchorbase,scale=1]
	\draw[-,thin] (0,0) to (0,1);
	\draw[-,thick] (0.02,0) to (0.02,.2) to (.66,.35) to
        (.66,.65) to (.02,.8) to (.02,1);
	\draw[-,thin] (.67,0) to (.67,1);
        \node at (0,1.13) {$\scriptstyle b$};
        \node at (.67,1.13) {$\scriptstyle a+r$};
        \node at (0,-.1) {$\scriptstyle a$};
        \node at (.67,-.1) {$\scriptstyle b+r$};
        \node at (-0.11,.5) {$\scriptstyle v$};
\end{tikzpicture}.
\end{equation*}
Next, switch the orders of the summations 
to get
$$
\sum_{v=0}^{\min(a,b)}
(-1)^v
\sum_{s=0}^v
(-1)^{s}
\left(\sum_{u=s}^{v}
 (-1)^{u-v}
\binom{u}{s}
\binom{v}{u}
\binom{u+r}{v-s}
\right) \begin{tikzpicture}[anchorbase,scale=1]
	\draw[-,thin] (0,0) to (0,1);
	\draw[-,thick] (0.02,0) to (0.02,.2) to (.66,.35) to
        (.66,.65) to (.02,.8) to (.02,1);
	\draw[-,thin] (.67,0) to (.67,1);
        \node at (0,1.13) {$\scriptstyle b$};
        \node at (.67,1.13) {$\scriptstyle a+r$};
        \node at (0,-.1) {$\scriptstyle a$};
        \node at (.67,-.1) {$\scriptstyle b+r$};
        \node at (-0.11,.5) {$\scriptstyle v$};
\end{tikzpicture}.
$$
The term in parentheses is equal to $\binom{v}{s}$; to see this, take
the identity from Lemma~\ref{A}, replace $m,n,r$ and $s$ with
$s+r, v-s, u-s$ and $v-u$, respectively, then multiply both sides by $\binom{v}{s}$.
Hence, we have
$$
\sum_{v=0}^{\min(a,b)}(-1)^v 
\left(\sum_{s=0}^v
(-1)^{s} \binom{v}{s}
\right)
\begin{tikzpicture}[anchorbase,scale=1]
	\draw[-,thin] (0,0) to (0,1);
	\draw[-,thick] (0.02,0) to (0.02,.2) to (.66,.35) to
        (.66,.65) to (.02,.8) to (.02,1);
	\draw[-,thin] (.67,0) to (.67,1);
        \node at (0,1.13) {$\scriptstyle b$};
        \node at (.67,1.13) {$\scriptstyle a+r$};
        \node at (0,-.1) {$\scriptstyle a$};
        \node at (.67,-.1) {$\scriptstyle b+r$};
        \node at (-0.11,.5) {$\scriptstyle v$};
\end{tikzpicture}
=
\sum_{v=0}^{\min(a,b)}(-1)^v 
\delta_{v,0}
\begin{tikzpicture}[anchorbase,scale=1]
	\draw[-,thin] (0,0) to (0,1);
	\draw[-,thick] (0.02,0) to (0.02,.2) to (.66,.35) to
        (.66,.65) to (.02,.8) to (.02,1);
	\draw[-,thin] (.67,0) to (.67,1);
        \node at (0,1.13) {$\scriptstyle b$};
        \node at (.67,1.13) {$\scriptstyle a+r$};
        \node at (0,-.1) {$\scriptstyle a$};
        \node at (.67,-.1) {$\scriptstyle b+r$};
        \node at (-0.11,.5) {$\scriptstyle v$};
\end{tikzpicture}
=
\begin{tikzpicture}[anchorbase,scale=1]
	\draw[-,thick] (0.02,0) to (0.02,.2) to (.66,.35) to
        (.66,.65) to (.02,.8) to (.02,1);
	\draw[-,thin] (.67,0) to (.67,1);
        \node at (0,1.13) {$\scriptstyle b$};
        \node at (.67,1.13) {$\scriptstyle a+r$};
        \node at (0,-.1) {$\scriptstyle a$};
        \node at (.67,-.1) {$\scriptstyle b+r$};
\end{tikzpicture},
$$
which is the left hand side of (\ref{want}).

\begin{lemma}\label{A}
Let $\binom{m}{r,s}$ be the trinomial coefficient 
$m(m-1)\cdots (m-r-s+1) / r!s!$ (interpreted as $0$ if $r < 0$ or $s
< 0$).
For $m \in \Z$ and $n \geq 0$, we have that
$$
\sum_{r+s=n} (-1)^s
\binom{m+r}{r,s} = 1.
$$
\end{lemma}

\begin{proof}
Use the recurrence relation
$\binom{m}{r,s} = \binom{m-1}{r,s} + \binom{m-1}{r-1,s} +
\binom{m-1}{r,s-1}$
and induction on $n$ to show that 
$$
\sum_{r+s=n} (-1)^s
\binom{m+r}{r,s} = 
\sum_{r+s=n} (-1)^s
\binom{m-1+r}{r,s}.
$$
Hence, we may assume that $m=0$, when the identity is clear.
\end{proof}

In the remainder of the appendix, we work in the category $\Web$
as defined in Definition~\ref{webcat}, so have the defining relations
(\ref{assrel})--(\ref{mergesplit}), and will prove the relations 
(\ref{jonsquare})--(\ref{braid}).
In particular, this shows that the relations
(\ref{assrel})--(\ref{mergesplit})
imply the square switch relations, justifying the equivalence of
presentations asserted in
Remark~\ref{mp}.

\vspace{2mm}
\noindent
{\em Proof of (\ref{jonsquare}).}
Note $a \geq d$.
To prove the first equality, we expand the left hand side 
as a sum of diagrams involving a crossing
using (\ref{mergesplit}), to see that
\begin{align*}
\begin{tikzpicture}[anchorbase,scale=1]
	\draw[-,thick] (0,0) to (0,1);
	\draw[-,thick] (0.015,0) to (0.015,.2) to (.57,.4) to (.57,.6)
        to (.015,.8) to (.015,1);
	\draw[-,line width=1.2pt] (0.6,0) to (0.6,1);
        \node at (0.6,-.1) {$\scriptstyle b$};
        \node at (0,-.1) {$\scriptstyle a$};
        \node at (0.3,.82) {$\scriptstyle c$};
        \node at (0.3,.19) {$\scriptstyle d$};
\end{tikzpicture}
&=
\sum_{t=\max(0,c-b)}^{\min(c,d)}
\begin{tikzpicture}[anchorbase,scale=1]
	\draw[-,thin] (0,0) to (0,1);
	\draw[-,thick] (0.02,0) to (0.02,.2) to (.65,.5) to (1.08,.7) to (1.08,1);
	\draw[-,thick] (1.08,0) to (1.08,.3) to (.65,.5) to (.02,.8) to (.02,1);
	\draw[-,line width=.7pt] (.01,0) to (0.01,.21) to (.35,.37) to (.35,.63) to
        (0.01,.79) to (.01,1);
	\draw[-,thin] (1.1,0) to (1.1,1);
        \node at (0,-.1) {$\scriptstyle a$};
        \node at (1.13,-.1) {$\scriptstyle b$};
        \node at (-0.3,.5) {$\scriptstyle a-d$};
        \node at (0.26,.5) {$\scriptstyle t$};
        \node at (.2,.18) {$\scriptstyle d$};
        \node at (.2,.84) {$\scriptstyle c$};
\end{tikzpicture}.
\end{align*}
Then use (\ref{assrel})--(\ref{trivial}).
A similar argument establishes the first equality in
(\ref{jonsquare2}).
Then to prove the second equality in (\ref{jonsquare}), we 
use the first equality from (\ref{jonsquare2})
to expand the
right hand side, with the variable $t$ replaced by $u$, to see that it equals
$$
\sum_{u=\max(0,c-b)}^{\min(c,d)}
\sum_{t=u}^{\min(c,d)}
\binom{a-b+c-d}{u}
\binom{b-c+t}{t-u}
\begin{tikzpicture}[anchorbase,scale=1]
	\draw[-,thin] (0,0) to (0,1);
	\draw[-,thick] (0.02,0) to (0.02,.2) to (.88,.8) to (.88,1);
	\draw[-,thick] (0.88,0) to (0.88,.2) to (.02,.8) to (.02,1);
	\draw[-,thin] (0.9,0) to (0.9,1);
        \node at (0,-.1) {$\scriptstyle a$};
        \node at (.9,-.1) {$\scriptstyle b$};
        \node at (.33,.18) {$\scriptstyle c-t$};
        \node at (.34,.85) {$\scriptstyle d-t$};
\end{tikzpicture}.
$$
Now switch the summations and use the standard binomial coefficient
identity
$$
\sum_{u=0}^t \binom{a-b+c-d}{u} \binom{b-c+t}{t-u} = \binom{a-d+t}{t}.
$$
(Proof: Compute $x^t$-coefficients in
$(1+x)^{a-b+c-d}(1+x)^{b-c+t} = (1+x)^{a-d+t}$ in two different ways.)

\vspace{2mm}
\noindent
{\em Proof of (\ref{jonsquare2}).} This follows by reflecting (\ref{jonsquare})
in a vertical axis.

\vspace{2mm}
\noindent
{\em Proof of (\ref{thickcrossing}).}
The first equality is immediate from
the $r=0$
case of (\ref{mergesplit}).
Also the final equality follows from the middle one on reflecting
in a vertical axis. It remains to establish the middle one.
For this, we proceed by induction on $a+b$. The base case $a=b=1$
reduces to the first equality.
For the induction step, we have by the first equality and the induction hypothesis that
$$
\begin{tikzpicture}[anchorbase,scale=1]
	\draw[-,line width=1.2pt] (0,0) to (.6,1);
	\draw[-,line width=1.2pt] (0,1) to (.6,0);
        \node at (0,-.1) {$\scriptstyle a$};
        \node at (0.6,-0.1) {$\scriptstyle b$};
\end{tikzpicture}
=
\begin{tikzpicture}[anchorbase,scale=1]
	\draw[-,line width=1.2pt] (0,0) to (.285,.3) to (.285,.7) to (0,1);
	\draw[-,line width=1.2pt] (.6,0) to (.315,.3) to (.315,.7) to (.6,1);
        \node at (0,-.1) {$\scriptstyle a$};
        \node at (0.6,-.1) {$\scriptstyle b$};
\end{tikzpicture}
-\sum_{t=1}^{\min(a,b)}
\begin{tikzpicture}[anchorbase,scale=1]
	\draw[-,thin] (0,0) to (0,1);
	\draw[-,thick] (0.02,0) to (0.02,.2) to (.58,.8) to (.58,1);
	\draw[-,thick] (0.58,0) to (0.58,.2) to (.02,.8) to (.02,1);
	\draw[-,thin] (0.6,0) to (0.6,1);
        \node at (0,-.1) {$\scriptstyle a$};
        \node at (0.6,-.1) {$\scriptstyle b$};
        \node at (-0.1,.5) {$\scriptstyle t$};
        \node at (0.7,.5) {$\scriptstyle t$};
\end{tikzpicture}
=
\begin{tikzpicture}[anchorbase,scale=1]
	\draw[-,line width=1.2pt] (0,0) to (.285,.3) to (.285,.7) to (0,1);
	\draw[-,line width=1.2pt] (.6,0) to (.315,.3) to (.315,.7) to (.6,1);
        \node at (0,-.1) {$\scriptstyle a$};
        \node at (0.6,-.1) {$\scriptstyle b$};
\end{tikzpicture}
- \sum_{t=1}^{\min(a,b)}
\sum_{s=0}^{\min(a,b)-t}
(-1)^s
\begin{tikzpicture}[anchorbase,scale=1]
	\draw[-,thin] (0,0) to (0,1);
	\draw[-,thick] (0.02,0) to (0.02,.2) to (.7,.3) to (.7,.7) to (1.08,.8) to (1.08,1);
	\draw[-,thick] (1.08,0) to (1.08,.2) to (.7,.3) to (.7,.7) to (.02,.8) to (.02,1);
	\draw[-,thin] (.01,0) to (0.01,.22) to (.35,.27) to (.35,.73) to
        (0.01,.78) to (.01,1);
	\draw[-,thin] (1.1,0) to (1.1,1);
        \node at (0,-.1) {$\scriptstyle a$};
        \node at (1.13,-.1) {$\scriptstyle b$};
        \node at (-0.1,.5) {$\scriptstyle t$};
        \node at (1.25,.5) {$\scriptstyle t$};
        \node at (0.25,.5) {$\scriptstyle s$};
\end{tikzpicture}.
$$
We saw a similar expression to this before in (\ref{ohgood});
 we showed there just using the relations
(\ref{assrel})--(\ref{trivial}) and the square switch relations
established now by (\ref{jonsquare})--(\ref{jonsquare2}) that
$$
\sum_{t=0}^{\min(a,b)}
\sum_{s=0}^{\min(a,b)-t}
(-1)^s
\begin{tikzpicture}[anchorbase,scale=1]
	\draw[-,thin] (0,0) to (0,1);
	\draw[-,thick] (0.02,0) to (0.02,.2) to (.7,.3) to (.7,.7) to (1.08,.8) to (1.08,1);
	\draw[-,thick] (1.08,0) to (1.08,.2) to (.7,.3) to (.7,.7) to (.02,.8) to (.02,1);
	\draw[-,thin] (.01,0) to (0.01,.22) to (.35,.27) to (.35,.73) to
        (0.01,.78) to (.01,1);
	\draw[-,thin] (1.1,0) to (1.1,1);
        \node at (0,-.1) {$\scriptstyle a$};
        \node at (1.13,-.1) {$\scriptstyle b$};
        \node at (-0.1,.5) {$\scriptstyle t$};
        \node at (1.25,.5) {$\scriptstyle t$};
        \node at (0.25,.5) {$\scriptstyle s$};
\end{tikzpicture}=
\begin{tikzpicture}[anchorbase,scale=1]
	\draw[-,line width=1.2pt] (0,0) to (.285,.3) to (.285,.7) to (0,1);
	\draw[-,line width=1.2pt] (.6,0) to (.315,.3) to (.315,.7) to (.6,1);
        \node at (0,-.1) {$\scriptstyle a$};
        \node at (0.6,-.1) {$\scriptstyle b$};
\end{tikzpicture}.
$$
Thus, we have shown that
$$
\begin{tikzpicture}[anchorbase,scale=1]
	\draw[-,line width=1.2pt] (0,0) to (.6,1);
	\draw[-,line width=1.2pt] (0,1) to (.6,0);
        \node at (0,-.1) {$\scriptstyle a$};
        \node at (0.6,-0.1) {$\scriptstyle b$};
\end{tikzpicture}
=\sum_{s=0}^{\min(a,b)}
(-1)^s
\begin{tikzpicture}[anchorbase,scale=1]
	\draw[-,thick] (0.02,0) to (0.02,.2) to (.7,.3) to (.7,.7) to (1.08,.8) to (1.08,1);
	\draw[-,thick] (1.08,0) to (1.08,.2) to (.7,.3) to (.7,.7) to (.02,.8) to (.02,1);
	\draw[-,thin] (.01,0) to (0.01,.22) to (.35,.27) to (.35,.73) to
        (0.01,.78) to (.01,1);
        \node at (0,-.1) {$\scriptstyle a$};
        \node at (1.13,-.1) {$\scriptstyle b$};
        \node at (0.25,.5) {$\scriptstyle s$};
\end{tikzpicture}=
\sum_{t=0}^{\min(a,b)}
(-1)^t
\begin{tikzpicture}[anchorbase,scale=1]
	\draw[-,thick] (0,0) to (0,1);
	\draw[-,thick] (0.015,0) to (0.015,.2) to (.57,.4) to (.57,.6)
        to (.015,.8) to (.015,1);
	\draw[-,line width=1.2pt] (0.6,0) to (0.6,1);
        \node at (0.6,-.1) {$\scriptstyle b$};
        \node at (0,-.1) {$\scriptstyle a$};
        \node at (-0.1,.5) {$\scriptstyle t$};
\end{tikzpicture},
$$
as required.

\vspace{2mm}
\noindent
{\em Proof of (\ref{serre}).}
This is explained in the proof of \cite[Lemma 2.2.1]{CKM} (and
actually plays no role in this article).

\vspace{2mm}
\noindent
{\em Proof of (\ref{swallows}).}
By reflection, we just need to prove the first equality, and moreover
we may assume that $a \geq b$.
Replacing the crossing with (\ref{askthink}) then using (\ref{assrel})--(\ref{trivial}) as usual, we have that
$$
\begin{tikzpicture}[anchorbase,scale=.7]
	\draw[-,line width=2pt] (0.08,.3) to (0.08,.5);
\draw[-,thick] (-.2,-.8) to [out=45,in=-45] (0.1,.31);
\draw[-,thick] (.36,-.8) to [out=135,in=-135] (0.06,.31);
        \node at (-.3,-.95) {$\scriptstyle a$};
        \node at (.45,-.95) {$\scriptstyle b$};
\end{tikzpicture}
=
\sum_{s=0}^{b} (-1)^s
\begin{tikzpicture}[anchorbase,scale=.7]
	\draw[-,line width=1.8pt] (0.08,.1) to (0.08,.5);
\draw[-,thick] (.47,-.8) to [out=100,in=-45] (0.09,.115);
\draw[-,thin] (-.29,-.8) to [out=80,in=-135] (0.05,.115);
\draw[-,thick] (-.275,-.8) to (-.26,-.7) to (.425,-.6) to (.39,-.45) to (-.18,-.3) to [out=70,in=-135] (0.07,.115);
        \node at (-.3,-.95) {$\scriptstyle a$};
        \node at (.43,-.95) {$\scriptstyle b$};
        \node at (-.26,-.02) {$\scriptstyle b$};
        \node at (.41,-.06) {$\scriptstyle a$};
        \node at (-.4,-.45) {$\scriptstyle s$};
\end{tikzpicture}
=
\left(\sum_{s=0}^{b}(-1)^s \binom{a}{s} \binom{a+b-s}{a}
\right)\begin{tikzpicture}[anchorbase,scale=.7]
	\draw[-,line width=2pt] (0.08,.1) to (0.08,.5);
\draw[-,thick] (.46,-.8) to [out=100,in=-45] (0.1,.11);
\draw[-,thick] (-.3,-.8) to [out=80,in=-135] (0.06,.11);
        \node at (-.3,-.95) {$\scriptstyle a$};
        \node at (.43,-.95) {$\scriptstyle b$};
\end{tikzpicture}.
$$
It remains to observe that the coefficient here equals $1$. 
This follows by Lemma~\ref{A}, taking $m:=a$ and $n:=b$.

\vspace{2mm}
\noindent
{\em Proof of (\ref{sliders}).}
Note the four identities are all equivalent upon reflection, so we
just prove the first one:
$$
\begin{tikzpicture}[anchorbase,scale=0.7]
	\draw[-,thick] (0.4,0) to (-0.6,1);
	\draw[-,thick] (0.08,0) to (0.08,1);
	\draw[-,thick] (0.1,0) to (0.1,.6) to (.5,1);
        \node at (0.6,1.13) {$\scriptstyle c$};
        \node at (0.1,1.16) {$\scriptstyle b$};
        \node at (-0.65,1.13) {$\scriptstyle a$};
\end{tikzpicture}
=
\begin{tikzpicture}[anchorbase,scale=0.7]
	\draw[-,thick] (0.7,0) to (-0.3,1);
	\draw[-,thick] (0.08,0) to (0.08,1);
	\draw[-,thick] (0.1,0) to (0.1,.2) to (.9,1);
        \node at (0.9,1.13) {$\scriptstyle c$};
        \node at (0.1,1.16) {$\scriptstyle b$};
        \node at (-0.4,1.13) {$\scriptstyle a$};
\end{tikzpicture}.
$$
We proceed by induction on $a+b+c$. The base
case is when $a=0$, which is trivial. For the induction step, notice
that the diagram on the right hand side 
is a reduced chicken foot diagram. The idea is to expand the
left hand side in terms of reduced chicken foot diagrams too, then the
equality will be apparent. First we rewrite the crossing at the bottom
of this diagram using (\ref{thickcrossing}):
$$
\begin{tikzpicture}[baseline=3mm,scale=0.8]
	\draw[-,thick] (0.4,0) to (-0.6,1);
	\draw[-,thick] (0.08,0) to (0.08,1);
	\draw[-,thick] (0.1,0) to (0.1,.6) to (.5,1);
        \node at (0.6,1.13) {$\scriptstyle c$};
        \node at (0.1,1.16) {$\scriptstyle b$};
        \node at (-0.65,1.13) {$\scriptstyle a$};
\end{tikzpicture}
 =
\sum_{s=0}^{\min(a,b+c)} (-1)^s
\begin{tikzpicture}[anchorbase,scale=.8]
	\draw[-,thin] (0.01,0) to (0.01,1);
	\draw[-,thick] (0.02,0) to (0.02,0.2) to (.88,0.4) to (.88,0.6) to
        (0.02,.8) to (0.02,1);
	\draw[-,thick] (0.03,0) to (0.03,0.2) to (.91,0.4) to (.91,1);
	\draw[-,thick] (.92,0) to (.92,0.8) to (1.3,1);
        \node at (0,-.15) {$\scriptstyle b+c$};
        \node at (0,1.15) {$\scriptstyle a$};
        \node at (-.1,.5) {$\scriptstyle s$};
        \node at (.9,1.15) {$\scriptstyle b$};
        \node at (1.4,1.15) {$\scriptstyle c$};
        \node at (.9,-.15) {$\scriptstyle a$};
\end{tikzpicture}
 =
\sum_{s=0}^{\min(a,b+c)} (-1)^s
\begin{tikzpicture}[anchorbase,scale=.8]
	\draw[-,thin] (0.01,0) to (0.01,1);
	\draw[-,thick] (0.02,0) to (0.02,0.2) to (.695,0.3) to (.695,0.7) to
        (0.02,.8) to (0.02,1);
	\draw[-,thick] (0.03,0) to (0.03,0.2) to (.71,0.3) to (.71,1);
	\draw[-,thick] (.73,0) to (.73,0.5) to (1.4,1);
        \node at (0,-.15) {$\scriptstyle b+c$};
        \node at (0,1.15) {$\scriptstyle a$};
        \node at (-.1,.5) {$\scriptstyle s$};
        \node at (.7,1.15) {$\scriptstyle b$};
        \node at (1.4,1.15) {$\scriptstyle c$};
        \node at (.73,-.15) {$\scriptstyle a$};
\end{tikzpicture}.
$$
By (\ref{mergesplit}), we have that
$$
\begin{tikzpicture}[anchorbase,scale=.8]
	\draw[-,line width=1.2pt] (0,0) to (.275,.3) to (.275,.7) to (0,1);
	\draw[-,line width=1.2pt] (.6,0) to (.315,.3) to (.315,.7) to (.6,1);
        \node at (-0.2,1.13) {$\scriptstyle a+b-s$};
        \node at (0.63,1.13) {$\scriptstyle c$};
        \node at (-0.2,-.1) {$\scriptstyle b+c-s$};
        \node at (0.63,-.1) {$\scriptstyle a$};
\end{tikzpicture}
=
\sum_{t=\max(0,s-b)}^{\min(a,c)}
\begin{tikzpicture}[anchorbase,scale=.8]
	\draw[-,thin] (0,0) to (0,1);
	\draw[-,thick] (0.02,0) to (0.02,.2) to (1.08,.8) to (1.08,1);
	\draw[-,thick] (1.08,0) to (1.08,.2) to (.02,.8) to (.02,1);
	\draw[-,thin] (1.1,0) to (1.1,1);
        \node at (-.2,1.13) {$\scriptstyle a+b-s$};
        \node at (1.13,1.13) {$\scriptstyle c$};
        \node at (-.2,-.1) {$\scriptstyle b+c-s$};
        \node at (1.13,-.1) {$\scriptstyle a$};
        \node at (1.23,.5) {$\scriptstyle t$};
        \node at (-.6,.5) {$\scriptstyle b+t-s$};
\end{tikzpicture}.
$$
We substitute this into our formula to obtain
\begin{equation*}
\begin{tikzpicture}[baseline=3mm,scale=0.8]
	\draw[-,thick] (0.4,0) to (-0.6,1);
	\draw[-,thick] (0.08,0) to (0.08,1);
	\draw[-,thick] (0.1,0) to (0.1,.6) to (.5,1);
        \node at (0.6,1.13) {$\scriptstyle c$};
        \node at (0.1,1.16) {$\scriptstyle b$};
        \node at (-0.65,1.13) {$\scriptstyle a$};
\end{tikzpicture}
 =
\sum_{s=0}^{\min(a,b+c)} \sum_{t=\max(0,s-b)}^{\min(a,c)}(-1)^s
\begin{tikzpicture}[anchorbase,scale=.8]
	\draw[-,thin] (0,0) to (0,1);
	\draw[-,thick] (0.01,0) to (0.01,.05) to (.42,.7) to (0.01,.85)
        to (0.01,1);
	\draw[-,thick] (0.03,0) to (0.03,.05) to (.64,1);
	\draw[-,thin] (0.04,0) to (0.04,.05) to (1.4,.96) to (1.4,1);
	\draw[-,thick] (1.05,0)  to (1.05,.05) to (.345,.5) to (.465,.7)
        to (0.03,.85) to (0.03,1);
	\draw[-,thin] (1.07,0) to (1.07,.05) to (1.42,.96) to (1.42,1);
        \node at (0,-.15) {$\scriptstyle b+c$};
        \node at (0,1.15) {$\scriptstyle a$};
        \node at (-.1,.5) {$\scriptstyle s$};
        \node at (1.4,.5) {$\scriptstyle t$};
        \node at (.7,1.18) {$\scriptstyle b$};
        \node at (1.4,1.15) {$\scriptstyle c$};
        \node at (1.05,-.15) {$\scriptstyle a$};
\end{tikzpicture}.
\end{equation*}
By (\ref{mergesplit}) again, we have that
$$
\begin{tikzpicture}[anchorbase,scale=.8]
	\draw[-,line width=1.2pt] (0,0) to (.275,.3) to (.275,.7) to (0,1);
	\draw[-,line width=1.2pt] (.6,0) to (.315,.3) to (.315,.7) to (.6,1);
        \node at (-0.1,1.15) {$\scriptstyle a-s$};
        \node at (0.63,1.15) {$\scriptstyle b$};
        \node at (-0.3,-.1) {$\scriptstyle b+t-s$};
        \node at (0.7,-.1) {$\scriptstyle a-t$};
\end{tikzpicture}
=
\sum_{u=\max(s,t)}^{\min(a,b+t)}
\begin{tikzpicture}[anchorbase,scale=.8]
	\draw[-,thin] (0,0) to (0,1);
	\draw[-,thick] (0.02,0) to (0.02,.2) to (1.08,.8) to (1.08,1);
	\draw[-,thick] (1.08,0) to (1.08,.2) to (.02,.8) to (.02,1);
	\draw[-,thin] (1.1,0) to (1.1,1);
        \node at (-.2,1.15) {$\scriptstyle a-s$};
        \node at (1.13,1.15) {$\scriptstyle b$};
        \node at (-.2,-.1) {$\scriptstyle b+t-s$};
        \node at (1.13,-.1) {$\scriptstyle a-t$};
        \node at (1.43,.5) {$\scriptstyle u-t$};
        \node at (-.4,.5) {$\scriptstyle u-s$};
\end{tikzpicture}.
$$
Using this,  (\ref{assrel})--(\ref{trivial}), 
and the induction hypothesis to pull a two-fold
split past the string of thickness $c-t$,
we simplify further to get
\begin{align*}
\begin{tikzpicture}[baseline=3mm,scale=0.8]
	\draw[-,thick] (0.4,0) to (-0.6,1);
	\draw[-,thick] (0.08,0) to (0.08,1);
	\draw[-,thick] (0.1,0) to (0.1,.6) to (.5,1);
        \node at (0.6,1.13) {$\scriptstyle c$};
        \node at (0.1,1.16) {$\scriptstyle b$};
        \node at (-0.65,1.13) {$\scriptstyle a$};
\end{tikzpicture}
& =
\sum_{s=0}^{\min(a,b+c)} \sum_{t=\max(0,s-b)}^{\min(a,c)}
\sum_{u=\max(s,t)}^{\min(a,b+t)}
(-1)^s \binom{u}{s}
\begin{tikzpicture}[anchorbase,scale=.8]
	\draw[-,thin] (0.34,0) to (.34,.05) to (0,.95) to (0,1);
	\draw[-,thin] (0.35,0) to (.35,.05) to (.695,.95) to (.695,1);
	\draw[-,thin] (1.03,0) to (1.03,.05) to (.705,.95) to (.705,1);
	\draw[-,thin] (0.36,0) to (.36,.05) to (1.39,.95) to (1.39,1);
	\draw[-,thin] (0.01,1) to (.01,.95) to (1.035,.05) to (1.035,0);
	\draw[-,thin] (1.04,0) to (1.04, .05) to (1.4,.95) to (1.4,1);
        \node at (0.35,-.15) {$\scriptstyle b+c$};
        \node at (0,1.15) {$\scriptstyle a$};
        \node at (-.07,.5) {$\scriptstyle u$};
        \node at (1.4,.5) {$\scriptstyle t$};
        \node at (.7,1.18) {$\scriptstyle b$};
        \node at (1.4,1.15) {$\scriptstyle c$};
        \node at (1.05,-.15) {$\scriptstyle a$};
\end{tikzpicture}\\
&=
\sum_{t=0}^{\min(a,c)}
\sum_{u=t}^{\min(a,b+t)}
\sum_{s=0}^u 
(-1)^s \binom{u}{s}
\begin{tikzpicture}[anchorbase,scale=.8]
	\draw[-,thin] (0.34,0) to (.34,.05) to (0,.95) to (0,1);
	\draw[-,thin] (0.35,0) to (.35,.05) to (.695,.95) to (.695,1);
	\draw[-,thin] (1.03,0) to (1.03,.05) to (.705,.95) to (.705,1);
	\draw[-,thin] (0.36,0) to (.36,.05) to (1.39,.95) to (1.39,1);
	\draw[-,thin] (0.01,1) to (.01,.95) to (1.035,.05) to (1.035,0);
	\draw[-,thin] (1.04,0) to (1.04, .05) to (1.4,.95) to (1.4,1);
        \node at (0.35,-.15) {$\scriptstyle b+c$};
        \node at (0,1.15) {$\scriptstyle a$};
        \node at (-.07,.5) {$\scriptstyle u$};
        \node at (1.4,.5) {$\scriptstyle t$};
        \node at (.7,1.18) {$\scriptstyle b$};
        \node at (1.4,1.15) {$\scriptstyle c$};
        \node at (1.05,-.15) {$\scriptstyle a$};
\end{tikzpicture}.
\end{align*}
Since $\sum_{s=0}^{u}
(-1)^s \binom{u}{s} = \delta_{u,0}$,
 which is zero unless $u=0$, when it is $1$, 
the only non-zero term arises when $u=t=0$, and we get 
exactly the right hand side we were after.

\vspace{2mm}
\noindent
{\em Proof of (\ref{symmetric}).}
We proceed by induction on $a+b$, the
case $a+b=1$ being trivial.
For the induction step, we may assume without loss of generality that
$a \leq b$. 
We claim for $0 \leq s < a$ that 
$$
\begin{tikzpicture}[anchorbase,scale=0.8]
  \draw[-,thin] (-.289,-.8) to (-.289,-.29);
  \draw[-,thin] (-.274,-.8) to (-.274,-.65) to (.279,-.29);
  \draw[-,thin] (.274,-.8) to (.274,-.65) to (-.279,-.29);
  \draw[-,thin] (.289,-.8) to (.289,-.29);
	\draw[-,thick] (0.28,-.3) to [out=90,in=-50] (0,.2) to [out=130,in=-90] (-0.28,.6);
	\draw[-,thick] (-0.28,-.3) to [out=90,in=-130] (0,.2) to [out=50,in=-90] (0.28,.6);
        \node at (0.3,-.95) {$\scriptstyle b$};
        \node at (-0.3,-.95) {$\scriptstyle a$};
        \node at (-.65,-.45) {$\scriptstyle a-s$};
       \node at (0.3,.77) {$\scriptstyle b$};
        \node at (-0.3,.75) {$\scriptstyle a$};
\end{tikzpicture}
=
\begin{tikzpicture}[anchorbase,scale=.8]
	\draw[-,thin] (0,0) to (0,1.4);
	\draw[-,thick] (0.02,0) to (0.02,.3) to (.58,1.1) to (.58,1.4);
	\draw[-,thick] (.58,0) to (.58,.3) to (.02,1.1) to (.02,1.4);
	\draw[-,thin] (.6,0) to (.6,1.4);
        \node at (0,-.1) {$\scriptstyle a$};
        \node at (.58,-.1) {$\scriptstyle b$};
        \node at (-0.16,.7) {$\scriptstyle s$};
       \node at (0.6,1.57) {$\scriptstyle b$};
        \node at (0,1.55) {$\scriptstyle a$};
\end{tikzpicture}.
$$
To see this, one uses (\ref{swallows})--(\ref{sliders}) plus the induction hypothesis to
pull the two-fold merges past the crossing.
Using the claim, (\ref{mergesplit}) and
(\ref{thickcrossing})--(\ref{swallows}),
we deduce that
$$
\mathord{
\begin{tikzpicture}[anchorbase,scale=0.8]
	\draw[-,thick] (0.28,-0.1) to[out=90,in=-90] (-0.28,.6);
	\draw[-,thick] (-0.28,-0.1) to[out=90,in=-90] (0.28,.6);
	\draw[-,thick] (0.28,-.8) to[out=90,in=-90] (-0.28,-0.1);
	\draw[-,thick] (-0.28,-.8) to[out=90,in=-90] (0.28,-0.1);
        \node at (0.3,-.95) {$\scriptstyle b$};
        \node at (-0.3,-.95) {$\scriptstyle a$};
       \node at (0.3,.77) {$\scriptstyle b$};
        \node at (-0.3,.75) {$\scriptstyle a$};
\end{tikzpicture}
}=
\mathord{
\begin{tikzpicture}[anchorbase,scale=0.8]
	\draw[-,thick] (0.28,0) to[out=90,in=-90] (-0.28,.6);
	\draw[-,thick] (-0.28,0) to[out=90,in=-90] (0.28,.6);
	\draw[-,thick] (-.3,-.8) to (-.01,-.5) to (-0.01,-.3) to[out=135,in=-90] (-0.28,0);
	\draw[-,thick] (0.3,-.8) to (.01,-.5) to (0.01,-.3) to[out=45,in=-90] (0.28,0);
        \node at (0.3,-.95) {$\scriptstyle b$};
        \node at (-0.3,-.95) {$\scriptstyle a$};
       \node at (0.3,.77) {$\scriptstyle b$};
        \node at (-0.3,.75) {$\scriptstyle a$};
\end{tikzpicture}
}
\!\!-\sum_{t=1}^{a}
\begin{tikzpicture}[anchorbase,scale=0.8]
  \draw[-,thin] (-.289,-.8) to (-.289,-.29);
  \draw[-,thin] (-.274,-.8) to (-.274,-.65) to (.279,-.29);
  \draw[-,thin] (.274,-.8) to (.274,-.65) to (-.279,-.29);
  \draw[-,thin] (.289,-.8) to (.289,-.29);
	\draw[-,thick] (0.28,-.3) to [out=90,in=-50] (0,.2) to [out=130,in=-90] (-0.28,.6);
	\draw[-,thick] (-0.28,-.3) to [out=90,in=-130] (0,.2) to [out=50,in=-90] (0.28,.6);
        \node at (0.3,-.95) {$\scriptstyle b$};
        \node at (-0.3,-.95) {$\scriptstyle a$};
       \node at (0.3,.77) {$\scriptstyle b$};
        \node at (-0.3,.75) {$\scriptstyle a$};
        \node at (-0.43,-.45) {$\scriptstyle t$};
\end{tikzpicture}
=
\mathord{
\begin{tikzpicture}[anchorbase,scale=0.8]
	\draw[-,thick] (-.3,-.8) to (-.01,-.3) to (-0.01,-.3) to (0.01,.1) to (-.3,.6);
	\draw[-,thick] (0.3,-.8) to (.01,-.3) to (0.01,-.3) to
        (0.01,.1) to (.3,.6);
        \node at (0.3,-.95) {$\scriptstyle b$};
        \node at (-0.3,-.95) {$\scriptstyle a$};
       \node at (0.3,.77) {$\scriptstyle b$};
        \node at (-0.3,.75) {$\scriptstyle a$};
\end{tikzpicture}
}
\!\!-\sum_{s=0}^{a-1}
\begin{tikzpicture}[anchorbase,scale=0.8]
  \draw[-,thin] (-.289,-.8) to (-.289,-.29);
  \draw[-,thin] (-.274,-.8) to (-.274,-.65) to (.279,-.29);
  \draw[-,thin] (.274,-.8) to (.274,-.65) to (-.279,-.29);
  \draw[-,thin] (.289,-.8) to (.289,-.29);
	\draw[-,thick] (0.28,-.3) to [out=90,in=-50] (0,.2) to [out=130,in=-90] (-0.28,.6);
	\draw[-,thick] (-0.28,-.3) to [out=90,in=-130] (0,.2) to [out=50,in=-90] (0.28,.6);
        \node at (0.3,-.95) {$\scriptstyle b$};
        \node at (-0.3,-.95) {$\scriptstyle a$};
       \node at (-.65,-.45) {$\scriptstyle a-s$};
       \node at (0.3,.77) {$\scriptstyle b$};
        \node at (-0.3,.75) {$\scriptstyle a$};
\end{tikzpicture}=
\sum_{s=0}^{a}
\begin{tikzpicture}[anchorbase,scale=.8]
	\draw[-,thin] (0,0) to (0,1.4);
	\draw[-,thick] (0.02,0) to (0.02,.3) to (.58,1.1) to (.58,1.4);
	\draw[-,thick] (.58,0) to (.58,.3) to (.02,1.1) to (.02,1.4);
	\draw[-,thin] (.6,0) to (.6,1.4);
        \node at (0,-.1) {$\scriptstyle a$};
        \node at (.58,-.1) {$\scriptstyle b$};
        \node at (-0.16,.7) {$\scriptstyle s$};
       \node at (0.6,1.57) {$\scriptstyle b$};
        \node at (0,1.55) {$\scriptstyle a$};
\end{tikzpicture}
-\!\!\sum_{s=0}^{a-1}
\begin{tikzpicture}[anchorbase,scale=0.8]
  \draw[-,thin] (-.289,-.8) to (-.289,-.29);
  \draw[-,thin] (-.274,-.8) to (-.274,-.65) to (.279,-.29);
  \draw[-,thin] (.274,-.8) to (.274,-.65) to (-.279,-.29);
  \draw[-,thin] (.289,-.8) to (.289,-.29);
	\draw[-,thick] (0.28,-.3) to [out=90,in=-50] (0,.2) to [out=130,in=-90] (-0.28,.6);
	\draw[-,thick] (-0.28,-.3) to [out=90,in=-130] (0,.2) to [out=50,in=-90] (0.28,.6);
        \node at (0.3,-.95) {$\scriptstyle b$};
        \node at (-0.3,-.95) {$\scriptstyle a$};
        \node at (-.65,-.45) {$\scriptstyle a-s$};
       \node at (0.3,.77) {$\scriptstyle b$};
        \node at (-0.3,.75) {$\scriptstyle a$};
\end{tikzpicture}=
\begin{tikzpicture}[anchorbase,scale=.8]
	\draw[-,thick] (0,0) to (0,1.4);
	\draw[-,thick] (.6,0) to (.6,1.4);
        \node at (0,-.1) {$\scriptstyle a$};
        \node at (.58,-.1) {$\scriptstyle b$};
       \node at (0.6,1.57) {$\scriptstyle b$};
        \node at (0,1.55) {$\scriptstyle a$};
\end{tikzpicture}.
$$

\vspace{2mm}
\noindent
{\em Proof of (\ref{braid}).}
Replace the crossing of the
strings of thickness $a,b$ on both sides 
with (\ref{askthink}). Then use (\ref{sliders})--(\ref{symmetric})
to pull the string of thickness 
$c$ past this expansion of the crossing.

\end{document}